\pdfoutput=1

\documentclass[twoside,11pt]{article}

\usepackage{times} 
\usepackage{helvet}
\usepackage{courier}
\usepackage[hyphens]{url}
\usepackage{graphicx} 
\usepackage{bmpsize}
\usepackage{multicol}
\usepackage{hyperref}
\usepackage{tikz}
\usetikzlibrary{arrows}
\usetikzlibrary{shadows} % Shadows for nodes
\usepackage[numbers]{natbib}
\usepackage{caption}
\usepackage[utf8]{inputenc} % allow utf-8 input
\usepackage[T1]{fontenc}    % use 8-bit T1 fonts
\usepackage{url}            % simple URL typesetting
\usepackage{booktabs}       % professional-quality tables
\usepackage{amsfonts}       % blackboard math symbols
\usepackage{nicefrac}       % compact symbols for 1/2, etc.
\usepackage{microtype}      % microtypography
\RequirePackage{amsthm,amsmath, mathabx}
\usepackage{bbm, amssymb, kbordermatrix, graphicx, subcaption, changepage}
\usepackage[ruled,noend,nofillcomment]{algorithm2e}
\usepackage{thmtools}
\usepackage{thm-restate}
\usepackage{wrapfig}
\usepackage[left=1in,right=1in,top=1in,bottom=1in]{geometry}
\usepackage[title]{appendix}

\allowdisplaybreaks

\DeclareMathOperator{\diag}{diag}
\DeclareMathOperator{\err}{err}

\DeclareMathOperator*{\argmin}{argmin}
\DeclareMathOperator*{\elemargmin}{elem-argmin}

\DeclareMathOperator{\ri}{ri}

\DeclareMathOperator{\cl}{cl}
\DeclareMathOperator{\aff}{aff}

\DeclareMathOperator{\minplus}{min_{>0}}
\DeclareMathOperator{\prob}{Prob}
\DeclareMathOperator{\graph}{graph}
\DeclareMathOperator{\setint}{int}

\newcommand{\calA}{\mathcal{A}}
\newcommand{\calE}{\mathcal{E}}
\newcommand{\calK}{\mathcal{K}}
\newcommand{\calM}{\mathcal{M}}
\newcommand{\calO}{\mathcal{O}}
\newcommand{\calS}{\mathcal{S}}
\newcommand{\calP}{\mathcal{P}}
\newcommand{\calU}{\mathcal{U}}
\newcommand{\calV}{\mathcal{V}}

\newcommand{\calZ}{\mathcal{Z}}

\newcommand{\bfu}{\mathbf{u}}

\newcommand{\bfx}{\mathbf{x}}
\newcommand{\bfy}{\mathbf{y}}

\newcommand{\bbR}{\mathbb{R}}

\newcommand{\bbP}{\mathbb{P}}
\newcommand{\bbQ}{\mathbb{Q}}
\newcommand{\bbN}{\mathbb{N}}

\newcommand{\X}{\mathcal{X}}
\newcommand{\Y}{\mathcal{Y}}

\newcommand{\tc}{\Pi_{\mbox{\tiny TC}}(P, Q)}

\newcommand{\tcrow}{\Pi(P(x,\cdot), Q(y,\cdot))}

\newcommand{\tceta}{\Pi_{\mbox{\tiny TC}}^\eta(P,Q)}
\newcommand{\tcetarow}{\Pi_\eta(P(x,\cdot), Q(y,\cdot))}
\newcommand{\tcetarv}{\Pi_{\mbox{\tiny TC}}^\eta(\bbP, \bbQ)}

\newcommand{\tcrv}{\Pi_{\mbox{\tiny TC}}(\bbP, \bbQ)}
\newcommand{\pitc}{\Pi_{\mbox{\tiny TC}}}

\newtheorem{ex}{Example} 
\newtheorem{thm}{Theorem}
\newtheorem{prop}[thm]{Proposition} 
\newtheorem{rem}[thm]{Remark}
\newtheorem{cor}[thm]{Corollary}
\newtheorem{defn}[thm]{Definition}

\newtheorem{lem}[thm]{Lemma}

\newcounter{tmp}

\begin{document}

%\editor{}

\title{Optimal Transport for Stationary Markov Chains \\ via Policy Iteration}

%\author{\name Kevin O'Connor \email koconn@live.unc.edu \\
%       \addr Department of Statistics and Operations Research\\
%       University of North Carolina, Chapel Hill\\
%       Chapel Hill, NC 27599, USA
%       \AND
%       \name Kevin McGoff \email kmcgoff1@uncc.edu \\
%       \addr Department of Mathematics and Statistics\\
%       University of North Carolina, Charlotte\\
%       Charlotte, NC 28223, USA
%       \AND
%       \name Andrew B Nobel \email nobel@email.unc.edu \\
%       \addr Department of Statistics and Operations Research\\
%       University of North Carolina, Chapel Hill\\
%       Chapel Hill, NC 27599, USA
%       }

\author{Kevin O'Connor, Kevin McGoff, and Andrew B Nobel}

\maketitle

\begin{abstract}
We study the optimal transport problem for pairs of stationary finite-state Markov chains, with an emphasis on the 
computation of optimal transition couplings.  Transition couplings are a constrained family of transport plans 
that capture the dynamics of Markov chains.  
Solutions of the optimal transition coupling (OTC) problem correspond to alignments of the 
two chains that minimize long-term average cost.   
We establish a connection between the OTC problem and Markov decision processes, 
and show that solutions of the OTC problem can be obtained via an adaptation of policy iteration.
For settings with large state spaces, we develop a fast approximate algorithm based on an entropy-regularized 
version of the OTC problem, and provide bounds on its per-iteration complexity.
We establish a stability result for both the regularized and unregularized algorithms, 
from which a statistical consistency result follows as a corollary.
We validate our theoretical results empirically through a simulation study, demonstrating that the 
approximate algorithm exhibits faster overall runtime with low error.
Finally, we extend the setting and application of our methods to hidden Markov models, and 
illustrate the potential use of the proposed algorithms in practice with an application to computer-generated music.
\end{abstract}

%\begin{keywords}
%  Optimal Transport, Markov Chains, Markov Decision Processes, Stationary Processes, Entropic Regularization
%\end{keywords}

\section{Introduction}\label{sec:intro}

The application and computation of optimal transport (OT) plans has recently received a great deal of attention within the machine learning community.
Applications of optimal transport in machine learning include generative modeling 
\citep{arjovsky2017wasserstein, deshpande2018generative, genevay2017learning,  kolouri2018sliced,salimans2018improving} 
and supervised learning \citep{frogner2015learning, janati2019wasserstein, Luise2018DifferentialPO}.
In this paper, we study the optimal transport (OT) problem in the case where the objects of interest are stationary Markov chains or processes possessing hidden Markov structure.
The problem of interest to us is distinct from traditional applications of coupling to Markov chains, e.g., 
to establish convergence to a stationary distribution.
Our interest is in the computation of optimal transport plans for Markov chains that explicitly account for both stationarity 
and Markovian structure.
In particular, we develop algorithms for computing solutions to a Markov-constrained form of the OT problem. 
The algorithms leverage recent advances in computational OT as well as techniques from Markov decision processes.

The principled extension of computational OT techniques to classes of distributions that possess additional structure, such as martingales or dependent processes, is an important direction of research.  Indeed, some variations of constrained OT have been considered in recent work \citep{beiglbock2013model,zaev2015monge,forrow2019statistical,moulos2020bicausal,backhoff2020estimating}, and  %capturing stochastic structure, for example,
%processes with serial dependence, is an important problem.
%Indeed, 
several recent applications of OT have focused on dependent observations \citep{schiebinger2019optimal,xu2018distilled}. 
%, including modeling the growth of cell populations over time 
%\citep{schiebinger2019optimal} and embedding natural language \citep{xu2018distilled}, 
%involve dependent processes.
Extensions of OT to dependent processes open the door to new applications in climate science, finance, epidemiology and other fields, where it is common for observations to possess temporal or spatial structure.
The OT problem that we consider is tailored to the alignment and comparison of Markov chains and hidden
Markov models (HMMs). % or observations therefrom.
As an illustration, we describe in Section \ref{sec:experiments} an application of the proposed techniques to 
the analysis of computer-generated music.

%The results of this paper are a rigorous advance of computational OT techniques to dependent processes.
%Our primary contributions are as follows:
The primary contributions of this paper are as follows:
\begin{itemize}
\item We formulate a constrained version of the OT problem for stationary Markov chains, referred to as the optimal transition coupling (OTC) problem.
The OTC problem aims to align the two chains of interest so as to minimize long-term average cost while preserving Markovity and stationarity.
%In this new problem, one considers a constrained set of couplings known as transition couplings, defined formally in Definitions \ref{def:tc_mats} and \ref{def:transitioncouplings}, that are stationary, Markovian, and have transition distributions satisfying a certain coupling condition.
%We argue that this is a natural adaptation of the OT problem in the context of stationary Markov chains.

\item We detail an extension of the OTC problem to HMMs.
%The extension solves one OT problem using a cost function derived from another OT problem.
In particular, we describe how one may couple a pair of HMMs via a coupling of their hidden chains using a cost that is derived from the OT cost between their emission distributions.

\item 
We establish a useful connection between the OTC problem and Markov decision processes (MDPs) that 
provides a means of computing optimal solutions in an efficient manner.
%In particular, we prove that the OTC problem corresponds to an average-cost, multichain MDP.
%This connection enables one to combine existing algorithms for OT with algorithms for MDPs to compute solutions of the OTC problem.
% apply existing algorithms for solving MDPs to the OTC problem.
%Owing to the special structure of the MDP associated with the OTC problem, we find that the celebrated policy iteration algorithm \citep{howard1960dynamic} admits a simple implementation and interpretation that is well-suited to the OTC problem.
%Owing to the special structure of the MDP associated with the OTC problem, 
Leveraging this connection, we arrive at an algorithm combining policy iteration \citep{howard1960dynamic} with OT solvers that we refer to as \texttt{ExactOTC} (Algorithm \ref{alg:pia}).
We state in Theorem \ref{thm:convergence_of_pia} that if the two Markov chains of interest are irreducible, then \texttt{ExactOTC} converges to a solution of the OTC problem in a finite number of iterations.
%Our proof of this result brings together results from OT, linear programming, Markov chain theory, and classical convergence results for Markov decision processes.

\item 
We introduce an entropically-constrained OTC problem and an associated regularized algorithm, referred to as \texttt{EntropicOTC} (Algorithm \ref{alg:fastentropic_pia}), that exhibits improved computational efficiency in theory and in practice.
%The proposed algorithm, referred to as \texttt{EntropicOTC} (Algorithm \ref{alg:fastentropic_pia}), uses an approximation in each of the two steps of every iteration: approximately evaluating the expected cost of a given transition coupling, and using entropic OT to find a better transition coupling.
In Theorems \ref{thm:policy_evaluation_complexity} and \ref{thm:policy_improvement_complexity}, we establish upper bounds on the computational complexity of this algorithm, demonstrating that the runtime of each iteration is nearly-linear in the dimension of the couplings under study.
This dependence is comparable to the state-of-the-art for computational OT.

\item We prove a stability result for the OTC problem, stated formally in Theorem \ref{thm:consistency}.
Consistency of the plug-in estimate of the optimal transition coupling and its expected cost follows as a corollary (see Corollary \ref{cor:consistency}).
\end{itemize}

The rest of the paper is organized as follows: 
We begin by providing some background on optimal transport and define the OTC problem in Section \ref{sec:background_on_otc}.
In Section \ref{sec:hmm}, we detail our extension of the OTC problem to HMMs.
In Section \ref{sec:computing}, we establish the connection between the OTC problem and MDPs and state our result regarding \texttt{ExactOTC} for obtaining optimal transition couplings.
A faster, regularized algorithm \texttt{EntropicOTC} for computing optimal transition couplings is described 
in Section \ref{sec:fast_approx_pia}.
In Section \ref{sec:consistency} we present our result regarding the stability of the OTC problem and the statistical consistency of optimal transition couplings computed from data.
In Section \ref{sec:experiments} we describe a simulation study and an application of our algorithms to computer-generated music.
We close with a discussion of our results in Section \ref{sec:conclusion}.
Proofs for all stated results may be found in Section \ref{sec:proofs}.
%Finally, we include some supplementary results and information in the Appendix found at the end of the paper.
Finally, an appendix containing some supplementary results and information may be found in the accompanying supplemental material.

\paragraph{Notation.} Let $\bbR_+$ be the non-negative reals and $\Delta_n = \{u \in \mathbb{R}^n_+ | \sum_{i=1}^n u_i = 1\}$ denote the probability simplex in $\bbR^n$.
Given a metric space $\calU$, let $\calM(\calU)$ denote the set of Borel probability measures on $\calU$.
For a vector $u \in \bbR^n$, let $\|u\|_\infty = \max_i |u_i|$ and $\|u\|_1 = \sum_i |u_i|$.
Occasionally we will treat matrices in $\bbR^{n\times n}$ as vectors in $\bbR^{n^2}$.

\section{The Optimal Transition Coupling Problem}
\label{sec:background_on_otc}

The optimal transport problem is defined in terms of couplings and a cost function.
Let $\calU$ and $\calV$ be metric spaces.
Given probability measures $\mu \in \calM(\calU)$ and $\nu \in \calM(\calV)$,
a \emph{coupling} of $\mu$ and $\nu$ is a 
probability measure $\pi \in \calM(\calU \times \calV)$ such that $\pi(A\times \calV) = \mu(A)$ 
and $\pi(\calU\times B) = \nu(B)$ for every measurable $A \subset \calU$ and $B \subset \calV$.
Let $\Pi(\mu, \nu)$ be the set of couplings of $\mu$ and $\nu$.  Let 
$c: \calU \times \calV \rightarrow \bbR$ be a cost function.  We interpret $c(u,v)$ 
as the cost of transporting one unit of a quantity from $u \in \calU$ to $v \in \calV$, or vice versa.
The optimal transport problem associated with $\mu$, $\nu$, and $c$ is the program
\begin{align}\label{eq:ot_problem}
\begin{split}
\mbox{minimize} \quad &\int c\, d\pi \\
\mbox{subject to} \quad &\pi \in \Pi(\mu, \nu).
\end{split}
\end{align}

As formulated, the problem \eqref{eq:ot_problem} makes no particular assumptions about the structure of
the measures $\mu$ and $\nu$.
In most existing applications, $\mu$ and $\nu$ represent the distribution of static quantities 
such as 3-dimensional point clouds, images of handwritten digits, social networks, or measurements of gene expression.
However, in other application areas, $\mu$ and $\nu$ may represent dynamic quantities that vary with
time or some other index.
For example, $\mu$ and $\nu$ might be distributions of words in a block of text, 
the heart rate or blood pressure of a patient over a period of observation, 
or the daily high temperatures at two different locations over some period of time.
In such cases, one may wish to constrain the types of couplings under consideration to ensure that they reflect the structure of the underlying distributions.
%In these cases, additional care is needed in order to study \eqref{eq:ot_problem}.

As a natural first step toward computational OT for dependent processes, we consider the case where 
$\mu$ and $\nu$ represent stationary Markov chains $X = (X_0, X_1, ...)$ and $Y = (Y_0, Y_1, ...)$ 
with values in finite sets $\X$ and $\Y$, respectively.  Markov chains are a natural choice:
%Unlike general stationary processes, which may exhibit strong dependence over long time scales, 
their simple dependence structure is conducive to computation, and they can be studied in terms 
of transition matrices.
%little generality is lost in making these assumptions, as 
Without loss of generality, assume that $\X$ and $\Y$ both contain $d$ points.
{Let $P, Q \in [0,1]^{d \times d}$ be the transition matrices, and let $p, q \in \Delta_d$ be the corresponding stationary 
distributions, of the chains $X$ and $Y$, respectively.}
For a brief overview of the necessary background on Markov chains, we refer the reader to Section \ref{sec:overview_of_proofs}.
For a more in-depth review of Markov chain theory, we refer the reader to \cite{levin2017markov}.
The extension of the OTC problem to hidden Markov models, detailed in Section \ref{sec:hmm}, 
enables us to apply our approach to non-Markovian processes with long-range dependence and Polish alphabets.

\begin{rem}
The optimal transport problem traces its roots back to the physical transportation of goods.
In particular, the optimal coupling offers a means of stochastically matching a supply of some goods to their demand so as to minimize the expected cost of transporting the goods.
In his book on the topic, Villani \citep{villani2008optimal} offers an example of transporting loaves of bread between bakeries and caf\'{e}s to build intuition for the optimal transport problem:

\begin{quote}
Consider a large number of bakeries, producing loaves, that should be transported each morning to caf\'{e}s where consumers will eat them.
The amount of bread that can be produced at each bakery, and the amount that will be consumed at each caf\'{e} are known in advance, and can be modeled as probability measures ... on a certain space ... (equipped with the natural metric such that the distance between two points is the shortest path joining them).
The problem is to find in practice where each unit of bread should go, in such a way as to minimize the total transport cost.
\end{quote}

In our setting, the collections of bakeries and caf\'{e}s correspond to the finite sets $\X$ and $\Y$.
However, unlike the static problem described by Villani, we consider a dynamic problem in which the number of loaves produced and consumed at the bakeries and caf\'{e}s evolves over time.
%Indeed, we suppose that these amounts are random and evolve day-to-day as Markov chains $X$ and $Y$ with fixed stationary distributions.
Indeed, we suppose that the amounts produced and consumed are determined by the distributions of stationary Markov chains $X$ and $Y$.
As we now have dependence over time to consider, the new problem is to synchronize the supply with the demand so as to minimize the total cost of transportation over the long term while still ensuring that the bakery and cafe owners are satisfied.
To make things easier for the delivery driver, one might agree to consider only transport plans that do not change over time (stationary) and under which the deliveries tomorrow only depend on the deliveries today (Markov).
%Finally, one may require that the number of loaves to transport from any given bakery tomorrow depends only on the amount transported from bakeries today and is independent of the amounts transported to caf\'{e}s today (and similarly for caf\'{e}s).
%This last condition corresponds to transition couplings of the supply chain $X$ and demand chain $Y$.
\end{rem}

{In principle, one may apply the standard optimal transport problem in the Markov setting by taking 
$\calU = \X$, $\calV = \Y$ and identifying an optimal coupling of the stationary distributions $p$ and $q$.
However, this marginal approach does not capture the dependence structure of the chains
$X$ and $Y$.
Consider, for example, the case when $\X = \Y = \{0, 1\}$ with single-letter cost $c(x,y) = \delta(x \neq y)$, and
\begin{equation*}
P= \kbordermatrix{&0&1 \\ 0 & \nicefrac{1}{2} & \nicefrac{1}{2} \\ 1 & \nicefrac{1}{2} & \nicefrac{1}{2}}
\quad\mbox{and}\quad Q = \kbordermatrix{&0&1 \\ 0 & 0& 1 \\ 1 & 1 & 0}.
\end{equation*}
Note that the process $X$ corresponding to $P$ is iid, while the process $Y$ corresponding to $Q$ is deterministic (after conditioning on the initial symbol $Y_0$).
Nevertheless, under a marginal analysis, the optimal transport distance between $X$ and $Y$ is zero since 
their stationary distributions $p$ and $q$ each coincide with the $(\nicefrac{1}{2},\nicefrac{1}{2})$ measure.
{\em In general, optimal coupling of stationary distributions yields a joint distribution on the product $\X \times \Y$,
but  it does not provide a means of generating a joint process having $X$ and $Y$ as marginals}.
We seek a variation of \eqref{eq:ot_problem} that captures and preserves the 
stochastic structure, namely stationarity and Markovity, of the processes $X$ and $Y$.}

As an alternative to a marginal analysis, one may consider instead the full measures $\bbP \in \calM(\X^\bbN)$ and
$\bbQ \in \calM(\Y^\bbN)$ of the processes $X$ and $Y$.
Formally, $\bbP$ is the unique probability measure on $\X^\bbN$ such that for any cylinder set 
$[a_i^j] := \{(x_0, x_1, ...) \in \X^\bbN: x_k = a_k, i \leq k \leq j\}$, 
\begin{equation*}
\bbP([a_i^j]) := p(a_i) \prod\limits_{k=i}^{j-1} P(a_k, a_{k+1}) .
\end{equation*}
The measure $\bbQ$ is defined similarly in terms of $q$ and $Q$.
By definition, the measures $\bbP$ and $\bbQ$ are stationary, and Markovian.  
However, a coupling of $\bbP$ and $\bbQ$
on the joint sequence space $\X^\bbN \times \Y^\bbN$ need not be stationary or Markovian.
To illustrate, let $X'$ and $Y'$ be iid Bernoulli$(\nicefrac{1}{2},\nicefrac{1}{2})$ processes, independent of each other, defined on the same probability space.
For $i \geq 0$ let $\tilde{X}_i = X_i'$, and let $\tilde{Y}_i = X_i'$ if $i$ is a power of $2$ and
$\tilde{Y}_i = Y_i'$ otherwise.  It is easy to see that the joint process 
$(\tilde{X}, \tilde{Y}) = (\tilde{X}_0, \tilde{Y}_0), (\tilde{X}_1, \tilde{Y}_1), \ldots$ is a coupling
of $X'$ and $Y'$, but it is neither stationary nor Markovian.
For further examples and discussion of non-Markovian couplings of Markov processes, see 
\cite{ellis1976thedj, ellis1978distances, ellis1980kamae, ellis1980conditions}.

A joint process $(\tilde{X}, \tilde{Y})$ arising from a non-stationary or non-Markovian coupling of $\bbP$ and $\bbQ$ 
has a very different stochastic structure than the processes $X$ and $Y$ themselves, and will be difficult to work with
computationally. 
Thus we wish to exclude such couplings from the feasible set of an optimal transport problem.
%Then one may let $\calU = \X^\bbN$, $\calV = \Y^\bbN$, $\bfx = (x_0, x_1, ...)$, $\bfy = (y_0, y_1, ...)$ and couple $\bbP$ and $\bbQ$, obtaining a probability measure on the joint sequence space $(\X\times\Y)^{\mathbb{N}}$.
%However, such a coupling need not be stationary or Markovian.
%Non-stationary or non-Markovian couplings may exhibit infinitely long-range dependence and are thus unsuitable for computation in general.
%In order to avoid this problem and capture the dynamics of $X$ and $Y$, we might restrict the feasible set to couplings of $\bbP$ and $\bbQ$ 
%that have the same dependence structure as $X$ and $Y$
An obvious fix is to consider the family $\Pi_{\mbox{\tiny M}}(\bbP, \bbQ)$, defined as the set of couplings $\bbP$ and $\bbQ$ that are stationary and Markovian.
Viewed as processes, elements of $\Pi_{\mbox{\tiny M}}(\bbP, \bbQ)$ correspond to joint processes $(\tilde{X}, \tilde{Y})$ that are stationary, Markov, and satisfy $\tilde{X} \sim X$ and $\tilde{Y} \sim Y$.
While this is a natural choice, the optimal transport cost associated with $\Pi_{\mbox{\tiny M}}$ 
may violate the triangle inequality, 
even when the underlying cost function $c$ is itself a metric, see \cite{ellis1976thedj, ellis1978distances}.
Moreover, the family $\Pi_{\mbox{\tiny M}}(\bbP, \bbQ)$ is not characterized by a simple set of constraints \citep{boyle2009hidden}.
Motivated by the need for ready interpretation and tractable computation, we consider the set of 
stationary Markov chains on $\X \times \Y$ whose transition distributions 
are couplings of those of $X$ and $Y$.  A formal definition is given below.
The resulting set of couplings, called transition couplings, is characterized by a simple set of linear constraints 
involving $P$ and $Q$, and one may show (see Appendix \ref{app:existence}) that the resulting OT cost does satisfy the triangle inequality as long as the underlying cost $c$ does.

In order to reduce notation when considering vectors and matrices indexed by 
elements of $\X \times\Y$, we will indicate only the cardinality of the index set and adopt an 
indexing convention whereby a vector $u \in \mathbb{R}^{d^2}$ is indexed as $u(x, y)$ 
and a matrix $R \in [0,1]^{d^2 \times d^2}$ is indexed as $R((x,y),(x',y'))$ for $(x,y), (x',y') \in \X\times\Y$.
Note also that vectors of the form $R((x,y),\cdot)$ will be regarded as row vectors.

\begin{defn}\label{def:tc_mats}
Let $P$ and $Q$ be transition matrices on finite state spaces $\X$ and $\Y$, respectively.
A transition matrix $R \in [0, 1]^{d^2\times d^2}$ is a \textbf{transition coupling} of $P$ and $Q$ 
if for every paired-state $(x,y) \in \X \times \Y$, the distribution $R((x,y),\cdot)$ is a coupling
of the distributions $P(x, \cdot)$ and $Q(y,\cdot)$, formally  
$R((x,y),\cdot) \in \Pi(P(x,\cdot), Q(y,\cdot))$.
Let $\tc$ denote the set of all transition couplings of $P$ and $Q$.
\end{defn}

Standard results in Markov chain theory ensure that each transition coupling 
$R \in \tc$ admits at least one stationary distribution $r \in \Delta_{d^2}$.
Using $r$ and $R$, one may construct a stationary Markov chain $(\tilde{X}, \tilde{Y}) = \{(\tilde{X}_i, \tilde{Y}_i)\}_{i \geq 0}$ taking values in $\X \times \Y$.
We will also refer to couplings constructed in this way as transition couplings, as stated in the following definition.

\begin{defn}\label{def:transitioncouplings}
Let $X$ and $Y$ be stationary Markov chains with transition matrices $P$ and $Q$ on the finite state spaces $\X$ and $\Y$, respectively.
A stationary Markov chain $(\tilde{X}, \tilde{Y}) = \{(\tilde{X}_i, \tilde{Y}_i)\}_{i \geq 0}$ taking values in $\X\times\Y$ with transition matrix $R \in [0,1]^{d^2 \times d^2}$ is a \textbf{transition coupling} of $X$ and $Y$ if $(\tilde{X}, \tilde{Y})$ is a coupling of $X$ and $Y$ and $R \in \tc$.
\end{defn}

\noindent
Each transition coupling of $X$ and $Y$ may be associated with a process measure $\pi \in \calM_s(\X^\bbN \times \Y^\bbN)$;
let $\tcrv$ denote the set of all such measures induced by transition couplings of $X$ and $Y$.
As the notation suggests, one may readily show that the process measure $\pi$ induced by a transition coupling of $X$ and $Y$ is itself a coupling of the process measures $\bbP$ and $\bbQ$ associated with $X$ and $Y$, respectively.
As all elements of $\tcrv$ are also stationary and Markovian, it follows that $\tcrv \subset \Pi_{\mbox{\tiny M}}(\bbP, \bbQ)$.

The couplings defined in Definition \ref{def:transitioncouplings} are sometimes referred to as ``Markovian couplings'' in the literature \citep{levin2017markov}, and they have been used, for example, to study diffusions 
\citep{banerjee2018coupling, banerjee2016coupling, banerjee2017rigidity}.
We refer to such couplings as ``transition couplings'' in order to distinguish them from elements of 
$\Pi_{\mbox{\tiny M}}(\bbP, \bbQ)$.
Note that $\tcrv \neq \emptyset$ since it contains the independent coupling, namely, the stationary 
Markov chain on $\X \times \Y$ with transition matrix $P\otimes Q((x,y), (x',y')) = P(x,x') \, Q(y,y')$ for all $(x,y)$ and $(x',y')$.
The independent coupling corresponds to a paired chain $(\tilde{X}, \tilde{Y}) = \{(\tilde{X}_i, \tilde{Y}_i)\}_{i \geq 0}$ 
where $\tilde{X}$ and $\tilde{Y}$ are equal in distribution to $X$ and $Y$, respectively, and evolve independently of one another.

A key advantage of considering $\tcrv$ over $\Pi_{\mbox{\tiny M}}(\bbP, \bbQ)$ is that the constraints defining $\tcrv$ are linear and thus computationally tractable (the constraints defining $\Pi_{\mbox{\tiny M}}(\bbP, \bbQ)$ are not).
As we prove in Proposition \ref{prop:transmat_char} below, the set $\tc$ of transition matrices actually characterizes the set $\tcrv$ of transition couplings if $X$ and $Y$ are irreducible.
Stated differently, the condition $R \in \tc$ is sufficient to ensure that a chain $(\tilde{X}, \tilde{Y})$ with transition matrix $R$ is a transition coupling of $X$ and $Y$. 
On the other hand, if $X$ or $Y$ is reducible, a stationary Markov chain with a transition matrix in $\tc$ need not be a coupling of $X$ and $Y$ as the stationary distributions of $P$ and $Q$ are not unique.
This follows from the fact that a transition coupling of reducible chains may admit as marginals any of the chains with transition matrices $P$ or $Q$.
So in order to solve the OTC problem by optimizing over $\tc$ instead of $\tcrv$, we must be careful to avoid this situation.
Proposition \ref{prop:transmat_char} ensures that this cannot occur if $X$ and $Y$ are irreducible.

\begin{restatable}[]{prop}{transmatchar}
\label{prop:transmat_char}
%\begin{prop}\label{prop:transmat_char}
Let $X$ and $Y$ be irreducible stationary Markov chains with transition matrices $P$ and $Q$, respectively.
Then any stationary Markov chain with a transition matrix contained in $\tc$ is a transition coupling of $X$ and $Y$.
%\end{prop}
\end{restatable}

\noindent 
As a result of Proposition \ref{prop:transmat_char}, we may avoid working explicitly with transition couplings of $X$ and $Y$ and work instead with the set of matrices $\tc$.

%As a consequence of Proposition \ref{prop:transmat_char}, we may study $\tcrv$ through the set of transition matrices $\tc$.
%For brevity, we will also use 	``transition couplings'' to refer to elements of $\tc$.
Letting $c: \X^\bbN \times \Y^\bbN \to \mathbb{R}$ be a cost function defined on sample sequences 
of $X$ and $Y$, we define the \emph{optimal transition coupling (OTC) problem} for $X$ and $Y$ with cost $c$
to be the program
\begin{align}\label{eq:otc_problem_rv}
\begin{split}
\mbox{minimize} \quad &\int c \,  d\pi \\
\mbox{subject to} \quad &\pi \in \tcrv.
\end{split}
\end{align}
The minimum in \eqref{eq:otc_problem_rv}, referred to as the OTC cost, assesses the degree to which the two chains may be ``synced up'' with respect to $c$. 
Any solution to \eqref{eq:otc_problem_rv} describes the joint distribution of the synchronized chains.
Moreover, as a consequence of the pointwise ergodic theorem, any optimal transition coupling $\pi \in \tcrv$ in Problem \eqref{eq:otc_problem_rv} is also optimal with respect to the averaged cost $((x_0, x_1, ...), (y_0, y_1, ...)) \mapsto\limsup_{n\rightarrow\infty} \frac{1}{n} \sum_{i=0}^{n-1} c((x_i, x_{i+1}, ...), (y_i, y_{i+1}, ...))$.
In this sense, the quality of an alignment (equivalently, transition coupling) of the two chains $X$ and $Y$ is assessed based on its long-term average cost.
%For further discussion of this fact, we refer the reader to \cite{oconnor2021estimation}.

In the remainder of the paper we assume that $c$ is a single-letter cost,
%Abusing notation, this is equivalent to the condition 
i.e., $c((x_0, x_1, ...),$ $(y_0, y_1, ...)) = \tilde{c}(x_0, y_0)$ for some cost function 
$\tilde{c}: \X \times \Y \rightarrow \bbR_+$. 
In most of what follows we identify $c$ and $\tilde{c}$, regarding $c$ as a function on $\X \times \Y$ and 
writing $c(x_0,y_0)$ when no confusion will arise.
The consideration of single-letter costs is motivated by our focus on computation and reflects existing work on computational OT, where a cost or metric is defined \emph{a priori} on static observations.
Single letter costs have also been the focus of previous work on optimal transport problems for stationary processes \citep{ornstein1973application,gray1975generalization}.
Our arguments may be easily adapted to the case when the cost depends on a finite number of coordinates.
In particular, any $k$-letter cost $c: \mathcal{X}^k \times \mathcal{Y}^k \rightarrow \mathbb{R}_+$ may be regarded as single-letter for the chains $\tilde{X} = (X_0^{k-1}, X_1^k, ...)$ and $\tilde{Y} = (Y_0^{k-1}, Y_1^k, ...)$ on $\mathcal{X}^k$ and $\mathcal{Y}^k$, respectively.
For single-letter costs, we show in Appendix \ref{app:existence} that optimal transition couplings exist, and that the OTC cost satisfies the triangle inequality whenever $c$ does.
Note that $c$ is necessarily bounded, as $\X$ and $\Y$ are finite.
Moreover, there is no loss in generality in assuming that $c$ is non-negative since our results also hold after adding a constant to $c$.

A primary contribution of this paper, and the focus of Sections \ref{sec:computing} and \ref{sec:fast_approx_pia}, is the development of efficient algorithms for computing solutions to the OTC problem \eqref{eq:otc_problem_rv}.
Note that this problem involves the minimization of a linear objective over the non-convex set $\tcrv$, which makes it difficult to find a solution with off-the-shelf methods.
Proposition \ref{prop:transmat_char} shows that one may optimize instead over the convex polyhedron $\tc$:
informally, the program \eqref{eq:otc_problem_rv} can be reformulated as minimizing 
$\mathbb{E} c(\tilde{X}_0, \tilde{Y}_0)$ over $R \in \tc$, where $(\tilde{X}, \tilde{Y})$ is a stationary Markov chain generated by $R$.
However, this reformulation has a non-convex objective, so some care is needed in order to obtain global solutions.

\subsection{Related Work}\label{sec:related_work}

Stationary couplings of stationary processes, known as \emph{joinings}, were first studied in \citep{furstenberg1967disjointness}.
Distances between processes based on joinings have been proposed in the ergodic theory literature
\citep{gray1975generalization, ornstein1973application, oconnor2021estimation}, but they have been explored primarily as a theoretical tool: 
no tractable algorithms have been proposed for computing such distances exactly.
In the context of Markov chains, coupling methods have been widely used as a tool to establish rates of convergence (see for instance \cite{griffeath1976coupling} or \cite{lindvall2002lectures}).
Examples of optimal Markovian couplings of Markov processes are studied in \cite{ellis1976thedj, ellis1978distances, ellis1980conditions, ellis1980kamae}.
Another line of work has explored total variation-type distances for models with Markovian structure.
For example, \cite{chen2014total} and \cite{kiefer2018computing} develop algorithms for and consider the computability of the total variation distance between hidden Markov models and labeled Markov chains.
% KL distance for HMMs \citep{juang1985probabilistic}
Similarly, \citep{daca2016linear} studies the inestimability of the total variation distance between Markov chains.
More recent work has proposed direct adaptations of the optimal transport problem for processes with Markovian structure.
\cite{moulos2020bicausal} studied the bicausal optimal transport problem for Markov chains and its connection to Markov decision processes.
Unlike the OTC problem, in the bicausal transport problem, couplings are not required to be stationary or Markov themselves.
\cite{oconnor2021graph} applies the OTC problem and the tools presented in this paper to the comparison and alignment of graphs.
We also remark that the optimal transition coupling problem appears in the unpublished manuscript \citep{aldous2009april}.

Some existing work \citep{song2016measuring, zhang2000existence} has studied a modified form of the OTC problem that we refer to as the \emph{1-step} transition coupling problem.
In the 1-step transition coupling problem the expected cost is measured with respect to the 1-step 
transition probabilities rather than the stationary distribution of the transition coupling.
In particular, a transition coupling $R \in \tc$ is 1-step optimal if for every $(x, y) \in \X\times\Y$,
\begin{equation*}
R((x,y), \cdot) \in \argmin\limits_{r \in \Pi(P(x, \cdot), Q(y, \cdot))} \,
\sum\limits_{(x', y')} r(x',y') \, c(x',y').
\end{equation*}
%I.e., the expected cost of the next step is minimized rather than the long-run average cost of the chain.
Loosely, one can view the OTC problem \eqref{eq:otc_problem_rv} as an infinite-step version of the 1-step OTC problem, wherein a transition coupling is chosen that minimizes the expected cost averaged over an infinite number of steps.
The 1-step transition coupling problem appears in \citep{song2016measuring} where it is used to assess the distance between Markov decision processes.
In another direction, \citep{zhang2000existence} show that solutions to the 1-step transition coupling problem exist for Markov processes on Polish state spaces and lower semicontinuous cost functions.
While the 1-step problem is computationally convenient, in some situations it will yield poor alignments of the two chains of interest.  
We provide an example to illustrate this in Appendix \ref{app:one_step}, showing that the 1-step approach can yield a transition coupling with arbitrarily high expected cost over time.

Other work has considered modifications of standard computational OT techniques for time series that do not necessarily possess Markovian structure.
\cite{cazelles2019wasserstein} study the Wasserstein-Fourier distance, which is 
the Wasserstein distance between normalized power spectral densities, while \cite{muskulus2011wasserstein} suggest using the optimal transport cost between the $k$-block empirical measures constructed from observed samples.
For general observed sequences, \cite{su2018order} consider only couplings that do not disturb the ordering of 
the two sequences too much, as quantified by the inverse difference moment.
Another line of work \citep{cohen2020aligning, cuturi2017soft, janati2020spatio} has explored distances between time series based on dynamic time warping (DTW).
Similar in spirit to OT, the DTW problem seeks an alignment of observations in two time series that respects the ordering of the respective observations and minimizes a total cost.
In contrast to these approaches, we seek a more direct modification of the optimal transport problem itself that best captures the Markovian dynamics.

Entropic regularization in OT traces its roots back to traffic modeling techniques in transportation theory \citep{wilson1969use}.
%For example, \cite{erlander1980optimal} proposed a transport model that incorporates an entropy term to match observed traffic patterns better than standard transport plans.
\cite{cuturi2013sinkhorn} showed how one may solve the entropy-regularized OT problem via a matrix scaling algorithm proposed by \cite{sinkhorn1967diagonal}.
Owing to the increased computational efficiency of matrix multiplication over linear programming, Cuturi's result placed entropic OT as an efficient alternative to standard OT in high-dimensional (large $d$) scenarios.
\cite{altschuler2017near} provided further analysis of Sinkhorn's algorithm, showing that for appropriate choice of regularization coefficient and number of iterations, it yields an approximation of the unregularized OT cost in near-linear time.
More recent work \citep{dvurechensky2018computational, lin2019efficient, guo2020fast} has considered alternative algorithms based on stochastic gradient decent for solving entropy-regularized OT problems.
An entropy rate-regularized optimal joining problem and its statistical properties are studied \cite{oconnor2021estimation}.
We remark that an upper bound on the entropy of each transition distribution of a Markov chain (as considered in this paper) implies an upper bound on the entropy rate of the chain.

\section{Extension of OTC to Hidden Markov Models}
\label{sec:hmm}
Markov models are often employed as components of more complex models for sequential observations.
Hidden Markov models (HMMs) are a widely used variant of the Markov model in which observations are modeled as conditionally independent random emissions arising from a latent Markov chain.
HMMs have been applied successfully to a variety of problems including speech recognition \citep{bahl1986maximum,varga1990hidden}, 
text segmentation \citep{yamron1998hidden}, and modeling disease progression \citep{williams2020bayesian}.
For a detailed overview, we refer the reader to the text \citep{zucchini2017hidden}.

Formally, a HMM may be characterized by a pair $(X, \phi)$ where $X = (X_0, X_1, ...)$ is an unobserved Markov chain taking values in a finite set $\X$, and a function $\phi: \X \rightarrow \calM(\calU)$ that maps each state $x \in \X$ to a distribution on a
fixed observation space $\calU$.
The pair $(X, \phi)$ gives rise to a stationary process $U = (U_0, U_1, ...)$ where $U_0, U_1, \ldots \in \calU$ are 
conditionally independent given $X$ with $U_i \sim \phi(X_i)$ for $i \geq 0$.
Note that the process $U$ may exhibit long-range dependence.
In this way, HMMs provide a simple means of modeling sequences with more complex dependence structures.

The OTC problem may be extended to processes with hidden Markov structure as follows.
Let $(X, \phi)$ and $(Y, \psi)$ be a pair of HMMs with observation spaces $\calU$ and $\calV$, respectively, and let $c: \calU \times \calV \rightarrow \mathbb{R}_+$.
Note that the cost $c$ is specified on the observed spaces $\calU$ and $\calV$ rather than the state spaces of the 
unobserved Markov chains $X$ and $Y$.
However, one may extend $c$ to a cost on $\X \times \Y$ by optimally coupling the emission distributions 
$\phi(x)$ and $\psi(y)$ for every pair $(x, y) \in \X \times \Y$.
In more detail, let $\theta: \X\times\Y \rightarrow \calM(\calU \times \calV)$ and 
$c': \X\times\Y \rightarrow \mathbb{R}_+$ be defined by
\begin{equation*}
\theta(x,y) \in \argmin\limits_{\pi \in \Pi(\phi(x), \psi(y))} \int c \, d\pi \quad\quad \mbox{and} \quad\quad c'(x,y) = \min\limits_{\pi \in \Pi(\phi(x), \psi(y))} \int c \, d\pi .
\end{equation*}
In other words, we define the functions $\theta: \X\times\Y \rightarrow \calM(\calU \times \calV)$ and $c' : \X\times\Y \rightarrow \mathbb{R}_+$ such that for every $(x,y) \in \X\times\Y$, $\theta(x, y)$ is an optimal coupling and $c'(x,y)$ is the OT cost of the emission distributions $\phi(x)$ and $\psi(y)$ with respect to $c$.
One may then find an optimal transition coupling $(X', Y')$ of $X$ and $Y$ with respect to $c'$ as 
in problem \eqref{eq:otc_problem_rv}.
The expected cost of this transition coupling corresponds to a cost between the 
HMMs $(X, \phi)$ and $(Y, \psi)$ taking the original cost $c$ into account.
Moreover, the pair $((X', Y'), \theta)$ defines an optimal joint HMM of $(X, \phi)$ and $(Y, \psi)$ from which samples in $\calU \times \calV$ may be drawn.

Leveraging the intuition from the standard OTC problem, the optimal transition coupling $((X', Y'), \theta)$ may be thought of as an alignment of the two HMMs $(X, \phi)$ and $(Y, \psi)$ with respect to $c$.
In this way, we may apply the OTC problem to any processes that can be embedded as or are well-approximated by HMMs.
Before proceeding, we remark that \cite{chen2019aggregated} also proposes an OT problem for HMMs based on coupling the emission distributions of the two HMMs of interest.
However, the latent Markov chains of either HMM are coupled using standard OT after a registration step.
Our approach captures the Markovity of the latent sequences more directly and allows one to generate new samples from the coupled HMM.

\section{Computing Optimal Transition Couplings}
\label{sec:computing}

In this section, we turn our attention toward our primary goal of developing tractable algorithms for solving 
the OTC problem \eqref{eq:otc_problem_rv}.
As discussed in Section \ref{sec:background_on_otc}, the OTC problem is a non-convex, constrained optimization problem and thus there is little hope of obtaining global solutions via generic optimization algorithms.
Adopting a more tailored approach, we draw a connection between the OTC problem and Markov decision processes (MDP).
Having established this connection, we may leverage the wealth of algorithms for obtaining global solutions to MDPs to solve the OTC problem.
As we will show, the framework of policy iteration naturally lends itself to our problem and leads to a computationally tractable algorithm combining standard MDP techniques with OT solvers.

\subsection{Connection to Markov Decision Processes}
\label{sec:connection_to_mdp}

A Markov decision process is characterized by a 4-tuple $(\calS, \calA, \calP, c')$ consisting of a state space $\mathcal{S}$, an action space $\calA = \bigcup_s \calA_s$ where $\calA_s$ is the set of allowable actions in state $s$, a set of transition distributions $\calP = \{p(\cdot | s, a): s \in \calS, a \in \calA\}$ on $\calS$, and a cost function $c': \mathcal{S} \times \calA \rightarrow\mathbb{R}$.
At each time step the process occupies a state $s \in \mathcal{S}$ and an agent chooses an action $a \in \calA_s$; the
process incurs a cost $c'(s, a)$ and then moves to a new state according to the distribution $p(\cdot | s, a)$.
Informally, the goal of the agent is to choose actions to minimize her average cost.
The behavior of an agent is described by a family $\gamma = \{ \gamma_s(\cdot) : s \in  \mathcal{S} \}$ of 
distributions $\gamma_s(\cdot) \in \calM(\mathcal{A}_s)$ on the set of admissible actions, which is known as a {\em policy}. 
An agent following policy $\gamma$ chooses her next
action according to $\gamma_s(\cdot)$ whenever the system is in state $s$, independently of her previous actions.

It is easy to see that, in conjunction with the transition distributions $\calP$,
every policy $\gamma$ induces a collection of Markov chains on the state space $\mathcal{S}$ indexed by initial states $s \in \mathcal{S}$.
In the average-cost MDP problem the goal is to identify a policy for which the induced Markov chain minimizes the limiting average 
cost, namely a policy $\gamma$ minimizing
\begin{equation}\label{eq:mdp_average_cost}
\overline{c}_\gamma(s) := \lim\limits_{T \rightarrow\infty} \frac{1}{T}\sum\limits_{t=1}^T \mathbb{E}_\gamma \left[c'(s_t, a_t) \bigg| s_0 = s\right],
\end{equation}
for each $s \in \calS$.
Note that the expectation in \eqref{eq:mdp_average_cost} is taken with respect to the Markov chain induced by $\gamma$.
In general, the limiting average cost $\overline{c}_\gamma(s)$ will depend on the initial state $s$, but if $\gamma$ induces an ergodic chain then the average cost will be constant. 
If all policies induce ergodic Markov chains, the MDP is referred to as ``unichain''; otherwise the MDP is classified as ``multichain''.
We refer the reader to \cite{puterman2005markov} for more details on MDPs.

The OTC problem \eqref{eq:otc_problem_rv} may readily be recast as an MDP.
In detail, let the state space $\mathcal{S} = \X \times \Y$, and let $s = (x,y)$ denote an element of $\mathcal{S}$. 
Define the set of admissible actions in state $s$ to be the corresponding set of row couplings 
$\calA_s = \tcrow$. 
For each state $s$ and action $r_s \in \calA_s$ define the transition distribution $p(\cdot | s, r_s) := r_s(\cdot)$, and the cost function 
$c'(s, r_s) = c(s) = c(x,y)$.  Note that $c'$ is independent of the action $r_s$. 
We refer to this MDP as TC-MDP.
%When there is no risk of confusion, we will use $s$ in the place of $(x,y)$ from now on.

Any policy $\gamma$ for TC-MDP specifies distributions over $\tcrow$ for each $(x,y) \in \X\times\Y$ and thus corresponds to a single distribution over $\tc$ that governs the random actions of the agent.
In TC-MDP it suffices to consider only deterministic policies $\gamma$, namely policies such that for each state $s = (x,y)$ the distribution $\gamma_s(\cdot)$ is a point mass at unique element of $\calA_{s} = \tcrow$.

\begin{restatable}[]{prop}{deterministicpolicy}
\label{prop:deterministic_policy}
%\begin{prop}\label{prop:deterministic_policy}
Let $\gamma$ be a policy for TC-MDP.
Then there exists a deterministic policy $\tilde{\gamma}$ such that $\overline{c}_\gamma(s) = \overline{c}_{\tilde{\gamma}}(s)$ for every $s \in \calS$.
%\end{prop}
\end{restatable}

%As a result, we need only consider deterministic policies in TC-MDP.
%Note also that any deterministic policy specifies a unique action, $a \in \calA$.
%By the definition of $\calA$, this corresponds to a unique element of $\tc$.
Thus optimization over $\tc$ is equivalent to optimization over deterministic policies.
Importantly, {\em a deterministic policy corresponds to a fixed transition coupling matrix} $R \in \tc$.
Going forward, we refer to $R \in \tc$ directly instead of the equivalent deterministic policy $\tilde{\gamma}$ in our notation.
We note that, even when $X$ and $Y$ are ergodic, the same may not be true of the stationary Markov chain induced by a transition coupling matrix $R \in \tc$ (see Appendix \ref{sec:redicible_tc}).
Specifically, a single element of $\tc$ may have multiple stationary distributions and thus give rise to multiple stationary Markov chains depending on the initial state $s \in \calS$.
Thus TC-MDP is classified as multichain.
Finally, we may formalize the relationship between the OTC problem and TC-MDP.

\begin{restatable}[]{prop}{otcismdp}
%\begin{prop}
\label{prop:otc_is_mdp}
If $X$ and $Y$ are irreducible, then any $R \in \Pi(P, Q)$ that is an optimal policy for TC-MDP corresponds to an optimal 
coupling $\pi_R \in \tcrv$ with expected cost $\min_{s \in \calS} \overline{c}_R(s)$.
%\end{prop}
\end{restatable}

\subsection{Policy Iteration}
\label{sec:policy_iteration}

Now that we have shown that the OTC problem can be viewed as an MDP, we can leverage existing algorithms for MDPs to obtain solutions.
To this end, we propose to adapt the framework of policy iteration \citep{howard1960dynamic}.
To facilitate our discussion, in what follows, we regard the cost function $c$ and limiting average cost $\overline{c}_R$ 
as vectors in $\bbR_+^{d^2}$.
For each $R \in \tc$, standard results \citep{puterman2005markov} guarantee that the limit $\overline{R} := \lim_{T\rightarrow\infty} T^{-1} \sum_{t=0}^{T-1} R^t$ exists.
When $R$ is aperiodic and irreducible, the Perron-Frobenius theorem ensures 
that $\overline{R} = \lim_{T\rightarrow\infty} R^T$ and the rows of $\overline{R}$ are equal to the stationary distributions of $R$.

In policy iteration, one repeatedly {evaluates} and {improves} policies.
In the context of TC-MDP, for a given transition coupling matrix $R \in \tc$ the evaluation step computes the average cost (\emph{gain}) vector $g = \overline{R} \, c$ and the total extra cost (\emph{bias}) vector $h = \sum_{t=0}^\infty R^t(c - g)$.
In practice, $g$ and $h$ may be obtained by solving a linear system of equations rather than evaluating infinite sums (see Algorithm \ref{alg:exact_pe}) \citep{puterman2005markov}.
The improvement step selects a new transition coupling matrix $R'$ that minimizes $R' \, g$ or, 
if no improvement is possible, $R' \, h$ in an element-wise fashion (see Algorithm \ref{alg:exact_pi}).
%In general, such a minimizer need not exist.
In more detail, we may select a transition coupling $R'$ such that for each $(x,y)$ 
the corresponding row $r = R'((x,y),\cdot)$ minimizes $r g$ (or $r h$) over couplings $r \in \tcrow$.
To denote the element-wise argmin, we write $\elemargmin_{R \in \tc} R \, g$ (or $R \, h$).
The improved matrix $R'$ is obtained by solving $d^2$ OT problems with marginals 
$P(x,\cdot)$ and $Q(y,\cdot)$ and cost $g$ (or $h$).
This special feature of TC-MDP enables us to find improved transition coupling matrices in a computationally efficient manner despite working with an infinite action space.
Once a fixed point in the evaluation and improvement process is reached, the procedure terminates.
The resulting algorithm will be referred to as \texttt{ExactOTC} (see Algorithm \ref{alg:pia}).
%Corresponding pseudocode can be found in Algorithms \ref{alg:pia}, \ref{alg:exact_pe} and \ref{alg:exact_pi}.
We initialize Algorithm \ref{alg:pia} to the independent transition coupling $P \otimes Q$, defined in Section \ref{sec:background_on_otc}.

%%
%% ExactOTC
%%
%\begin{wrapfigure}{r}{0.50\linewidth}
\begin{center}
\begin{minipage}{0.64\textwidth}
\begin{algorithm}[H]
\DontPrintSemicolon
\SetAlgoLined
%\KwResult{Optimal transition coupling}
%\KwData{$P, Q$}
$R_0 \leftarrow P \otimes Q$, $n \leftarrow 0$\;
 \While{not converged}{
  \tcc{transition coupling evaluation}
  $(g_n, h_n) \leftarrow \texttt{ExactTCE}(R_n)$\;
  \tcc{transition coupling improvement}
  $R_{n+1} \leftarrow \texttt{ExactTCI}(g_n, h_n, R_n, \tc)$\;
  $n \leftarrow n+1$\;
 }
 \Return{$R_n$}
 \caption{\texttt{ExactOTC}}\label{alg:pia}
\end{algorithm}
\end{minipage}
\end{center}
%\vspace{-7mm}
%\end{wrapfigure}
%\vspace{-5mm}
%\hspace{2mm}
%%
%% ExactPE
%%
%\begin{center}
\begingroup
\setcounter{tmp}{\value{algocf}}% store current value of theorem counter
\setcounter{algocf}{0} %assign desired value to theorem counter
\renewcommand{\thealgocf}{1\alph{algocf}}
\begin{minipage}{0.52\textwidth}
\begin{algorithm}[H]
\SetKwInOut{Input}{input}
\DontPrintSemicolon
\SetAlgoLined
%\KwResult{Bias and gain vectors}
\Input{$R$}
 Solve for $(g, h, w)$ such that
  \begin{equation*}
  \left[\begin{array}{ccc} I - R & 0 &0 \\ I & I - R & 0 \\ 0 & I & I - R \end{array}\right] \left[\begin{array}{c} g \\ h \\ w \end{array}\right] = \left[\begin{array}{c} 0 \\ c \\ 0 \end{array}\right]
  \end{equation*}

  \Return{$(g, h)$}
  \caption{\texttt{ExactTCE}}\label{alg:exact_pe}
\end{algorithm}
\end{minipage}
\endgroup
\setcounter{algocf}{\value{tmp}}
%\hspace{5mm}
%%
%% ExactPI
%%
\begingroup
\setcounter{tmp}{\value{algocf}}% store current value of theorem counter
\setcounter{algocf}{1} %assign desired value to theorem counter
\renewcommand{\thealgocf}{1\alph{algocf}}
\begin{minipage}[c]{0.42\linewidth}
\begin{algorithm}[H]
\SetKwInOut{Input}{input}
\DontPrintSemicolon
\SetAlgoLined 
%\KwResult{Improved transition coupling}
\Input{$g, h, R_0, \Pi$}
\tcc{element-wise argmin}
 	 $R' \leftarrow \elemargmin_{R \in \Pi} R g$\;
  \If{$R'g = R_0 g$}{
 		 $R' \leftarrow \elemargmin_{R \in \Pi} R h$\;
 	\If{$R'h = R_0 h$}{
 		\Return{$R_0$}\;
 		}
 	\Else{
 		\Return{$R'$}\;
 	}
 	}
 \Else{
	 \Return{$R'$}\;
	}
 \caption{\texttt{ExactTCI}}\label{alg:exact_pi} 
\end{algorithm}
\end{minipage} 
\endgroup
\setcounter{algocf}{\value{tmp}}

For finite state and action spaces, policy iteration is known to yield an optimal policy for the average-cost MDP in a finite number of steps \citep{puterman2005markov}.
While policy iteration may fail to converge for general compact action spaces 
\citep{dekker1987counter, schweitzer1985undiscounted, puterman2005markov}, as is the case for TC-MDP,
we may exploit the polyhedral structure of $\tc$ to establish the following convergence result.

\begin{restatable}[]{thm}{convergenceofpia}
\label{thm:convergence_of_pia}
%\begin{thm}\label{thm:convergence_of_pia}
Algorithm \ref{alg:pia} converges to a solution $(g^*, h^*, R^*)$ of TC-MDP in a finite number of iterations.
Moreover, if $X$ and $Y$ are irreducible, $R^*$ is the transition matrix of an optimal transition coupling of $X$ and $Y$.
%\end{thm}
\end{restatable}

Recall from Proposition \ref{prop:otc_is_mdp} that an optimal solution to TC-MDP necessarily yields an optimal solution to \eqref{alg:pia}.
Thus Theorem \ref{thm:convergence_of_pia} ensures that a solution to the OTC problem can be obtained from Algorithm \ref{alg:pia} in a finite number of iterations.
%Our proof relies on the fact that solutions to \eqref{eq:otc_problem_rv} occur on the extreme points of $\tc$.
%As the set of transition couplings is a polytope, it has a finite number of extreme points.
%Since \texttt{ExactPI} only ever yields extreme points (as solutions to a linear program), the action space for our problem is effectively finite.
%Leveraging classical convergence results for multichain, average-cost MDP's, we then determine that the algorithm converges to a solution in a finite number of steps.
A proof of this result can be found in Section \ref{app:pia_convergence}.

\begin{rem}
One may in principle adapt other MDP algorithms to solve the OTC problem.
However, the standard alternatives to policy iteration either do not admit a computationally tractable implementation (e.g. linear programming) or are not as conducive to a convergence analysis (e.g. value iteration).
We choose policy iteration because it balances both of these features, admitting a practical implementation while also enabling a theoretical convergence analysis.
We acknowledge that OTC solvers based on policy iteration may not be preferable in every scenario and leave a detailed exploration of other MDP algorithms for the OTC problem to future work.
\end{rem}

\section{Fast Approximate Policy Iteration}\label{sec:fast_approx_pia}

The simplicity of Algorithm \ref{alg:pia} in conjunction with the theoretical guarantee of 
Theorem \ref{thm:convergence_of_pia} make it an appealing method for solving the OTC 
problem when the cardinality $d$ of the state spaces of $X$ and $Y$ is small.
However, each call to Algorithm \ref{alg:exact_pe} involves solving a system of $3d^2$ linear equations, requiring a total of $\calO(d^6)$ operations.
Furthermore, each call to Algorithm \ref{alg:exact_pi} entails solving $d^2$ linear programs each with $\calO(d)$ constraints, which can be accomplished in a total time of $\calO(d^5 \log d)$.
We note that a similar dependence on the dimension of each coupling is observed in exact OT algorithms, 
such as the network simplex algorithm in \citep{peyre2019computational}.
For even moderate values of $d$, this may be too slow for practical use.

To alleviate the poor scaling with the dimension of the couplings in the standard OT problem, one may use entropic regularization, 
whereby a negative entropy term is added to the OT objective.
\cite{cuturi2013sinkhorn} showed that solutions to the entropy-regularized OT problem may be obtained efficiently via Sinkhorn's algorithm \citep{sinkhorn1967diagonal}.
More recently, \cite{altschuler2017near} proved that Sinkhorn's algorithm yields an approximation of the OT cost with error bounded by $\varepsilon$ in near-linear time with respect to the dimension of the couplings under consideration.
Subsequent work \citep{dvurechensky2018computational, lin2019efficient, guo2020fast} has proposed and studied alternative algorithms for approximating the optimal transport cost, each with runtime scaling at least linearly with the dimension of the couplings in the problem.
%From a computational point of view, these are the best known bounds for arbitrary discrete measures.
One might hope that a similar dependence on the size of the elements of $\tc$ may be achievable for the OTC problem by employing regularization.

In this section, we extend entropic regularization techniques to the OTC problem.  This extension leads to 
an approximate algorithm that 
runs in $\tilde{\calO}(d^4)$ time per iteration, where $\tilde{\calO}(\cdot)$ omits non-leading poly-logarithmic factors.
This complexity is nearly-linear in the dimension $d^4$ of the transition couplings.
We first propose a truncation-based approximation of the \texttt{ExactTCE} transition coupling evaluation algorithm, which we call \texttt{ApproxTCE}.
When the transition coupling to be evaluated satisfies a simple regularity condition,
we show that one can obtain approximations of the gain and bias from \texttt{ApproxTCE} with error bounded by $\varepsilon$ in $\tilde{\calO}(d^4 \log \varepsilon^{-1})$ time.

Mirroring the derivation of entropic OT, we then propose an entropy-regularized approximation of the 
\texttt{ExactTCI} transition coupling improvement algorithm, called \texttt{EntropicTCI}.
We perform a new analysis of the Sinkhorn algorithm (described in Section \ref{sec:entropic_regularization}) that is tailored to transition coupling improvement to show that \texttt{EntropicTCI} yields an improved transition coupling with error bounded by $\varepsilon$ in $\tilde{\calO}(d^4 \varepsilon^{-4})$ time.
Combining these two algorithms, we obtain the \texttt{EntropicOTC} algorithm, which runs in $\tilde{\calO}(d^4 \varepsilon^{-4})$ time per iteration.
We provide empirical support for these theoretical results through a simulation study in Section \ref{sec:experiments}.
We find that the improved efficiency at each iteration of \texttt{EntropicOTC} leads to a much faster runtime in practice as compared to \texttt{ExactOTC}.
Our experiments also show that \texttt{EntropicOTC} yields an expected cost that closely approximates the unregularized OTC cost.

\subsection{Constrained Optimal Transition Coupling Problem}
We begin by defining a constrained set of transition couplings.
Let $\calK(\cdot \| \cdot)$ be the Kullback-Leibler (KL) divergence defined for $u, v \in \Delta_{d^2}$ 
by $\calK(u \| v) = \sum_{s} u(s) \log(u(s) / v(s))$ with the convention that $0 \log(0/0) = 0$ and $\calK(u \| v) = + \infty$
if $u(s) > 0$ and $v(s) = 0$ for some index $s$.
For every $\eta > 0$ and $(x,y) \in \X \times \Y$, define the set
\begin{equation*}
\tcetarow = \left\{ r \in \tcrow: \calK\big(r \| P\otimes Q((x,y),\cdot)\big) \leq \eta\right\},
\end{equation*}
and the subset of transition coupling matrices 
\begin{equation*}
\tceta = \{ R \in \tc: R((x,y),\cdot) \in \tcetarow, \, \forall (x,y) \in \X \times \Y\}.
\end{equation*}
%AN: removed
%The set $\tceta$ may be regarded as the transition coupling analogue of the constrained set of couplings $U_\alpha(r, c)$ (written as $\Pi_\alpha(r, c)$ in our notation) originally considered in \cite{cuturi2013sinkhorn}.
Elements of $\tceta$ have rows that are close in KL-divergence to the rows of the independent transition coupling $P \otimes Q$.
When $P$ and $Q$ are aperiodic and irreducible, the same is true of $P \otimes Q$.
Fix $\eta > 0$ and let $\tcetarv$ be the set of transition couplings with transition matrices in $\tceta$.  The \textit{entropic OTC problem} is
\begin{align}\label{eq:entropic_otc}
\begin{split}
\mbox{minimize} \quad &\int c \, d\pi \\
\mbox{subject to} \quad &\pi \in \tcetarv.
\end{split}
\end{align}
For completeness, we establish in Appendix \ref{app:existence} that a solution to \eqref{eq:entropic_otc} exists.
As the divergence $\calK(r \| P \otimes Q(s, \cdot))$ is bounded for $r \in \Pi(P(x,\cdot), Q(y, \cdot))$ 
and $s=(x,y) \in \X\times\Y$ \citep{cuturi2013sinkhorn},
% AN some changes
%, the same is true of $\calK(R(s,\cdot) \| P\otimes Q(s,\cdot))$ 
%for $R \in \tc$ for each $s \in \X \times \Y$.
the program \eqref{eq:entropic_otc} coincides with the unconstrained OTC problem for sufficiently large $\eta$.
%As such, as $\eta$ increases, the solutions of \eqref{eq:entropic_otc} should be close in an appropriate sense to the solutions of \eqref{eq:otc_problem_rv}.
Finally, note that \eqref{eq:entropic_otc} corresponds to an MDP in the same way that \eqref{eq:otc_problem_rv} does but with a constrained set of policies.
In the rest of the section, we develop computationally efficient alternatives to Algorithms \ref{alg:exact_pe} and \ref{alg:exact_pi} for this constrained MDP.

\subsection{Fast Approximate Transition Coupling Evaluation}
Next, we propose a fast approximation of Algorithm \ref{alg:exact_pe}.
Recall from our previous discussion that the gain vector $g$ corresponding to any aperiodic and irreducible $R \in \tc$ is constant and thus may be written as $g = g_0 \mathbbm{1}$ for a scalar $g_0$.
Fixing such an $R \in \tc$ and $L, T \geq 1$, we approximate the gain $g$ by averaging the cost over $L$ steps of the Markov chain corresponding to $R$ from each possible starting point in $\X\times \Y$.
Moreover, we approximate the bias $h$ by summing the total extra cost over $T$ steps with respect to the approximate gain $\tilde{g}$. 
Formally, let $\tilde{g} := (d^{-2} (R^L c)^\top \mathbbm{1})\mathbbm{1}$ and $\tilde{h} := \sum_{t=0}^T R^t(c - \tilde{g})$.
The resulting algorithm, which we refer to as \texttt{ApproxTCE}, is detailed in Algorithm \ref{alg:fast_pe}.

%\begin{wrapfigure}{r}{0.28\textwidth}
%\vspace{-5mm}
\begingroup
\setcounter{tmp}{\value{algocf}}% store current value of theorem counter
\setcounter{algocf}{0} %assign desired value to theorem counter
\renewcommand{\thealgocf}{2\alph{algocf}}
\begin{center}
\begin{minipage}[c]{0.32\textwidth}
\begin{algorithm}[H]
\SetKwInOut{Input}{input}
\DontPrintSemicolon
\SetAlgoLined 
%\KwResult{Approximate gain and bias}
\Input{$R$, $L$, $T$}
$\tilde{g} \leftarrow (d^{-2} (R^L c)^\top \mathbbm{1})\mathbbm{1}$\;
$\tilde{h} \leftarrow \sum_{t=0}^T R^t(c - \tilde{g})$\;
\Return{$(\tilde{g}, \tilde{h})$}
\caption{\texttt{ApproxTCE}}\label{alg:fast_pe} 
\end{algorithm}
\end{minipage} 
\end{center}
\endgroup
%\vspace{-5mm}
%\end{wrapfigure}

The approximations $\tilde{g}$ and $\tilde{h}$ can be computed in $\mathcal{O}(L d^4)$ and $\calO(T d^4)$ time, respectively. 
Since $g$ and $h$ are equal to the limits of $\tilde{g}$ and $\tilde{h}$ as $L, T \rightarrow \infty$, we expect that larger $L$ and $T$ will yield better approximations.
One must ensure that the $L$ and $T$ that are required for a good approximation do not grow too quickly with $d$.
We show that this is the case in Theorem \ref{thm:policy_evaluation_complexity} below.
We will say that a transition matrix $R \in [0,1]^{d^2 \times d^2}$ with stationary distribution $\lambda \in \Delta_{d^2}$ is mixing with coefficients $M \in \mathbb{R}_+$ and $\alpha \in [0,1)$ if for every $t \in \mathbb{N}$, $\max_{s \in \mathcal{X}\times\mathcal{Y}} \|R^t(s, \cdot) - \lambda\|_1 \leq M \alpha^t$.
Recall that $R$ is mixing whenever it is aperiodic and irreducible.
%\vspace{1mm}
\begin{restatable}[]{thm}{policyevaluationcomplexity}
\label{thm:policy_evaluation_complexity}
%\begin{lemma}\label{lemma:policy_evaluation_complexity}
%Fix  $\varepsilon > 0$ and $R \in \tceta$ and let $(g, h)$ be the output of $\emph{\texttt{ExactPE}}(R)$.
Let $R \in \tc$ be aperiodic and irreducible with mixing coefficients $M \in \bbR_+$ and $\alpha \in [0,1)$ and gain and bias vectors $g \in \bbR^{d^2}$ and $h \in \bbR^{d^2}$, respectively.
Then for any $\varepsilon > 0$, there exist $L, T \in \mathbb{N}$ such that $\emph{\texttt{ApproxTCE}}(R, L, T)$ yields $(\tilde{g}, \tilde{h})$ satisfying $\|\tilde{g} - g\|_\infty \leq \varepsilon$ and $\|\tilde{h} - h\|_1 \leq \varepsilon$ in $\tilde{\calO}\left(\frac{d^4}{\log \alpha^{-1}} \log\left(\frac{M}{\varepsilon (1 - \alpha)}\right)\right)$ time.
%\end{lemma}
\end{restatable}
\noindent In particular, \texttt{ApproxTCE} does approximate \texttt{ExactTCE} in time scaling like $\tilde{\calO}(d^4)$.
Explicit choices of $L$ and $T$ are given in the proof of Theorem \ref{thm:policy_evaluation_complexity}, which may be found in Section \ref{app:complexity}.

\begin{rem}
\label{rem:adaptive_LandT}
In practice, values of $L$ and $T$ satisfying the conclusion of Theorem \ref{thm:policy_evaluation_complexity} are unknown.
%In most cases, Algorithm \ref{alg:fast_pe} will tend to run more quickly than Algorithm \ref{alg:exact_pe} even for large fixed values of $L$ and $T$.
In our experiments, we found that running Algorithm \ref{alg:fast_pe} with large, fixed values of $L$ and $T$ yields a high approximation accuracy while still running significantly more quickly than Algorithm \ref{alg:exact_pe}.
Alternatively, $L$ and $T$ may be chosen adaptively by computing vectors $\tilde{g}$ and $\tilde{h}$ iteratively for larger and larger values of $L$ and $T$ until some convergence criterion is satisfied or $L$ and $T$ hit some prespecified thresholds.
For example, letting $\tilde{g}^L$ and $\tilde{h}^T$ be the iterates of this procedure, one may iterate until $\|\tilde{g}^L - \tilde{g}^{L-1}\|_\infty < \varepsilon$ and $\|\tilde{h}^T - \tilde{h}^{T-1}\|_\infty < \varepsilon$.
This approach achieves the same worst-case complexity as Algorithm \ref{alg:fast_pe} but allows for time-savings when the chain $R$ mixes quickly.
\end{rem}

For a set $\calU \subset \bbR^n$, let $B_\varepsilon(u) \subset \bbR^n$ be the open ball of radius $\varepsilon > 0$ centered at $u \in \calU$, and let $\aff(\calU)$ denote the affine hull, defined as $\aff(\calU) = \{\sum_{i=1}^k \alpha_i u_i : k \in \mathbb{N}, u_1, ..., u_k \in \calU, \sum_{i=1}^k \alpha_i = 1\}$.
Let $\ri(\cdot)$ denote the relative interior, defined as $\ri(\calU) =\{ u \in \calU: \exists \varepsilon > 0 \mbox{  s.t.  } B_\varepsilon(u) \cap \aff(\calU) \subset \calU\}$.

%AN: Removed
%As discussed previously, a transition coupling matrix $R \in \tc$ need not be aperiodic and irreducible even if $P$ and $Q$ are.
%As such Theorem \ref{thm:policy_evaluation_complexity} will not apply for all $R \in \tc$.
%However, one may verify that as long $P$ and $Q$ are aperiodic and irreducible, \emph{most} transition couplings of $P$ and $Q$ are as well.
%In particular, we prove in Proposition \ref{prop:interior_tc} that every $R \in \ri(\tc)$, where $\ri(\cdot)$ denotes the relative interior, is aperiodic and irreducible and thus satisfy the conditions of Theorem \ref{thm:policy_evaluation_complexity}.

\begin{restatable}[]{prop}{interiortc}
\label{prop:interior_tc}
If $P$ and $Q$ are aperiodic and irreducible
then every $R \in \ri(\tc)$ is also aperiodic and irreducible, and thus mixing.
\end{restatable}

As a consequence of Proposition \ref{prop:interior_tc}, we need only verify that $R \in \ri(\tc)$ to ensure that Theorem \ref{thm:policy_evaluation_complexity} holds and that we may perform fast transition coupling evaluation via \texttt{ApproxTCE}.
As we show in Theorem \ref{thm:policy_improvement_complexity}, this condition is naturally guaranteed when employing entropic OT techniques for speeding up the transition coupling improvement step.

\subsection{Entropic Transition Coupling Improvement}\label{sec:entropic_regularization}
Next we describe a means of speeding up Algorithm \ref{alg:exact_pi}.
%Note that since the gain vector for any element of $\tceta$ is constant, we need only improve policies with respect to the bias vector.
For the MDP corresponding to the entropic OTC problem, exact policy improvement can be performed by calling \texttt{ExactTCI} with $\Pi = \tceta$.
However, no computation time is saved by doing this.
Instead, we settle for an algorithm that yields approximately improved transition couplings with better computational efficiency.
To find such an approximation, we reconsider the linear optimization problems that comprise the transition coupling improvement step.
Namely, for each $(x,y) \in \X \times \Y$,
\begin{align}\label{eq:ent_policy_improvement}
\begin{split}
\mbox{minimize} \quad & \langle r, h \rangle \\
\mbox{subject to} \quad &r \in \tcetarow.
\end{split}
\end{align}
By standard arguments, \eqref{eq:ent_policy_improvement} is equivalent to
% \citep{dessein2018regularized}, 
\begin{align}\label{eq:penalized_ot}
\begin{split}
\mbox{minimize} \quad &\langle r, h \rangle + \frac{1}{\xi} \sum_{s'} r(s') \log r(s') \\
\mbox{subject to} \quad & r\in \tcrow,
\end{split}
\end{align}
for some $\xi \in [0, \infty]$ depending on $(x,y)$, $\eta$ and $h$.
%Problem \eqref{eq:penalized_ot} is an instance of an entropy-regularized OT problem.
The reformulation \eqref{eq:penalized_ot}
suggests that one use computational techniques for entropic OT in the place of linear programming to perform transition coupling improvement for the constrained OTC problem.
%Entropy regularization is a common technique for obtaining fast approximations of the OT cost \citep{peyre2019computational}.
In particular, we use the \texttt{ApproxOT} algorithm of \citep{altschuler2017near}, detailed in Appendix \ref{app:complexity}.
Using \texttt{ApproxOT} instead of solving \eqref{eq:penalized_ot} exactly, we obtain the \texttt{EntropicTCI} algorithm detailed in Algorithm \ref{alg:approx_entropic_pi}.

%\begin{wrapfigure}{r}{0.52\textwidth}
%\vspace{-5mm}
\begingroup
\setcounter{tmp}{\value{algocf}}% store current value of theorem counter
\setcounter{algocf}{1} %assign desired value to theorem counter
\renewcommand{\thealgocf}{2\alph{algocf}}
\begin{center}
\begin{minipage}[c]{0.63\textwidth}
\begin{algorithm}[H]
\SetKwInOut{Input}{input}
\DontPrintSemicolon
\SetAlgoLined 
%\KwResult{Approximately improved transition coupling}
\Input{$h, \xi, \varepsilon$}

   \For{$(x, y) \in \X \times\Y$}{
 	 $R(s, \cdot) \leftarrow \texttt{ApproxOT}(P(x,\cdot)^\top, Q(y,\cdot)^\top, h, \xi, \varepsilon)$\;
 }
% \If{$R' h == R_{\tiny \mbox{old}} h$}{
% 	\Return{$R_{\tiny \mbox{old}}$}\;
% 	}
% \Else{
 	\Return{$R$}\;
% }
 \caption{\texttt{EntropicTCI}}\label{alg:approx_entropic_pi} 
\end{algorithm}
\end{minipage} 
\end{center}
\endgroup
%\vspace{-5mm}
%\end{wrapfigure}

To provide further intuition for Algorithm \ref{alg:approx_entropic_pi}, it is helpful to consider the constrained OTC problem from an alternate perspective.
For a probability measure $r \in \Delta_{d^2}$, let $H(r) = -\sum_s r(s) \log r(s)$ be its entropy.
Then by duality theory, the constrained OTC problem \eqref{eq:entropic_otc} may be written as the finite-dimensional optimization problem
\begin{align}\label{eq:entropic_otc_finite}
\begin{split}
\mbox{minimize} \quad & \langle c, \lambda \rangle - \sum\limits_s \frac{1}{\xi(s)} H(R(s,\cdot))\\
\mbox{subject to} \quad & R \in \Pi_{\mbox{\tiny TC}}(P, Q) \\
& \lambda R = \lambda \\
& \langle \mathbbm{1}, \lambda \rangle = 1,
\end{split}
\end{align}
for some $\xi \in [0, \infty]^{d^2}$.
In order to solve the problem above, we study its Lagrangian.
Let $\alpha, \beta \in \mathbb{R}^{d^3}$, $\gamma \in \mathbb{R}^{d^2}$, and $\delta \in \mathbb{R}$ be Lagrange multipliers.
The Lagrangian may be written as
\begin{align*}
\mathcal{L}(R, \lambda, \alpha, \beta, \gamma, \delta) &= \langle c, \lambda \rangle - \sum\limits_{x,y} \frac{1}{\xi(x,y)} H(R((x,y),\cdot)) \\
& \quad + \sum\limits_{x,y,x'} \alpha(x,y,x') \left(\sum\limits_{y'} R((x,y), (x',y')) - P(x,x')\right) \\
& \quad + \sum\limits_{x,y,y'} \beta(x,y,y') \left(\sum\limits_{x'} R((x,y), (x',y')) - Q(y,y')\right) \\
& \quad + \sum\limits_{x',y'} \gamma(x',y') \left(\sum\limits_{x,y} \lambda(x,y) R((x,y),(x',y')) - \lambda(x',y')\right) \\
& \quad + \delta \left(\sum\limits_{x,y} \lambda(x,y) - 1\right).
\end{align*}
Taking the partial derivative of $\mathcal{L}$ with respect to $R((x,y), (x',y'))$ and setting it equal to zero, 
we find that  
\begin{equation*}
R(s, (x',y')) = \exp\left\{ -\xi(s) \alpha(s,x') - \frac{1}{2}\right\} \exp\left\{-\xi(s) \lambda(s) \gamma(x',y')\right\} \exp\left\{-\xi(s) \beta(s,y') - \frac{1}{2}\right\},
\end{equation*}
where we have used $s = (x,y)$ to reduce notation.
When viewed as a $d \times d$ matrix, $R((x,y), \cdot)$ can be written as $U K V$ where $U$ and $V$ are both non-negative diagonal matrices.  Note that when $\xi(x,y) < \infty$, this implies that $R$ is aperiodic and irreducible since $R$ lies in the relative interior of $\tc$ (see Theorem \ref{thm:policy_improvement_complexity}).

A similar matrix form appears in the analysis of \cite{cuturi2013sinkhorn}.
An important difference is the matrix $K \in \mathbb{R}_+^{d \times d}$, which satisfies
\begin{equation*}
K(x',y') = \exp\left\{-\xi(x,y) \lambda(x,y) \gamma(x',y')\right\}.
\end{equation*}
In \cite{cuturi2013sinkhorn}, one finds that $K = e^{-\xi C}$ where $C$ is the cost matrix.
To better understand this difference, it is helpful to look at the partial derivative of the Lagrangian with respect to $\lambda(x,y)$.
Evaluating this partial derivative and setting it equal to zero, we find
\begin{equation*}
\gamma(x,y) = c(x,y) + \sum\limits_{x',y'} R((x,y), (x',y')) \gamma(x',y') + \delta.
\end{equation*}
Absorbing the scalar $\delta$ into $c$ to obtain an augmented cost $\tilde{c} = c + \delta$, we have
\begin{equation*}
\gamma(x,y) = \tilde{c}(x,y) + \sum\limits_{x',y'} R((x,y), (x',y')) \gamma(x',y').
\end{equation*}
Letting $g$ be the gain of the policy $R$ with respect to $\tilde{c}$, we recognize that the equation above is the Bellman recursion for the bias of $R$ with respect to the cost $\tilde{c} + g$.
As $R$ is aperiodic and irreducible, $g$ is a constant vector.
Moreover, as the bias is invariant under constant shifts in cost, $\gamma$ is exactly the bias $h$ that 
appears in \texttt{EntropicTCI}.
Returning to the form of $R((x,y), \cdot)$ established earlier, we find that
\begin{equation*}
R((x,y), \cdot) = U \exp\left\{-\xi(x,y) \lambda(x,y) h\right\} V = U \exp\{-\tilde{\xi}(x,y) h \} V,
\end{equation*}
for non-negative diagonal matrices $U$ and $V$ and the constant $\tilde{\xi}(x,y) := \xi(x,y) \lambda(x,y)$.
In this way, the bias $h$ plays the role of the cost matrix $C$ of \cite{cuturi2013sinkhorn}.

In order to solve the program \eqref{eq:entropic_otc_finite}, one must grapple with the interdependence between the bias $h$ and the policy $R$.
A natural approach for doing so is to consider an alternating optimization algorithm in which one repeatedly solves for the bias $h$ from a given policy $R$, then solves for a new policy $R$ given the bias $h$.
Indeed, this is the procedure one follows in \texttt{ExactOTC}.
In practice, given a policy $R$, one approximately computes the bias $h$ (\texttt{ApproxTCE}) in order to save time.
Given a bias vector $h$, one solves for a new policy $R$ by performing Sinkhorn iterations with the bias $h$ as a cost matrix for each $R((x,y),\cdot)$ (\texttt{EntropicTCI}).

It was shown in \cite{altschuler2017near} that \texttt{ApproxOT} yields an approximation of the OT 
cost in near-linear time with respect to the size of the couplings of interest.
However, in order to control the approximation error of \texttt{EntropicTCI}, we rely on a different analysis showing that one can obtain an approximation of the entropic optimal coupling in near-linear time (see Lemma \ref{lemma:bound_on_sinkhorn_error}) .
To the best of our knowledge, this result does not exist in the literature, so we provide a proof in Section \ref{app:complexity}.
Using this result, we show the complexity bound below.

\begin{restatable}[]{thm}{policyimprovementcomplexity}
\label{thm:policy_improvement_complexity}
%\begin{lemma}\label{lemma:policy_improvement_complexity}
Let $P$ and $Q$ be aperiodic and irreducible, $h \in \bbR^{d^2}$, $\xi > 0$, and $\varepsilon > 0$.
Then $\emph{\texttt{EntropicTCI}}(h, \xi, \varepsilon)$ returns $\hat{R} \in \ri(\tc)$ with $\max_s \|\hat{R}(s,\cdot) - R^*(s,\cdot)\|_1 \leq \varepsilon$ for some $R^* \in \argmin_{R' \in \tc} R'h - \nicefrac{1}{\xi} H(R')$ in $\tilde{\calO}(d^4 \varepsilon^{-4})$ time.
%\begin{equation}
%\tilde{\mathcal{O}}\left(\frac{d^4 \xi M\alpha \kappa^4 \|c\|_\infty^5}{\varepsilon^4 (1 - \alpha)}\left(\frac{\xi^2 M^2 \alpha^2 \|c\|_\infty^2}{(1 - \alpha)^2} + (\log b)^2\right)\right),
%\end{equation}
%where $b = \minplus\left\{P(x,x')\right\} \wedge \minplus\left\{Q(y,y')\right\}$, $\xi=\max_{(x,y)}\xi_{(x,y)}$ and $\kappa$ is defined in the proof.
%\end{lemma}
\end{restatable}

\noindent To summarize, this result states that \texttt{EntropicTCI} yields an approximately improved transition coupling in $\tilde{\calO}(d^4)$ time rather than $\tilde{\calO}(d^5)$ as previously discussed.
In practice, further speedups are possible by utilizing the fact that the $d^2$ entropic OT problems to be solved are decoupled and thus may be computed in parallel.

\subsection{\texttt{EntropicOTC}}
Finally, using Algorithms \ref{alg:fast_pe} and \ref{alg:approx_entropic_pi}, we define the \texttt{EntropicOTC} algorithm, detailed in Algorithm \ref{alg:fastentropic_pia}.
Essentially, \texttt{EntropicOTC} is defined by replacing \texttt{ExactTCE} and \texttt{ExactTCI} by the efficient alternatives, \texttt{ApproxTCE} and \texttt{EntropicTCI}.
As stated in Theorem \ref{thm:policy_improvement_complexity}, \texttt{EntropicTCI} returns transition couplings in the relative interior of $\tc$, so the iterates of \texttt{EntropicOTC} are not restricted to the finite set of extreme points of $\tc$.
Thus, convergence for Algorithm \ref{alg:fastentropic_pia} must be assessed differently than in Algorithm \ref{alg:pia}.
In our simulations we found that the element-wise inequality $\tilde{g}_{n+1} \geq \tilde{g}_n$ works well as an indicator of convergence.

%\begin{wrapfigure}{r}{0.48\linewidth}
%\vspace{-13mm}
%\vspace{-5mm}
\setcounter{algocf}{1} %assign desired value to theorem counter
\begin{center}
\begin{minipage}[c]{0.64\textwidth}
\begin{algorithm}[H]
\SetKwInOut{Input}{input}
\DontPrintSemicolon
\SetAlgoLined
%\KwResult{Approximate optimal transition coupling}
\Input{$L, T, \xi, \varepsilon$}
$n \leftarrow 0$\;
 \While{$n = 0$ or $\tilde{g}_{n+1} < \tilde{g}_n$}{
  \tcc{transition coupling evaluation}
 $(\tilde{g}_n, \tilde{h}_n) \leftarrow \texttt{ApproxTCE}(R_n, L, T)$\;
 \noindent\noindent\tcc{transition coupling improvement}
 $R_{n+1} \leftarrow \texttt{EntropicTCI}(\tilde{h}_n, \xi, \varepsilon)$\;
 $n \leftarrow n+1$\;
 }
\Return{$R_{n+1}$}\;
 \caption{\texttt{EntropicOTC}}\label{alg:fastentropic_pia}
\end{algorithm}
\end{minipage}
\end{center}
%\vspace{-5mm}
%\end{wrapfigure}

\section{Consistency}
\label{sec:consistency}

The computational and theoretical results presented above assume that one has complete knowledge of the 
transition matrices $P$ and $Q$ of the Markov chains $X$ and $Y$ under study.
In practice, one may not have direct access to $P$ and $Q$, but may instead have estimates 
$\hat{P}_n$ and $\hat{Q}_n$ derived from $n$ observations of the chains $X$ and $Y$.
In the simplest case, $\hat{P}_n$ and $\hat{Q}_n$ may be obtained from the observed relative
frequencies of each transition between states.
In Theorem \ref{thm:consistency} below, we show that the cost and solution sets of the standard and regularized
optimal transition coupling problems possess natural stability properties with respect to the marginal transition matrices. 
As a corollary, we obtain a consistency result for the OTC problem applied to the estimates $\hat{P}_n$ and $\hat{Q}_n$.

Recall that we use $\Delta_d$ to denote the probability simplex in $\bbR^d$ and note that the set of $d \times d$-dimensional transition matrices may be written as $\Delta_d^d$.
Likewise, the set of $d^2 \times d^2$-dimensional transition matrices may be written as $\Delta_{d^2}^{d^2}$.
Note that we endow the sets of $d\times d$- and $d^2 \times d^2$-dimensional transition matrices with the topologies they inherit as subsets of $\mathbb{R}^{d\times d}$ and $\mathbb{R}^{d^2 \times d^2}$, respectively, and adopt the same convention for the set $\Delta_{d^2} \times \Delta_{d^2}^{d^2}$.
Now, we may reformulate Problems \eqref{eq:otc_problem_rv} and \eqref{eq:entropic_otc} as follows:
\vspace{-8mm}
\begin{multicols}{2}
\begin{align}
\begin{split}
\mbox{minimize} \quad & \langle c, \lambda \rangle \\
\mbox{subject to}\quad & R \in \Pi(P, Q) \\
& \lambda R = \lambda \\
& \lambda \in \Delta_{d^2}.
\end{split}\tag{I}
\label{eq:probI}
\end{align}

\begin{align}
\begin{split}
\mbox{minimize} \quad & \langle c, \lambda \rangle \\
\mbox{subject to}\quad & R \in \Pi_\eta(P, Q) \\
& \lambda R = \lambda \\
& \lambda \in \Delta_{d^2}.
\end{split}\tag{II}
\label{eq:probII}
\end{align}
\end{multicols}
\noindent 
Let $\rho(P, Q)$ and $\rho_\eta(P, Q)$ denote the optimal values of Problems \eqref{eq:probI} and \eqref{eq:probII}, 
respectively, and let
$\Phi^*(P, Q)$ and $\Phi^*_\eta(P, Q)$ denote the associated sets of optimal solutions 
$(\lambda, R) \in \Delta_{d^2} \times \Delta_{d^2}^{d^2}$ to Problems \eqref{eq:probI} and \eqref{eq:probII}, respectively.
For metric spaces $\mathcal{U}$ and $\mathcal{Z}$, we will say that a function $F: \calU \rightarrow 2^\calZ$ is upper semicontinuous at a point $u_0 \in \calU$ 
if for any neighborhood $V$ of $F(u_0)$, there exists a neighborhood $U$ of $u_0$ such that $F(u) \subset V$
for every $u \in U$.

\begin{restatable}[]{thm}{consistency}
\label{thm:consistency}
Let $P, Q \in \Delta_d^d$ be irreducible transition matrices.
Then the following hold:
\begin{itemize}
\item $\rho(\cdot, \cdot)$ is continuous and $\Phi^*(\cdot, \cdot)$ is upper semicontinuous at $(P, Q)$
\item For any $\eta > 0$, $\rho_\eta(\cdot, \cdot)$ is continuous and $\Phi^*_\eta(\cdot, \cdot)$ is upper semicontinuous at $(P, Q)$
\end{itemize}
\end{restatable}

Theorem \ref{thm:consistency} states that the optimal values and optimal solution sets of the OTC and entropic OTC problems are stable in the marginal transition matrices $P$ and $Q$.
We may use this result to prove a consistency result for either problem when applied to estimates $\hat{P}_n$ and $\hat{Q}_n$ derived from data.
In stating the following result, we make use of the following definition: 
For a sequence of sets $\{A_n\}_{n \geq 0}$ in a topological space $\mathcal{A}$, let
$\limsup_{n \rightarrow \infty} A_n = \bigcap_{n=0}^\infty \cl\left(\bigcup_{m=n}^\infty A_m\right)$, where $\cl(\cdot)$ denotes the closure with respect to topology of $\mathcal{A}$.
Note that the presence of $\cl(\cdot)$ in our definition of limit superior of a sequence of sets differs from that commonly used in probability but is consistent with the definition appearing, for example, in \cite{rockafellar2009variational}.

\begin{cor}
\label{cor:consistency}
Let $X = \{X_i\}_{i \geq 0}$ and $Y = \{Y_i\}_{i\geq 0}$ be stationary, ergodic processes taking values in $\X$ and $\Y$, and defined on a common Borel probability space.
Suppose further that $X$ and $Y$ have marginal, one-step transition matrices $P$ and $Q$, respectively.
Let $\hat{P}_n$ and $\hat{Q}_n$ be the one-step transition matrices estimated via relative frequencies from the sequences $X_0, ..., X_{n-1}$ and $Y_0, ..., Y_{n-1}$.
Then with probability one, the following hold: 
\begin{itemize}

\item $\rho(\hat{P}_n, \hat{Q}_n) \rightarrow \rho(P, Q)$ 
\, and \, $\limsup\limits_{n\rightarrow\infty} \Phi^*(\hat{P}_n, \hat{Q}_n) \subseteq \Phi^*(P, Q)$

\item For any $\eta > 0$, $\rho_\eta(\hat{P}_n, \hat{Q}_n) \rightarrow \rho_\eta(P, Q)$ 
\, and \, $\limsup\limits_{n\rightarrow\infty} \Phi^*_\eta(\hat{P}_n, \hat{Q}_n) \subseteq \Phi^*_\eta(P, Q)$

\end{itemize}
\end{cor}

Corollary \ref{cor:consistency} allows us to apply the computational tools 
described above to real data in a principled manner.
In particular, when the marginal transition matrices $P$ and $Q$ are unknown, 
we may use $\hat{P}_n$ and $\hat{Q}_n$ as proxies in the OTC problem to estimate the 
set of optimal transition couplings and their expected cost when $n$ is large.
Note that we do not require the generating processes themselves to be Markov:
they need only be stationary and ergodic, so that the estimates $\hat{P}_n$ and $\hat{Q}_n$ 
converge to the true one-step transition matrices $P$ and $Q$ as $n$ tends to infinity.

%\begin{rem}
%As stated, Theorem \ref{thm:consistency} and Corollary \ref{cor:consistency} do not apply to the OTC problem for HMMs described in Section \ref{sec:background_on_otc}.
%In particular, the uncertainty in the cost $c'$ and $\theta$, arising from any uncertainty in $\phi$ and $\psi$, is not accounted for in the statement of either result.
%However, we believe the proofs extend in a relatively straightforward manner.
%We leave a formal argument to future work.
%\end{rem}

\section{Experiments}
\label{sec:experiments}
In this section, we validate the proposed algorithms empirically by applying them to stationary Markov chains derived from both synthetic and real data.
We begin by comparing the runtime of the proposed algorithms and approximation error of \texttt{EntropicOTC} via a simulation study.
Subsequently, we illustrate the potential use of the OTC problem in practice through an application to computer-generated music.

We remark that an application of the OTC problem to graphs is studied in \cite{oconnor2021graph}.
In particular, a weighted graph may be associated with a stationary Markov chain by means of a simple random walk on its nodes with transition probabilities proportional to its edge weights.
Leveraging this perspective, we propose to perform OT on the graphs of interest by applying the OTC problem to their associated Markov chains.
In the aforementioned work, we demonstrate that this approach performs on par with state-of-the-art graph OT methods in a variety of graph comparison and alignment tasks on real and synthetic data.

\texttt{Matlab} implementations of \texttt{ExactOTC} and \texttt{EntropicOTC} as well as code for reproducing the experimental results to follow are available at \url{https://github.com/oconnor-kevin/OTC}.
For \texttt{ApproxOT} and related OT algorithms, we used the implementation found at \url{https://github.com/JasonAltschuler/OptimalTransportNIPS17}.

\begin{figure}[t]
\centering
\begin{subfigure}{0.48\linewidth}
\begin{center}
\includegraphics[width=\linewidth]{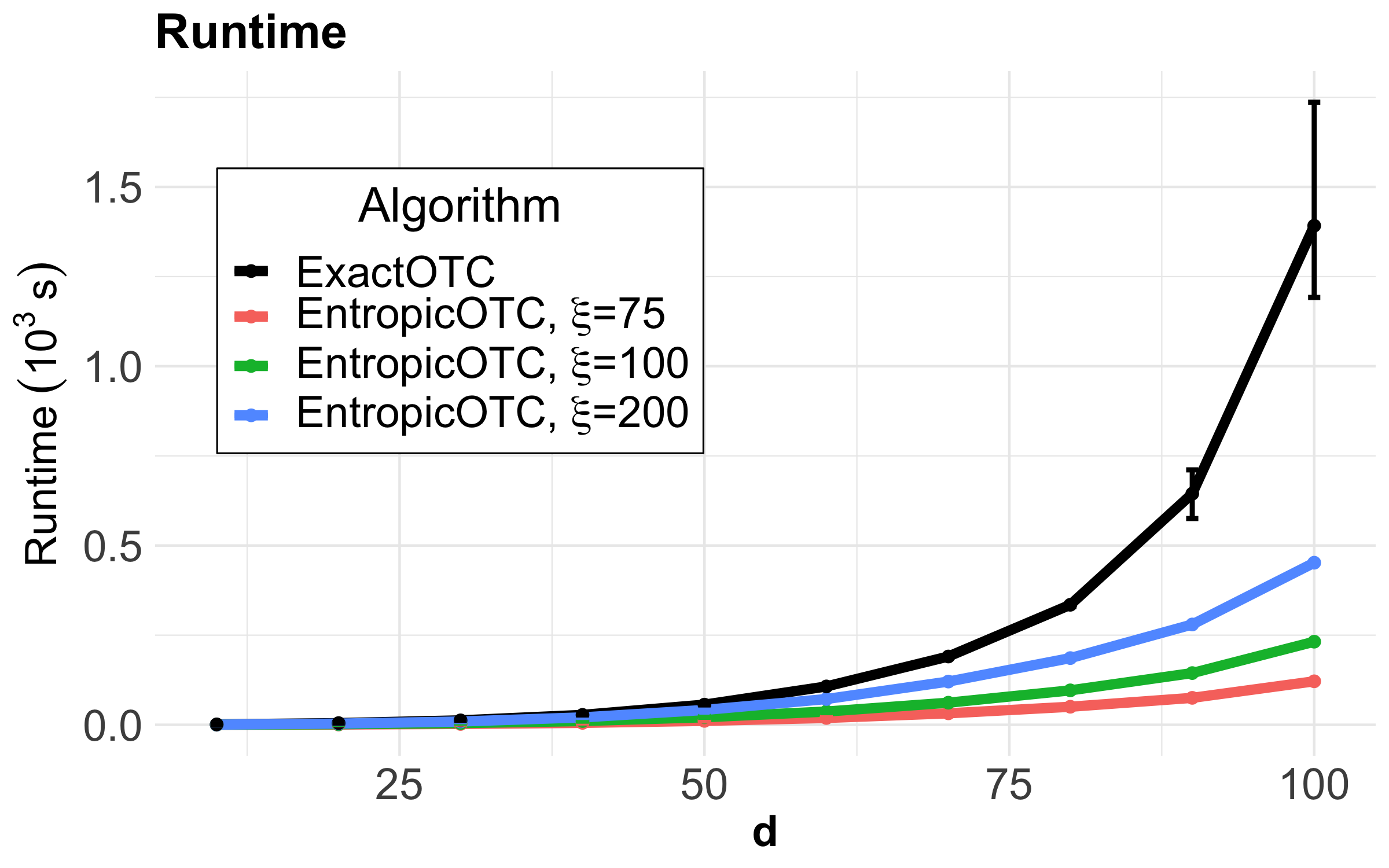}
\end{center}
\end{subfigure}
\begin{subfigure}{0.48\linewidth}
\begin{center}
\includegraphics[width=\linewidth]{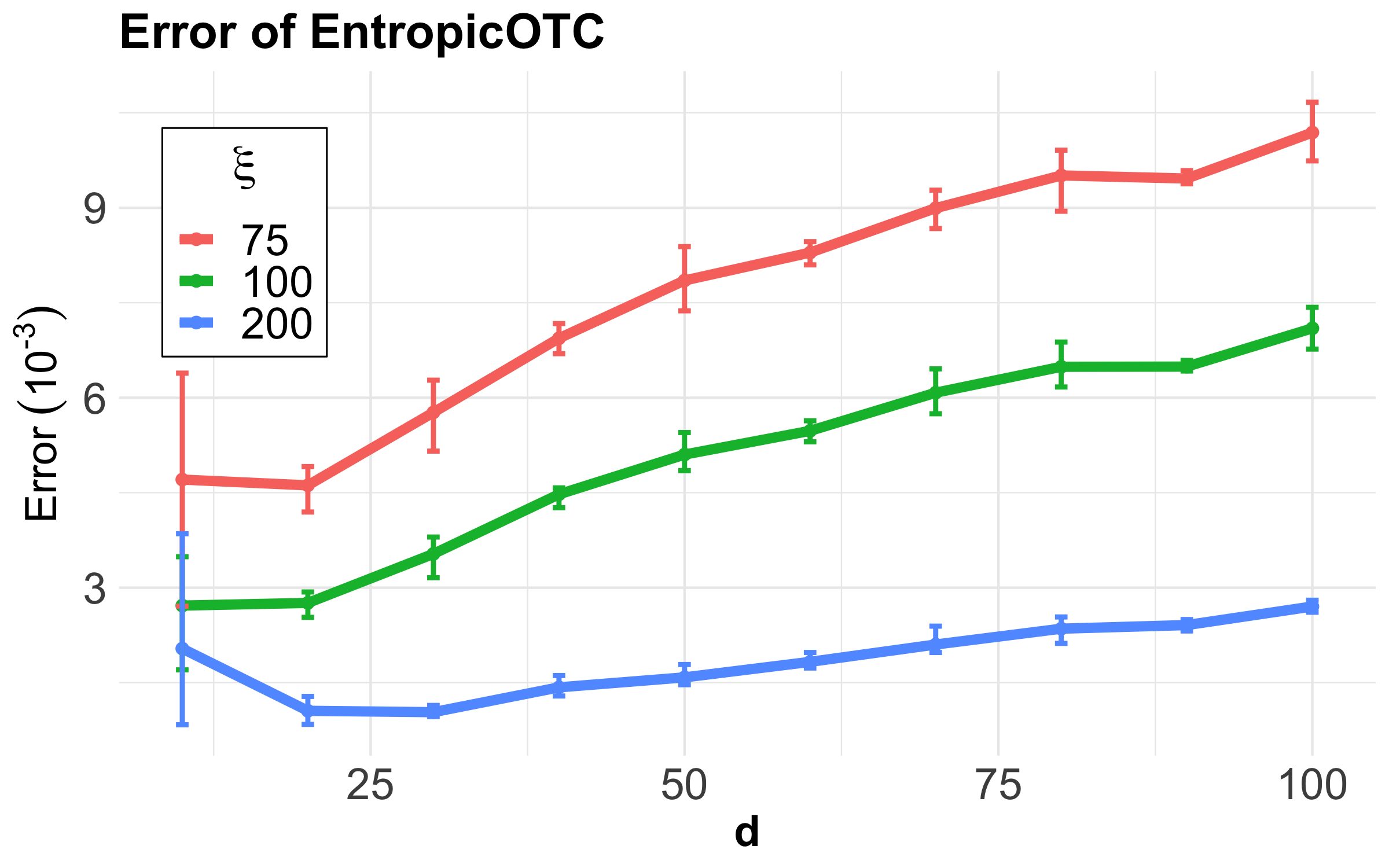}
\end{center}
\end{subfigure}
\caption{A comparison of total runtimes between \texttt{ExactOTC} and \texttt{EntropicOTC} and approximation errors of \texttt{EntropicOTC} for a range of $d$ and $\xi$ via simulation.
Error bars show the minimum and maximum values observed over five simulations.
Note that the error bars for the runtimes of \texttt{EntropicOTC} are not visible because little variation in runtime was observed over the simulations performed.
Runtime is reported in units of $10^3$ seconds while error is reported in units of $10^{-3}$ relative to the maximum value of the cost function $c$.
%Note that the time saved at each iteration from \texttt{FastEntropicOTC} results in faster runtime overall without a substantial loss in accuracy.
}
\label{fig:exp_results}
%\vspace{-5mm}
\end{figure}

\subsection{Simulation Study}
In order to validate the use of Algorithm \ref{alg:fastentropic_pia} as a fast alternative to Algorithm \ref{alg:pia}, we performed a simulation study to compare their runtimes and the error of the entropic OTC cost as an approximation of the OTC cost.
For each choice of the marginal state space size $d \in \{10, 20, ..., 100\}$, we perform five simulations, obtaining estimates of the runtimes and approximation error in each.
%Additionally in each simulation, we compute the runtimes and error of \texttt{EntropicOTC} (with respect to the output of \texttt{ExactOTC}) over a range of regularization coefficients $\xi \in \{75, 100, 200\}$.
In each simulation, we generate transition matrices $P \in [0,1]^{d\times d}$ and $Q \in [0,1]^{d\times d}$ and a cost matrix $c \in \mathbb{R}_+^{d\times d}$ by drawing each element of the matrix of interest independently from a standard normal distribution and then applying an appropriate normalization to the matrix.
In the case of the transition matrices, we apply a softmax normalization with weight 0.1 to each row of $P$ and $Q$:
\begin{equation*}
P(x, x') \mapsto \frac{e^{0.1 P(x,x')}}{\sum_{\tilde{x}} e^{0.1 P(x, \tilde{x})}}, \quad\quad Q(y, y') \mapsto \frac{e^{0.1 Q(y, y')}}{\sum_{\tilde{y}} e^{0.1 Q(y, \tilde{y})}}.
\end{equation*}
For the cost matrix, we apply an absolute value element-wise so that $c \in \mathbb{R}_+^{d\times d}$ and then divide each element by the maximum element in the matrix so that $\|c\|_\infty = 1$.
After generating both the transition matrices and cost matrix, we run both \texttt{ExactOTC} and \texttt{EntropicOTC} for each $\xi \in \{75, 100, 200\}$ until convergence.
%In all experiments, we used $L = 100$ and $T = 1000$ as we found no added benefit to increasing either.
In all runs of \texttt{EntropicOTC}, we choose $L$ and $T$ adaptively as described in Remark \ref{rem:adaptive_LandT} with tolerance ($\varepsilon$) equal to $10^{-12}$ and upper bounds of $100$ and $1000$, respectively.
For each choice of $\xi \in \{75, 100, 200\}$, we use $50$, $100$, and $200$ Sinkhorn iterations, respectively.
Runtimes of \texttt{ExactOTC} and \texttt{EntropicOTC} in a given iteration are measured from the start to convergence and thus correspond to \textit{total} runtime rather than the runtime of individual iterations.
The approximation error of \texttt{EntropicOTC} in a given iteration is measured by taking the absolute difference between the expected cost returned by \texttt{EntropicOTC} and that returned by \texttt{ExactOTC}.
Note that after randomization, the cost function $c$ is scaled to $\|c\|_\infty = 1$ and the error is reported on that scale.

The results of the simulation study are shown in Figure \ref{fig:exp_results}.
The error bars in either plot denote the maximum and minimum values observed for each choice of parameters over the five repeated simulations.
In our simulations, we found that the time savings in each iteration of \texttt{EntropicOTC} resulted in substantial time savings over the entire runtime of the algorithm without substantial loss of accuracy.
For example, when $d = 100$ and $\xi = 100$, we observed that \texttt{EntropicOTC} yielded a time savings of roughly 80\% compared to \texttt{ExactOTC}.
Moreover, weakening the regularization by increasing $\xi$ reduces the error of \texttt{EntropicOTC} with little additional runtime.
This supports our theoretical findings, indicating that \texttt{EntropicOTC} is a good alternative to \texttt{ExactOTC} when $d$ is large.

%\end{center}
%\vspace{-5mm}
%\begin{center}
\begin{figure}[t]
\begin{subfigure}{0.49\textwidth}
\centering
\includegraphics[scale=0.44]{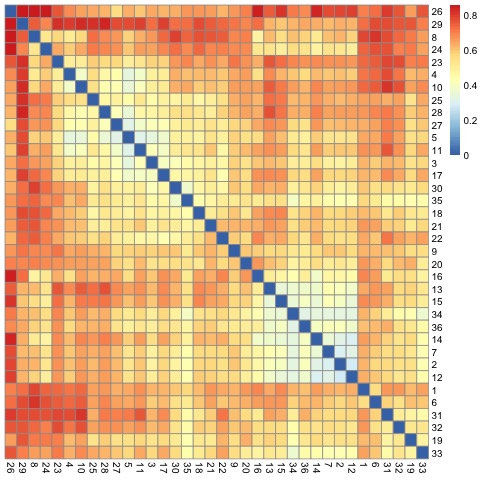}
\caption{\texttt{ExactOTC}}
\end{subfigure}
\hspace{2mm}
\begin{subfigure}{0.49\textwidth}
\centering
\includegraphics[scale=0.44]{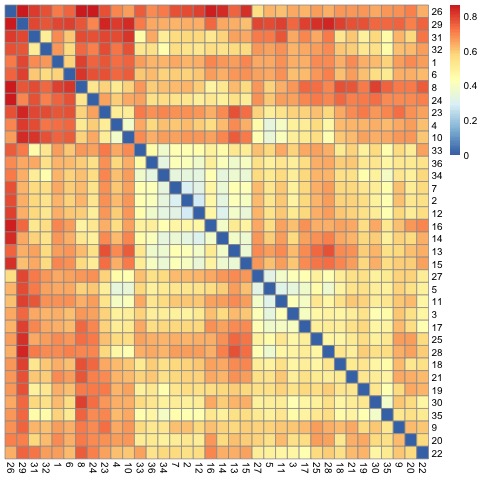}
\caption{\texttt{EntropicOTC}}
\end{subfigure}
\caption{Heatmap of costs for all pairs of pieces as computed by \texttt{ExactOTC} and \texttt{EntropicOTC}.
Lower cost (indicated by blue) indicates a better correspondence between the two pieces.
The list of pieces and composers considered may be found in Table \ref{table:pieces}.}
\label{fig:mds}
\end{figure}
%\end{center}
%\vspace{-15mm}

\subsection{Application to Computer-Generated Music}\label{sec:music_exp}
Next we illustrate the OTC problem in practice through an application to aligning and comparing computer-generated music.
HMMs and other state-space models have been explored as a tool for modeling musical arrangements \cite{ames1989markov, liu2002modeling, weiland2005learning, allan2005harmonising, pikrakis2006classification, ren2010dynamic, bell2011algorithmic, yanchenko2017classical, das2018analyzing}.
In this line of work, sequences of notes are commonly modeled as a stationary processes with latent Markovian structure.
As described in Section \ref{sec:background_on_otc}, the OTC problem easily extends to this setting, allowing one to apply OT methods to analyzing generative models for music.
We utilize the computational tools developed above for two tasks: comparing pieces based on the sequences of notes they contain and generating paired sequences of notes based on existing pieces.

We analyzed a dataset of 36 pieces of classical music from 3 different classical composers (Bach, Beethoven and Mozart) downloaded from \href{}{https://www.mfiles.co.uk/classical-midi.htm}.
The pieces considered along with the composer, musical key, and reference number between 1 and 36 may be found in Table \ref{table:pieces}.
For each piece, a 3-layer HMMs with 5 hidden states was trained using the code provided in \cite{yanchenko2017classical}.
We refer the reader to \cite{yanchenko2017classical} and \cite{oliver2004layered} for details on layered HMMs but note that once a layered HMM is trained it may be recast as a standard HMM and thus the extension of OTC to HMMs described in Section \ref{sec:hmm} still applies.
We considered two different cost functions between notes.
The first cost function equal to 0 if the two notes are equal or some number of octaves (intervals of 12 semitones) apart, and 1 otherwise.
The second cost function is 0 when the first cost function is 0, 1 when the two notes are 5 or 7 semitones apart (perfect consonance), 2 when the two notes are 4 or 9 semitones apart (imperfect consonance) and 10 otherwise.
This tiered cost function incorporates a preference for unison over perfect consonance, perfect consonance over imperfect consonance, and imperfect consonance over dissonance.

%\begin{center}
\begin{figure}[h!]
\centering
\includegraphics[width=\textwidth]{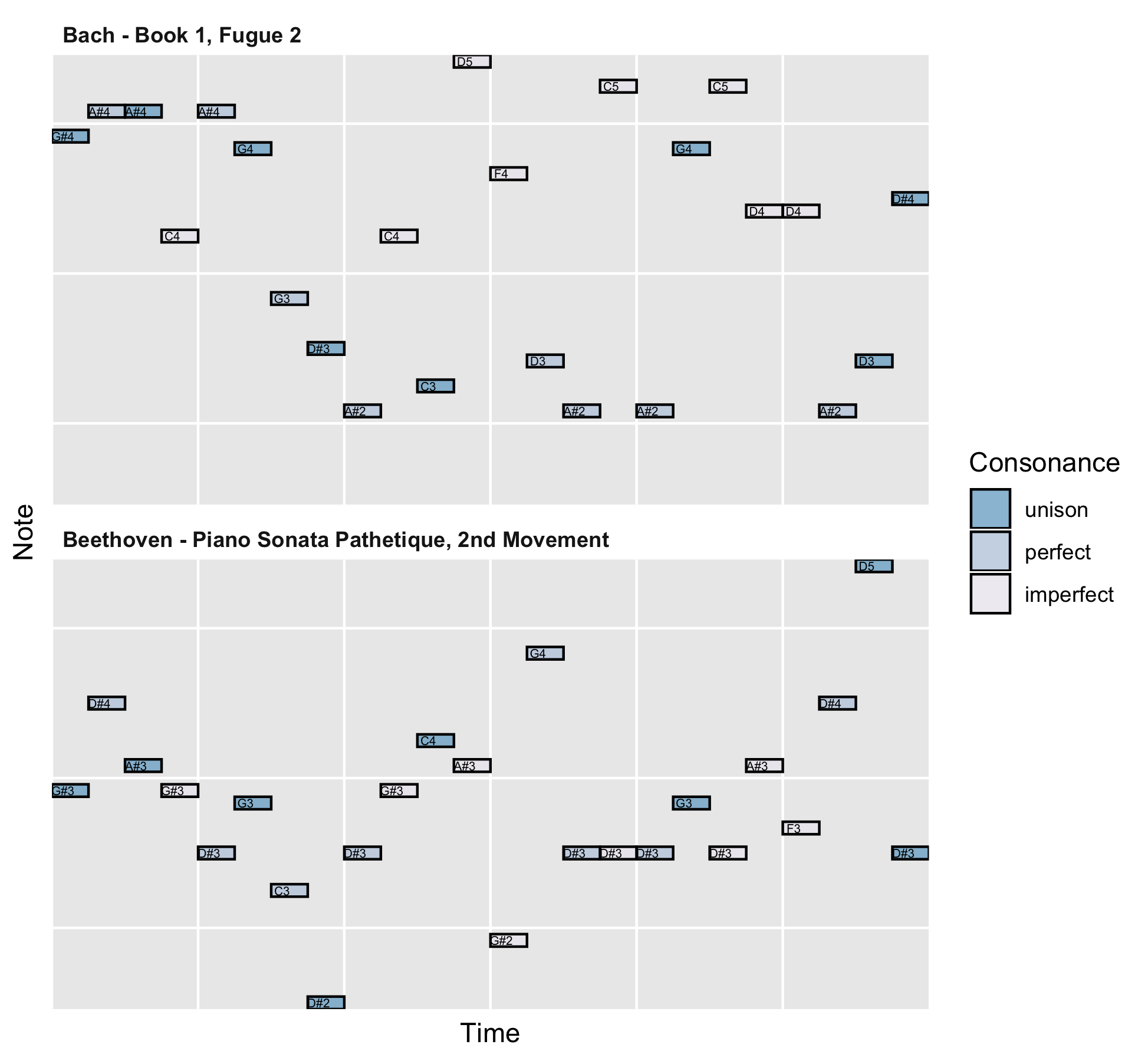}
\caption{An illustration of samples drawn from an optimal transition coupling of Bach's Book 1, Fugue 2 and Beethoven's Sonata Path\'{e}tique, Movement 2, both in C minor.
The color of each sampled note denotes its consonance with the other note played at the same time.}
\label{fig:music_samples}
\end{figure}
%\end{center}
%\vspace{-15mm}

In the first task, we computed the OTC cost for every pair of pieces, obtaining a pairwise cost matrix.
Note that when running \texttt{EntropicOTC}, we use $L = 100$, $T = 1000$, $\xi = 50$, and 20 Sinkhorn iterations.
The cost matrices obtained using \texttt{ExactOTC} and \texttt{EntropicOTC} are both depicted in Figure \ref{fig:mds}.
The correspondence between rows and columns of the two heatmaps and the musical pieces considered can be found in Table \ref{table:pieces}.
We remark that pieces in the same key tended to have lower OTC cost.
For example, Bach's Fugue 2 from Book 1 (2 in Figure \ref{fig:mds}) and Fugue 2 from Book 2 (12 in Figure \ref{fig:mds}), both in C minor, had the lowest OTC and entropic OTC costs among all pairs considered.
We observe that pairwise costs obtained by either algorithm only differ by $8 \times 10^{-3}$ on average.
In other words, \texttt{EntropicOTC} approximates the result of \texttt{ExactOTC} with high accuracy.

In the second task, we explored the samples generated from the optimal transition coupling of each pair of fitted HMMs.
The optimal transition coupling maximizes the probability of generating consonant pairs of notes while preserving the distributions of the two sequences.
This results in sequences that sound harmonious together more frequently.
In Figure \ref{fig:music_samples}, we provide a paired sequence drawn from the output of \texttt{ExactOTC} applied to pieces from Bach and Beethoven.
Note that no dissonant pairs of notes were sampled in this sequence.
Audio files for this sequence and sequences drawn from other pairings may be found in the accompanying supplemental materials.

\section{Discussion}\label{sec:conclusion}

In this paper, we introduced an optimal transport problem for stationary Markov chains that takes the Markovian dynamics into account called the optimal transition coupling (OTC) problem.
Intuitively, the OTC problem aims to synchronize the Markov chains of interest so as to minimize long-term average cost.
We demonstrated how this problem may be easily extended to formulate an OT problem for HMMs.
In the interest of computation, we recast this problem as a Markov decision process and leveraged this connection to prove that solutions can be obtained via an adaptation of the policy iteration algorithm, referred to as \texttt{ExactOTC}.
Mirroring the development of entropic OT in \cite{cuturi2013sinkhorn}, we also proposed an entropic OTC problem and an associated approximate algorithm, \texttt{EntropicOTC}, which scales better with dimension.
For cases when the marginal Markov chains must be estimated from data, we showed that the plug-in estimates for either problem are consistent.
We showed empirically that \texttt{EntropicOTC} approximates the OTC cost with high accuracy and substantially faster runtime than \texttt{ExactOTC} in large state space regimes.
Finally, we illustrated the use of the OTC problem and the proposed algorithms in practice via an application to computer-generated music.

Future work may consider extending the ideas of the OTC problem to processes with more flexible structure such as Gibbs processes or dynamical linear models.
We expect that the extension of our work to processes with richer temporal structure will present interesting computational challenges.
Alternatively, future work may explore further applications of the OTC problem in practice.
Our approach to analyzing computer-generated music may be easily transferred to any data that may be modeled by an HMM.
HMMs and other sequence models with hidden Markov structure are commonly used in a variety of fields including genomics, speech recognition, protein folding, and natural language processing.

%\clearpage 
%
\begin{table}[h!]
\begin{center}
\begin{tabular}{|c|c|c|c|}
\hline 
 & \textbf{Composer} & \textbf{Piece} & \textbf{Key}\\ 
\hline 
1 & Bach & Toccata and Fugue & D minor \\
\hline
2 & Bach & Book 1, Fugue 2 & C minor \\
\hline
3 & Bach & Book 1, Fugue 10 & E minor \\
\hline
4 & Bach & Book 1, Fugue 14 & F\# minor \\
\hline
5 & Bach & Book 1, Fugue 24 & B minor \\
\hline
6 & Bach & Book 1, Prelude 1 & C major \\
\hline
7 & Bach & Book 1, Prelude 2 & C minor \\
\hline
8 & Bach & Book 1, Prelude 3 & C\# major \\
\hline
9 & Bach & Book 1, Prelude 6 & D minor \\
\hline
10 & Bach & Book 1, Prelude 14 & F\# minor \\
\hline
11 & Bach & Book 1, Prelude 24 & B minor \\
\hline
12 & Bach & Book 2, Fugue 2 & C minor \\
\hline
13 & Bach & Book 2, Fugue 7 & D\# major \\
\hline
14 & Bach & Book 2, Prelude 2 & C minor \\
\hline
15 & Bach & Book 2, Prelude 7 & D\# major \\
\hline
16 & Bach & Book 2, Prelude 12 & F minor \\
\hline 
17 & Bach & Bourr\'{e}e in E minor & E minor \\
\hline
18 & Bach & 2 Part Invention, No. 13 & A minor \\
\hline
19 & Bach & 2 Part Invention, No. 4 & D minor \\
\hline
20 & Bach & Prelude in C major & C major \\
\hline
21 & Beethoven & F\"{u}r Elise & A minor \\
\hline
22 & Beethoven & Minuet in G & G major \\
\hline
23 & Beethoven & Moonlight Sonata, Movement 1 & C\# minor \\
\hline
24 & Beethoven & Sonata Path\'{e}tique, Movement 2 & C minor \\
\hline
25 & Beethoven & Symphony No. 7, Movement 2 & A minor \\
\hline
26 & Beethoven & Symphony No. 9, Movement 4 & D minor \\
\hline
27 & Beethoven & Violin Sonata 1, Movement 1 & D major \\
\hline
28 & Mozart & Piano Sonata No. 11, Movement 3 & A major \\
\hline
29 & Mozart & Horn Concerto 4, Movement 3 & D\# major \\
\hline
30 & Mozart & Minuet and Trio, K.1 & G major \\
\hline
31 & Mozart & Minuet in F major, K.2 & F major \\
\hline
32 & Mozart & \"{O}sterreichische Bundeshymne & D\# major \\
\hline
33 & Mozart & Piano Concerto No. 21, Movement 2 & C major \\
\hline
34 & Mozart & Piano Sonata No. 13, Movement 1 & A\# major \\
\hline
35 & Mozart & Piano Sonata No. 16 & C major \\
\hline
36 & Mozart & Symphony No. 40, Movement 1 & G minor \\
\hline
\end{tabular} 
\end{center}
\caption{Pieces considered in the application of OTC to computer-generated music.}
\label{table:pieces}
\end{table}

\section{Proofs}
\label{sec:proofs}
\subsection{Overview of Proofs}
\label{sec:overview_of_proofs}
In what follows, we detail the proofs of our results.
We begin by introducing some additional notation, covering some preliminaries on Markov chains, and remarking on some technical aspects relating to our results.

\subsubsection{Additional notation}
We adopt the following additional notation:
For a finite set $\calU \subset \bbR$, we define $\minplus \calU = \min\{u \in \calU : u > 0\}$.
We define the inner product $\langle \cdot, \cdot \rangle$ for matrices $U, V \in \bbR^{n\times n}$ by $\langle U, V \rangle := \sum_{i,j} U_{ij} V_{ij}$.
All vector and matrix equations and inequalities should be understood to hold element-wise.
For $i \leq j$, we let $u_i^j = (u_i, ..., u_j)$ and we will denote infinite sequences by boldface, lowercase letters such as $\bfu = (u_0, u_1, ...)$.
For a collection of sets $\calU_{s} \subset \bbR^{d^2}$ indexed by $s \in \X \times \Y$, we define $\bigotimes_{s} \calU_s$ to be the set of matrices $U \in \bbR^{d^2\times d^2}$ such that for every $s \in \X \times \Y$, $U(s,\cdot) \in \calU_{s}$.
In particular, we write $\tc = \bigotimes_{(x,y)} \tcrow$.

\subsubsection{Preliminaries on Markov chains}
For a finite metric space $\calU$, we say that a measure $\mu \in \calM(\calU^\bbN)$ is \emph{Markov} or \emph{corresponds to a Markov chain taking values in $\calU$} if for any cylinder set $[u_0 \cdots u_k] \subset \calU^\bbN$, $\mu([u_0 \cdots u_k]) / \mu([u_0 \cdots u_{k-1}]) = \mu([u_{k-1} u_k]) / \mu([u_{k-1}])$, where we let $\nicefrac{0}{0} = 0$.
We say that $\mu$ is \emph{stationary} if $\mu = \mu \circ \sigma^{-1}$, where $\sigma: \calU^\bbN \rightarrow \calU^\bbN$ is the left-shift map defined such that for any $\bfu \in \calU^\bbN$, $\sigma(\bfu)_i = u_{i+1}$.
When $\calU$ has cardinality $n \geq 1$, we define the transition matrix $U \in \bbR^{n\times n}$ of $\mu$ such that for every $u_{k-1}, u_k \in \calU$, $U(u_{k-1}, u_k) = \mu([u_{k-1} u_k]) / \mu([u_{k-1}])$.
If $\mu$ is also stationary, its stationary distribution $\lambda_U \in \Delta_n$ is defined such that $\lambda_U(u) = \mu([u])$ for any $u \in \calU$.
We say that $\mu$ or $U$ is \emph{irreducible} if for every $u, u' \in \calU$, there exists $k \geq 1$, possibly depending on $u$ and $u'$, such that $U^k(u,u') > 0$.
We call $\mu$ or $U$ \emph{aperiodic} if $\gcd\{k \geq 1: U^t(u,u') > 0\} = 1$ for every $u, u' \in \calU$.
Note that if $\mu$ is irreducible, its stationary distribution $\lambda_U$ is unique.
Furthermore, if $\mu$ is also aperiodic, there exists $M < \infty$ and $\alpha \in (0, 1)$ such that for any $t \geq 1$, $\max_u \|U^t(u, \cdot) - \lambda_U\|_1 \leq M\alpha^t$.
For more details on basic Markov chain theory, we refer the reader to \cite{levin2017markov}.

\subsubsection{Technical considerations}
We endow the finite set $\X\times\Y$ with the discrete topology and $\X^\bbN\times\Y^\bbN$ with the corresponding product topology.
For each $(x,y) \in \X\times\Y$ and $\eta > 0$, we endow both $\tcrow$ and $\tcetarow$ with the subspace topology inherited from the Euclidean topology on $\bbR^{d^2}$.
Similarly, we endow $\tc$ and $\tceta$ with the subspace topologies inherited from the Euclidean topology on $\bbR^{d^2\times d^2}$.
Unless stated otherwise, continuity of any function will be understood to mean with respect to the corresponding topology above.

\subsection{Proofs from Section \ref{sec:background_on_otc}}\label{app:finite_dim}
\transmatchar*
%\begin{proof}[Proof of Proposition \ref{prop:transmat_char}]
\begin{proof}
Let $\pi \in \calM((\X \times \Y)^\bbN)$ be the distribution of a stationary Markov chain with transition matrix $R \in \tc$ and stationary distribution $r \in \Delta_{d^2}$.
Furthermore, let $r_{\X}$ and $r_{\Y} \in \Delta_d$ be the $\X$ and $\Y$ marginals of $r$, respectively.
For a metric space $\calU$ and a probability measure $\mu \in \calM(\calU^\bbN)$, we define $\mu_k \in \calM(\calU^k)$ as the $k$-dimensional marginal distribution of $\mu$.
Formally, for any cylinder set $[a_0^{k-1}] = \{\bfu \in \calU^\bbN: u_j = a_j, 0 \leq j \leq k-1\}$, $\mu_k(a_0^{k-1}) := \mu([a_0^{k-1}])$.

We wish to show that $\pi \in \tcrv$.
Since $\pi$ corresponds to a stationary Markov chain and $R \in \tc$ by assumption, it suffices to show that $\pi \in \Pi(\bbP, \bbQ)$.
We will do this by showing that $\pi_k \in \Pi(\bbP_k, \bbQ_k)$ for every $k \geq 1$.
Starting with $k = 1$, for any $y \in \Y$,
\begin{align*}
r_{\Y}(y) &= \sum\limits_x r(x, y) \\
&= \sum\limits_x \sum\limits_{x', y'} r(x', y') R((x', y'), (x, y)) \\
&= \sum\limits_{x', y'} r(x', y') \sum\limits_x R((x', y'), (x, y)) \\
&= \sum\limits_{x',y'} r(x',y') Q(y', y) \\
&= \sum\limits_{y'} r_{\Y}(y') Q(y', y).
\end{align*}
We have proven that $r_{\Y}$ is invariant with respect to $Q$.
Since $Q$ is irreducible, the stationary distribution $q$ of $Q$ is unique.
Thus, $r_{\Y} = q$.
A similar argument will show that $r_{\X} = p$.
Thus, $r \in \Pi(p, q)$ and therefore, $\pi_1 \in \Pi(\bbP_1, \bbQ_1)$.

Now suppose that $\pi_k \in \Pi(\bbP_k, \bbQ_k)$ for some $k \geq 1$.
Fixing $y_0^k \in \Y^{k+1}$, it follows that
\begin{align*}
\sum\limits_{x_0^k} \pi_{k+1}(x_0^k, y_0^k) &= \sum\limits_{x_0^k} \pi_k(x_0^{k-1}, y_0^{k-1}) R((x_{k-1}, y_{k-1}), (x_k, y_k)) \\
&= \sum\limits_{x_0^{k-1}} \pi_k(x_0^{k-1}, y_0^{k-1}) Q(y_{k-1}, y_k) \\
&= \bbQ_k(y_0^{k-1}) Q(y_{k-1}, y_k) \\
&= \bbQ_{k+1}(y_0^k).
\end{align*}
Again the proof for the other marginal is identical.
So we find that $\pi_{k+1} \in \Pi(\bbP_{k+1}, \bbQ_{k+1})$ and since $k\geq 1$ was arbitrary, we conclude that $\pi \in \pitc(\bbP, \bbQ)$.
\end{proof}

\subsection{Proofs from Section \ref{sec:computing}}
\subsubsection{Existence of a deterministic policy}
\deterministicpolicy*
%\begin{proof}[Proof of Proposition \ref{prop:deterministic_policy}]
\begin{proof}
Before proving the result, it will be helpful to fix some additional notation.
Let $\gamma = \{\gamma_s(\cdot): s \in \X\times\Y\}$ be a policy for TC-MDP.
Recall that for each $s = (x,y)$, $\gamma_s(\cdot)$ describes a distribution on $\calA_s = \tcrow$.
Define the deterministic policy $\tilde{\gamma} = \{\tilde{\gamma}_s(\cdot): s \in \X\times\Y\}$ such that for every $s$, $\tilde{\gamma}_s(\cdot)$ assigns probability one to
\begin{equation*}
\tilde{r}_s := \int_{\calA_s} r_s \gamma_s(dr_s).
\end{equation*}
Here, $\tilde{r}_s$ is the expected action taken by the agent while occupying a state $s$ and following the policy $\gamma$.
Note that $\tilde{r}_s \in \calA_s$ due to the convexity of $\calA_s$.
As such, we may collect the row vectors $\{\tilde{r}_s: s \in \X\times\Y\}$ into a single transition matrix $\tilde{R} \in \tc$ where $\tilde{R}(s,\cdot) = \tilde{r}_s(\cdot)$ for every $s \in \X\times\Y$.
In what follows, let $\prob_\gamma(\cdot | s_0)$ and $\prob_{\tilde{\gamma}}(\cdot | s_0) \in \calM(\{\calA\times(\X\times\Y)\}^\bbN)$ be the probability measures corresponding to the action-state processes with initial state $s_0$ induced by $\gamma$ and $\tilde{\gamma}$, respectively.
In particular,
\begin{equation*}
\prob_\gamma(dr_{s_0}, s_1,..., dr_{s_{t-1}}, s_t | s_0) = \gamma_{s_0}(dr_{s_0}) r_{s_0}(s_1) \cdots \gamma_{s_{t-1}}(dr_{s_{t-1}}) r_{s_{t-1}}(s_t)
\end{equation*}
and the analogous statement holds for $\prob_{\tilde{\gamma}}(\cdot | s_0)$.
In the case of $\tilde{\gamma}$, one may also show that $\prob_{\tilde{\gamma}}(s_t | s_0) = \tilde{R}^t(s_0, s_t)$.
Finally, let $\mathbb{E}_\gamma[\cdot | s_0]$ and $\mathbb{E}_{\tilde{\gamma}}[\cdot | s_0]$ denote expectation with respect to $\prob_\gamma(\cdot | s_0)$ and $\prob_{\tilde{\gamma}}(\cdot | s_0)$, respectively.

Now, we can prove the result.
For any $s_0 \in \X\times\Y$ and $t \geq 1$,
\begin{align*}
\mathbb{E}_\gamma\left[c(s_t) | s_0\right] &= \sum\limits_{s_t} c(s_t) \prob_\gamma(s_t | s_0) \\
&= \sum\limits_{s_t} c(s_t) \int_{\calA_{s_0}} \sum\limits_{s_1} \cdots \int_{\calA_{s_{t-1}}} \prob_{\gamma}(dr_{s_0}, s_1, ..., dr_{s_{t-1}}, s_t | s_0) \\
&= \sum\limits_{s_t} c(s_t) \int_{\calA_{s_0}} \sum\limits_{s_1} \cdots \int_{\calA_{s_{t-1}}} \gamma_{s_0}(dr_{s_0}) \, r_{s_0}(s_1) \cdots \gamma_{s_{t-1}}(dr_{s_{t-1}}) \, r_{s_{t-1}}(s_t) \\
&= \sum\limits_{s_1^t} c(s_t) \int_{\calA_{s_0}} \cdots \int_{\calA_{s_{t-1}}} \gamma_{s_0}(dr_s) \, r_{s_0}(s_1) \cdots \gamma_{s_{t-1}}(dr_{s_{t-1}}) \, r_{s_{t-1}}(s_t) \\
&= \sum\limits_{s_1^t} c(s_t) \tilde{r}_{s_0}(s_1) \cdots \tilde{r}_{s_{t-1}}(s_t) \\
&= \sum\limits_{s_1^t} c(s_t) \tilde{R}(s_0,s_1) \cdots \tilde{R}(s_{t-1}, s_t) \\
&= \sum\limits_{s_t} c(s_t) \tilde{R}^t(s_0,s_t) \\
&= \sum\limits_{s_t} c(s_t) \prob_{\tilde{\gamma}}(s_t | s_0) \\
&= \mathbb{E}_{\tilde{\gamma}}\left[c(s_t) | s_0\right].
\end{align*}
Thus, for every $s \in \X\times\Y$,
\begin{equation*}
\overline{c}_\gamma(s) = \lim\limits_{T\rightarrow\infty} \frac{1}{T} \sum\limits_{t=1}^T \mathbb{E}_\gamma\left[c(s_t) | s_0 = s\right] = \lim\limits_{T\rightarrow\infty} \frac{1}{T} \sum\limits_{t=1}^T \mathbb{E}_{\tilde{\gamma}}\left[c(s_t) | s_0=s\right] = \overline{c}_{\tilde{\gamma}}(s).
\end{equation*}
\end{proof}

\subsubsection{Correspondence between TC-MDP and the OTC problem}\label{app:otc_is_tcmdp}
Next, we prove Proposition \ref{prop:otc_is_mdp} showing that optimal solutions to TC-MDP necessarily provide optimal solutions to the OTC problem.
We rely on the basic idea of recurrent classes of states for finite-state Markov chains.
For details on recurrence for Markov chains, we refer the reader to \cite{levin2017markov}.
For any $R \in \tc$, let $\Lambda(R) := \{ \lambda \in \calM(\X): \lambda R = \lambda\}$ denote the set of stationary distributions for $R$ and let $\bigsqcup$ denote a disjoint union.
Before proving the proposition, we require a lemma stating that for a given transition coupling matrix $R \in \tc$, the stationary distribution of $R$ that incurs the least expected cost may be chosen to be the unique stationary distribution of one of $R$'s recurrent classes.

\begin{lem}\label{lemma:extreme_points_of_stat_dist}
Let $R \in \tc$ and let $\calS_r$ be the set of states belonging to some recurrent class of $R$.
Moreover, for every $s \in \calS_r$, let $\lambda_{R,s} \in \Lambda(R)$ denote the stationary distribution of $R$ corresponding to the recurrent class in which $s$ lies.
Then $\lambda_{R, s}$ is uniquely defined and 
\begin{equation*}
\min\limits_{s \in \calS_r} \langle c, \lambda_{R,s} \rangle = \min\limits_{\lambda \in \Lambda(R)} \langle c, \lambda \rangle.
\end{equation*}
\end{lem}
\begin{proof}
The uniqueness of $\lambda_{R,s}$ follows from the fact that the chain obtained by restricting $R$ to the recurrent class of $s$ is necessarily irreducible.
Now suppose that $R$ has $m$ recurrent classes $\{S_r^i\}_{i=1}^m$ and thus $\calS_r = \bigsqcup_{i=1}^m S_r^i$.
Then by \cite[Theorem A.5]{puterman2005markov}, there exist $m$ linearly independent stationary distributions of $R$.
Note that necessarily, the unique stationary distributions $\{\lambda_i\}_{i=1}^m$ corresponding to the $m$ recurrent classes of $R$ are linearly independent and constitute such a choice.
Moreover, it is straightforward to show that $\Lambda(R)$ is equal to the convex hull of $\{\lambda_i\}_{i=1}^m$ and is thus compact.
Then since minima of linear functions over a compact, convex set occur at the extreme points of the feasible set, 
\begin{equation*}
\min\limits_{s \in \calS_r} \langle c, \lambda_{R, s} \rangle = \min\limits_{i=1, ..., m} \langle c, \lambda_i \rangle = \min\limits_{\lambda \in \Lambda(R)} \langle c, \lambda \rangle.
\end{equation*}
\end{proof}

\otcismdp*
\begin{proof}
For every $R \in \tc$ and $s \in \calS$, let $\lambda_{R,s} \in \Lambda(R)$ be the stationary distribution of $R$ defined by
\begin{equation*}
\lambda_{R,s} := \lim\limits_{T \rightarrow\infty} \frac{1}{T} \sum\limits_{t=1}^T R^t(s,\cdot).
\end{equation*}
Note that $\lambda_{R,s}$ is well-defined by \cite[Theorem A.5]{puterman2005markov}.
Moreover, we will use $\calS_r(R)$ to refer to the set of all states in $\calS$ that belong to a recurrent class of $R$.
Since the space $\calS$ is finite, $\calS_r(R)$ is necessarily non-empty for every $R \in \tc$.
Finally, note that whenever $s \in \calS_r(R)$, $\lambda_{R,s}$ is the unique stationary distribution of $R$ associated with the recurrent class in which $s$ lies.

Now let $R^* \in \tc$ be optimal for TC-MDP.
We will construct a transition coupling $\pi_* \in \tcrv$ from $R^*$ that is optimal in the OTC problem.
Note that by definition, $\overline{c}_R(s) = \langle c, \lambda_{R,s} \rangle$.
Then by the optimality of $R^*$ in TC-MDP, $\langle c, \lambda_{R^*, s} \rangle = \min_{R \in \tc} \langle c, \lambda_{R,s} \rangle$ for every $s \in \calS$.
So by Lemma \ref{lemma:extreme_points_of_stat_dist},
\begin{equation}\label{eq:tcmdp_solution_is_opt}
\min\limits_{s \in \calS_r} \langle c, \lambda_{R^*, s} \rangle = \min\limits_{s \in \calS} \langle c, \lambda_{R^*, s} \rangle = \min\limits_{R \in \tc} \min\limits_{s \in \calS} \langle c, \lambda_{R, s} \rangle = \min\limits_{R \in \tc} \min\limits_{\lambda \in \Lambda(R)} \langle c, \lambda \rangle.
\end{equation}
Let $s^* \in \argmin_{s \in \calS_r} \langle c, \lambda_{R^*, s} \rangle$ and define $\pi_* \in \tcrv$ to be the transition coupling with transition matrix $R^*$ and stationary distribution $\lambda_{s^*}$.
Then by \eqref{eq:tcmdp_solution_is_opt},
\begin{equation*}
\int c \, d\pi_* = \langle c, \lambda_{R^*, s^*} \rangle = \min\limits_{R \in \tc} \min\limits_{\lambda \in \Lambda(R)} \langle c, \lambda \rangle.
\end{equation*}
But at this point, we recognize that the quantity on the right is exactly the OTC cost.
To see this, note that by Proposition \ref{prop:transmat_char} every $\pi \in \tcrv$ is uniquely characterized by a transition matrix $R \in \tc$ and a stationary distribution $\lambda \in \Lambda(R)$, and $\int c \, d\pi = \langle c, \lambda \rangle$.
Thus
\begin{equation*}
\int c \, d\pi_* = \min\limits_{R \in \tc} \min\limits_{\lambda \in \Lambda(R)} \langle c, \lambda \rangle = \min\limits_{\pi \in \tcrv} \int c \, d\pi,
\end{equation*}
and we conclude that $\pi_*$ is optimal for the OTC problem.
Finally, by construction, $\int c \, d\pi_* = \min_{s \in \calS} \langle c, \lambda_{R^*, s} \rangle = \min_{s \in \calS} \overline{c}_{R^*}(s)$.
\end{proof}

%% Convergence of ExactOTC
\subsubsection{Convergence of \texttt{ExactOTC}}\label{app:pia_convergence}
Next, we prove the convergence of Algorithm \ref{alg:pia} to a solution of TC-MDP.
For any polyhedron $\calP \in \bbR^{n\times n}$, let $\calE(\calP)$ denote the extreme points of $\calP$.
Recall that if $\calP$ is bounded, a linear function on $\calP$ achieves its minimum on $\calE(\calP)$ \citep{bertsimas1997introduction}.
Note that for every $(x,y) \in \X\times\Y$, since $\tcrow$ is a bounded subset of $\bbR^{d^2}$ defined by a finite set of linear equality and inequality constraints, it is a bounded polyhedron.

\convergenceofpia*
%\begin{proof}[Proof of Theorem \ref{thm:convergence_of_pia}]
\begin{proof}
We will first show that Algorithm \ref{alg:pia} converges to some $(g^*, h^*, R^*)$ and then argue that this is a solution to TC-MDP.
Recall that for every $s=(x,y)$, $\calA_{s} = \tcrow$ and $\calA = \bigcup_s \calA_s$.
In this proof, it is most convenient to consider the concatenatation of the state-action spaces instead of the union $\bigcup_s \calA_s$.
Abusing notation, we let $\calA = \bigotimes_s \calA_s$ for the remainder of the proof.
Furthermore, let $\calA'_s = \calE(\calA_s)$ be the set of extreme points of $\calA_s$.
As $\calA_s$ is a bounded polyhedron, $\calA'_s$ is finite.
For every $n \geq 1$, let $(g_n, h_n, R_n)$ be the $n$'th iterate of Algorithm \ref{alg:pia}.
Since the rows of $R_n$ are solutions of the linear programs in Algorithm \ref{alg:exact_pi},$R_n(s,\cdot) \in \calE(\calA'_s)$ for every $s$.
Thus the iterates of Algorithm \ref{alg:pia} are the same as the iterates of the policy iteration algorithm for the restricted MDP $(\X\times\Y, \bigcup_s \calA'_s, \{p(\cdot | s, a)\}, c)$ constructed by restricting the state-action spaces $\calA_s$ of TC-MDP to $\calA'_s$ for each $s$.
Since $\calA'_s$ is finite for every $s$, standard results \citep[Theorem 9.2.3]{puterman2005markov} ensure that the iterates $\{(g_n, h_n, R_n)\}$ of Algorithm \ref{alg:pia} will converge to a solution $(g^*, h^*, R^*)$ in a finite number of iterations.
Thus, we need only show that any stationary point of Algorithm \ref{alg:pia} is necessarily a solution to TC-MDP.

Let $(g^*, h^*, R^*)$ be a stationary point of Algorithm \ref{alg:pia}.
Then $R^* = \texttt{ExactTCI}(g^*, h^*, R^*, $ $\bigotimes_s \calA'_s)$ and consequently, $R^*(s,\cdot) \in \argmin_{r \in \calA'_s} rh^*$ for every $s$.
Since $\calA_s$ is a bounded polyhedron, $\min_{r \in \calA_s} rh^*  = \min_{r \in \calA'_s} rh^*$ and we find that $R^*(s,\cdot) \in \argmin_{r \in \calA_s} rh^*$.
Since $\calA = \bigotimes_s \calA_s$, we may write $R^* \in \argmin_{R \in \calA} Rh^*$ where the minimum is understood to be element-wise.
Using the assumption that $(g^*, h^*, R^*)$ is a stationary point of Algorithm \ref{alg:pia} again, $(g^*, h^*) = \texttt{ExactTCE}(R^*)$.
It follows that 
\begin{equation}\label{eq:rstar_eq}
g^* + h^* = R^* h^* + c.
\end{equation}
Since $R^* \in \argmin_{R \in \calA} R h^*$, we obtain
\begin{equation*}
g^* + h^* = \min\limits_{R \in \calA} Rh^* + c.
\end{equation*}
Then by \cite[Theorem 9.1.2 (c)]{puterman2005markov}, $g^*$ is the optimal expected cost for TC-MDP.
Moreover, by \eqref{eq:rstar_eq} and \cite[Theorem 8.2.6 (b)]{puterman2005markov}, $g^* = \overline{R}^*\!\! c = \overline{c}_{R^*}$, where we remind the reader that $\overline{R}^* = \lim_{T\rightarrow\infty} \nicefrac{1}{T} \sum_{t=0}^{T-1} R^{*t}$.
Thus $R^*$ has optimal expected cost among policies for TC-MDP and we conclude that $(g^*, h^*, R^*)$ is a solution to TC-MDP.

If $X$ and $Y$ are irreducible, then by Proposition \ref{prop:transmat_char}, every transition coupling matrix in $\tc$ induces a transition coupling in $\tcrv$.
Since $R^*$ has minimal expected cost over all elements of $\tc$, it attains the minimum in Problem \eqref{eq:otc_problem_rv} and is thus an optimal transition coupling.
\end{proof}

\subsection{Proofs from Section \ref{sec:fast_approx_pia}}\label{app:complexity}
\subsubsection{Complexity of approximate transition coupling evaluation}\label{sec:complexity_of_fastpe}
\policyevaluationcomplexity*
%\begin{proof}[Proof of Lemma \ref{lemma:policy_evaluation_complexity}]
\begin{proof}
Briefly, we remind the reader that $g = \overline{R} c$ and $h = \sum_{t=0}^\infty R^t (c - g)$, and that for integers $L, T \geq 1$ to be chosen later,
\begin{equation*}
\tilde{g} = \langle \nicefrac{1}{d^2} R^L c, \mathbbm{1}\rangle \mathbbm{1}\quad \mbox{and} \quad \tilde{h} = \sum_{t=0}^T R^t (c - \tilde{g}).
\end{equation*}
Note that the expression for $\tilde{g}$ may also be written as
\begin{equation*}
\tilde{g} = \left(\frac{1}{d^2} \sum\limits_s R^L(s, \cdot) c\right) \mathbbm{1}.
\end{equation*}
We begin by studying the approximation error for $\tilde{h}$ by first considering the intermediate quantity $h' := \sum_{t=0}^T R^t(c - g)$.
By the triangle inequality,
\begin{equation}\label{eq:htilde_triangle_ineq}
\|\tilde{h} - h\|_1 \leq \|\tilde{h} - h'\|_1 + \|h' - h\|_1,
\end{equation}
so it suffices to control the two terms on the right hand side.
Using H\"older's inequality, it follows that
\begin{align*}
\|\tilde{h} - h'\|_1 &= \left\|\sum\limits_{t=0}^T R^t(\tilde{g} - g)\right\|_1 \\
& \leq \sum\limits_{t=0}^T \left\|R^t (\tilde{g} - g)\right\|_1 \\
& \leq d^2 \sum\limits_{t=0}^T \max\limits_s \left|R^t(s, \cdot)(\tilde{g} - g)\right| \\
&\stackrel{(*)}{\leq} d^2 \sum\limits_{t=0}^T \|\tilde{g} - g\|_\infty \\
&= (T + 1) d^2 \|\tilde{g} - g\|_\infty,
\end{align*}
where (*) uses the fact that $\|R^t(s,\cdot)\|_1 = 1$ for every $t \geq 1$ and $s \in \X\times\Y$.
Next we wish to bound $\|h' - h\|_1$.
Since $R^t \overline{R} = \overline{R}$ for any $t \geq 1$, we may write $h$ and $h'$ as
\begin{equation*}
h = \sum\limits_{t=0}^\infty (R^t - \overline{R}) c \quad \mbox{and} \quad h' = \sum\limits_{t=0}^T (R^t - \overline{R}) c.
\end{equation*}
Moreover, since $R$ is aperiodic and irreducible, the Perron-Frobenius theorem implies that $\overline{R}(s, \cdot) = \lambda_R$ for every $s \in \X\times\Y$, where $\lambda_R \in \Delta_{d^2}$ is the unique stationary distribution of $R$.
Now by H\"older's inequality and the mixing assumption on $R$,
\begin{align*}
\|h' - h\|_1 &= \left\|\sum\limits_{t=T+1}^\infty (R^t - \overline{R}) c\right\|_1 \\
& \leq \sum\limits_{t=T+1}^\infty \|(R^t - \overline{R}) c\|_1 \\
& \leq d^2 \sum\limits_{t=T+1}^\infty \max\limits_s \left|(R^t(s,\cdot) - \lambda_R) c\right| \\
& \leq \|c\|_\infty d^2 \sum\limits_{t=T+1}^\infty \max\limits_s \left\| R^t(s,\cdot) - \lambda_R\right\|_1 \\
& \leq \|c\|_\infty d^2 \sum\limits_{t=T+1}^\infty M \alpha^t \\
& = M \|c\|_\infty \frac{\alpha^{T+1}}{1 - \alpha} d^2.
\end{align*}
Thus by \eqref{eq:htilde_triangle_ineq},
\begin{equation}\label{eq:htilde_triangle_ineq2}
\|\tilde{h} - h\|_1 \leq (T + 1) \|\tilde{g} - g\|_\infty d^2 + M \|c\|_\infty \frac{\alpha^{T+1}}{1 - \alpha} d^2.
\end{equation}
So in order to bound $\|\tilde{h} - h\|_1$, we require a bound on $\|\tilde{g} - g\|_\infty$.
Using the fact that $\tilde{g}$ and $g$ are constant vectors, H\"older's inequality and the mixing assumption on $R$,
\begin{align*}
\|\tilde{g} - g\|_\infty &= \left\|\left(\frac{1}{d^2} \sum\limits_s R^L(s,\cdot) c\right) \mathbbm{1} - \overline{R} c\right\|_\infty \\
&= \left|\frac{1}{d^2} \sum\limits_s R^L(s,\cdot) c - \lambda_R c\right| \\
&\leq \frac{1}{d^2} \sum\limits_s \left|(R^L(s, \cdot) - \lambda_R) c \right| \\
&\leq \frac{1}{d^2} \sum\limits_s \|c\|_\infty \|R^L(s,\cdot) - \lambda_R\|_1 \\
& \leq \frac{1}{d^2} \sum\limits_s M\alpha^L \|c\|_\infty \\
& \leq M \alpha^L \|c\|_\infty.
\end{align*}
Plugging this into \eqref{eq:htilde_triangle_ineq2},
\begin{equation*}
\|\tilde{h} - h\|_1 \leq M \alpha^L \|c\|_\infty (T + 1) d^2 + M \|c\|_\infty \frac{\alpha^{T+1}}{1 - \alpha} d^2.
\end{equation*}
Then choosing 
\begin{equation}\label{eq:choice_of_t}
T + 1 \geq \frac{1}{\log \alpha^{-1}} \log\left(\frac{2M \|c\|_\infty d^2 \varepsilon^{-1}}{(1 - \alpha)}\right) = \tilde{\calO}\left(\frac{1}{\log \alpha^{-1}} \log\left(\frac{M}{\varepsilon (1 - \alpha)}\right)\right)
\end{equation}
and
\begin{equation}\label{eq:choice_of_l}
L \geq \frac{\log\left(2(T+1)M\|c\|_\infty d^2 \varepsilon^{-1}\right)}{\log \alpha^{-1}} = \tilde{\calO}\left(\frac{1}{\log \alpha^{-1}} \log \left(\frac{M}{\varepsilon}\right)\right) ,
\end{equation}
we obtain $\|\tilde{h} - h\|_1 \leq \varepsilon$.
Note that for this choice of $L$, $\|\tilde{g} - g\|_\infty \leq \varepsilon / 2(T+1)$.
Since $T+ 1 \geq 1$, this implies that $\|\tilde{g} - g\|_\infty \leq \varepsilon$.
So the error for $\tilde{g}$ is controlled at the desired level as well.

Now consider the cost of computing $\tilde{g}$ and $\tilde{h}$.
Computing $\tilde{g}$ requires $L$ multiplications of a vector in $\bbR^{d^2}$ by $R \in \bbR^{d^2 \times d^2}$, which takes $\calO(Ld^4)$ time, followed by an inner product with $\mathbbm{1} \in \bbR^{d^2}$, multiplication with $\mathbbm{1} \in \bbR^{d^2}$ and multiplication by $\nicefrac{1}{d^2}$, each in $\calO(d^2)$ time.
This requires $\calO(L d^4) + \calO(d^2) + \calO(d^2) + \calO(d^2) = \calO(L d^4)$ time.
Letting $L$ be the minimum integer satisfying \eqref{eq:choice_of_l}, this takes time
\begin{equation*}
\mathcal{O}(Ld^4) = \tilde{\calO}\left(\frac{d^4}{\log \alpha^{-1}} \log \left(\frac{M}{\varepsilon}\right)\right).
\end{equation*}
On the other hand, given $\tilde{g}$, computing $\tilde{h}$ requires computing $c - \tilde{g} \in \bbR^{d^2}$ in $\calO(d^2)$ operations then multiplying by $R \in \bbR^{d^2\times d^2}$ $T+1$ times in $\calO(T d^4)$ time.
Finally, the sum may also be evaluated in $\calO(T d^4)$, requiring a total time of $\calO(d^2) + \calO(Td^4) + \calO(Td^4) = \calO(T d^4)$.
Letting $T$ be the minimum integer satisfying \eqref{eq:choice_of_t}, this takes time
\begin{equation}\label{eq:h_tilde_complexity}
\calO(Td^4) = \tilde{\calO}\left(\frac{d^4}{\log \alpha^{-1}} \log\left(\frac{M}{\varepsilon (1 - \alpha)}\right)\right).
\end{equation}
In total, we find that $\texttt{ApproxTCE}(R, L, T)$ takes time
\begin{equation*}
\tilde{\calO}\left(\frac{d^4}{\log \alpha^{-1}} \log \left(\frac{M}{\varepsilon}\right)\right) + \tilde{\calO}\left(\frac{d^4}{\log \alpha^{-1}} \log\left(\frac{M}{\varepsilon (1 - \alpha)}\right)\right) = \tilde{\calO}\left(\frac{d^4}{\log \alpha^{-1}} \log\left(\frac{M}{\varepsilon (1 - \alpha)}\right)\right).
\end{equation*}
\end{proof}

\subsubsection{Aperiodicity and irreducibility of elements of $\ri(\tc)$}\label{app:structure_of_reg_tc}
Next we prove Proposition \ref{prop:interior_tc} regarding the aperiodicity and irreducibility of elements of $\ri(\tc)$.
We begin with two elementary lemmas about the independent transition coupling.

\begin{lem}\label{lemma:ind_coup_of_order_m}
For any $k \geq 1$, $(P \otimes Q)^k = P^k \otimes Q^k$.
\end{lem}
\begin{proof}
The result clearly holds for $k = 1$, so assume that it holds for some $k \geq 1$.
For any $(x,y)$, $(x',y') \in \X\times\Y$, we can show
\begin{align*}
(P\otimes Q)^{k+1}((x,y),(x',y')) &= \sum\limits_{\tilde{x},\tilde{y}} (P\otimes Q)^k((x,y),(\tilde{x},\tilde{y})) \, P\otimes Q((\tilde{x},\tilde{y}),(x',y')) \\
&= \sum\limits_{\tilde{x},\tilde{y}} P^k(x,\tilde{x}) \, Q^k(y,\tilde{y}) \, P(\tilde{x}, x') \, Q(\tilde{y},y') \\
&= \sum\limits_{\tilde{x}} P^k(x,\tilde{x})\, P(\tilde{x}, x') \, \sum\limits_{\tilde{y}} Q^k(y,\tilde{y}) \, Q(\tilde{y},y') \\
&= P^{k+1}(x,x') \, Q^{k+1}(y,y') \\
&= P^{k+1} \otimes Q^{k+1}((x,y), (x',y')).
\end{align*}
By induction, the lemma is proven.
\end{proof}

\begin{lem}\label{lemma:irreducibility_of_ind_coup}
If $P$ and $Q$ are aperiodic and irreducible, then the independent transition coupling $P\otimes Q$ is aperiodic and irreducible.
\end{lem}
\begin{proof}
Since $P$ and $Q$ are aperiodic and irreducible, there exist $\ell_0, m_0 \geq 1$ such that for any $\ell \geq \ell_0$ and $m \geq m_0$, $P^\ell > 0$ and $Q^m > 0$ \citep[Proposition 1.7]{levin2017markov}.
Defining $k_0 := \ell_0 \vee m_0$, for every $k \geq k_0$, $P^k, Q^k > 0$.
By Lemma \ref{lemma:ind_coup_of_order_m}, it follows that $(P\otimes Q)^k = P^k \otimes Q^k > 0$ for all $k \geq k_0$.
Thus $P\otimes Q$ is irreducible.
Furthermore, for every $s \in \X \times \Y$, $\gcd\{k \geq 1: (P\otimes Q)^k(s,s) > 0 \} = \gcd\{..., k_0, k_0+1,...\} = 1$ and we conclude that $P\otimes Q$ is also aperiodic.
\end{proof}

Next we prove Proposition \ref{prop:interior_tc}.
Recall that for a set $\calU \subset \bbR^n$, $B_\varepsilon(u) \subset \bbR^n$ denotes the open ball of radius $\varepsilon > 0$ centered at $u \in \calU$, $\aff(\calU)$ denotes the affine hull, defined as $\aff(\calU) = \{\sum_{i=1}^k \alpha_i u_i : k \in \mathbb{N}, u_1, ..., u_k \in \calU, \sum_{i=1}^k \alpha_i = 1\}$, and $\ri(\calU)$ denotes the relative interior, defined as $\ri(\calU) =\{ u \in \calU: \exists \varepsilon > 0 \mbox{  s.t.  } B_\varepsilon(u) \cap \aff(\calU) \subset \calU\}$.

\interiortc*
\begin{proof}
First we establish that $P \otimes Q(s, s') > 0$ implies that $R(s, s') > 0$ for every $s, s' \in \X \times \Y$.
Suppose for the sake of contradiction that there exist $s, s' \in \X\times\Y$ such that $P\otimes Q(s,s') > 0$ and $R(s, s') = 0$.
By definition, there is some $\varepsilon > 0$ such that $B_\varepsilon(R) \cap \aff(\tc) \subset \tc$.
Defining $R' = R + \frac{\varepsilon}{2} d$ where $d = (R - P\otimes Q)/\|R - P\otimes Q\|_2$, one may verify that $R' \in B_\varepsilon(R) \cap \aff(\tc)$.
Thus by the choice of $R$, we have $R' \in \tc$.
However, our assumptions imply that $R'(s,s') < 0$, a contradiction.
This proves the preliminary claim.

By nature of the fact that $R((x,y), \cdot) \in \Pi(P(x,\cdot), Q(y, \cdot))$, one may easily establish that the reverse implication holds: $R(s,s') > 0$ implies that $P\otimes Q(s,s') > 0$ for every $s, s' \in \X\times\Y$.
As such, one may find a positive constant $a > 0$ such that $a P \otimes Q \leq R$ where the inequality is understood to hold element-wise.
Now, by Lemma \ref{lemma:irreducibility_of_ind_coup}, $P \otimes Q$ is aperiodic and irreducible.
Thus there exists $k \geq 1$ such that $(P\otimes Q)^k > 0$.
Thus, $R^k \geq a^k (P \otimes Q)^k > 0$ and it follows that $R$ is aperiodic and irreducible as well.
The mixing property of $R$ follows from \cite[Theorem 4.9]{levin2017markov}.
\end{proof}

\subsubsection{Complexity of entropic transition coupling improvement}\label{sec:complexity_of_entropic_pi}
Next we aim to prove Theorem \ref{thm:policy_improvement_complexity}, showing that \texttt{EntropicTCI} returns an improved transition coupling with error bounded by $\varepsilon > 0$ in $\tilde{\calO}(d^4 \varepsilon^{-4})$ time.
Recall that \texttt{EntropicTCI} improves policies by solving $d^2$ entropy-regularized OT transport problems, calling the \texttt{ApproxOT} algorithm \cite{altschuler2017near} for each problem.
Before we can prove Theorem \ref{thm:policy_improvement_complexity}, we must analyze the computational complexity of \texttt{ApproxOT}.
In the following discussion as well as Lemma \ref{lemma:bound_on_sinkhorn_error}, we find it most convenient to adopt the notation of \cite{altschuler2017near}.
Thus, we fix two probability vectors $r \in \Delta_m$ and $c \in \Delta_n$, a non-negative cost matrix $C \in \bbR_+^{m\times n}$, a regularization parameter $\xi > 0$, and an error tolerance $\varepsilon > 0$.
For vectors in $\bbR^m$ or $\bbR^n$ and matrices in $\bbR^{m\times n}$, we temporarily drop the double-indexing convention, using subscripts instead to denote elements (i.e. $u_i$ and $X_{ij}$).
Finally, for a coupling $X \in \Pi(r,c)$, let $H(X) = -\sum_{ij} X_{ij} \log X_{ij}$ be the Shannon entropy.

Recall that the entropic OT problem is defined as,
\begin{align}\label{eq:appendix_eot_problem}
\begin{split}
\mbox{minimize}\quad & \langle X, C\rangle - \frac{1}{\xi} H(X) \\
\mbox{subject to}\quad & X \in \Pi(r,c).
\end{split}
\end{align}
In \cite{cuturi2013sinkhorn}, Cuturi showed that solutions to \eqref{eq:appendix_eot_problem} have a computationally convenient form.
Namely, if $X_\xi^* \in \Pi(r, c)$ is the solution to \eqref{eq:appendix_eot_problem}, then it is unique and can be written as $X_\xi^* = \diag(e^{u^*}) K \diag(e^{v^*})$ for some $u^* \in \bbR^m$ and $v^* \in \bbR^n$, where $K = e^{-\xi C}$.
As a result, \eqref{eq:appendix_eot_problem} can be formulated as a matrix scaling problem and solved using Sinkhorn's algorithm \citep{sinkhorn1967diagonal}.

More recent work \cite{altschuler2017near} introduced the \texttt{ApproxOT} algorithm (Algorithm \ref{alg:approx_ot}), which combines Sinkhorn's algorithm with a rounding step to obtain an approximate solution to the OT problem.
In particular, \texttt{ApproxOT} runs \texttt{Sinkhorn} (Algorithm \ref{alg:sinkhorn}) to obtain a coupling of the form $X' = \diag(e^{u'}) K \diag(e^{v'}) \in \Pi(r', c')$, where $\|r - r'\|_1 + \|c - c'\|_1 \leq \varepsilon$, then applies \texttt{Round} (Algorithm \ref{alg:round}) to $X'$ to obtain $\hat{X} \in \Pi(r, c)$.
\texttt{ApproxOT} was originally intended for approximating the OT cost, but we use it to approximate the regularized optimal coupling $X_\xi^* \in \Pi(r, c)$.
In particular, we wish to show that for appropriate choice of parameters, \texttt{ApproxOT} yields a coupling $\hat{X} \in \Pi(r, c)$ such that $\|\hat{X} - X_\xi^*\|_1 \leq \varepsilon$ in $\tilde{\calO}(mn \varepsilon^{-4})$ time.
To the best of our knowledge, this result has not appeared in the literature.
So we state and prove it in Lemma \ref{lemma:bound_on_sinkhorn_error}.

%\begin{wrapfigure}{r}{0.45\textwidth}
\begin{center}
\begin{minipage}{0.52\textwidth}
\begin{algorithm}[H]
\SetKwInOut{Input}{input}
\SetKwInOut{Result}{result}
\DontPrintSemicolon
\SetAlgoLined
\Result{Optimal coupling}
\Input{$r, c, C, \xi, \varepsilon$}
\tcc{Subset to positive elements}
$\mathcal{R} \leftarrow \{i: r_i > 0\}$, $\mathcal{C} \leftarrow \{j: c_j > 0\}$\;
$\mathcal{S} \leftarrow \mathcal{R} \times \mathcal{C}$, $\tilde{r} \leftarrow r_{\mathcal{R}}$, $c \leftarrow c_{\mathcal{C}}$\;

\tcc{Set parameters}
$J \leftarrow 4 \log n \|C_{\mathcal{S}}\|_\infty / \varepsilon - \log \min_{ij} \{\tilde{r}_i, \tilde{c}_j\}$\;
$\varepsilon' \leftarrow \varepsilon^2 / 8J$\;
$K \leftarrow \exp(-\xi C_{\mathcal{S}})$\;

\tcc{Approximate Sinkhorn projection}
$X' \leftarrow \texttt{Sinkhorn}(K, \tilde{r}, \tilde{c}, \varepsilon')$\;

\tcc{Round to feasible coupling}
$X' \leftarrow \texttt{Round}(X', \Pi(\tilde{r}, \tilde{c}))$\;

\tcc{Replace zeroes}
$\hat{X} \leftarrow 0_{d \times d}$, $\hat{X}_{\mathcal{S}} \leftarrow X'$\;
\Return{$\hat{X}$}
\caption{\texttt{ApproxOT}}\label{alg:approx_ot}
\end{algorithm}
\end{minipage}
\hspace{1mm}
\begin{minipage}{0.45\textwidth}
\begin{algorithm}[H]
\SetKwInOut{Input}{input}
\SetKwInOut{Result}{result}
\DontPrintSemicolon
\SetAlgoLined
\Result{Approximate Sinkhorn projection}
\Input{$K, r, c, \varepsilon'$}
$k \leftarrow 0$\;
$X_0 \leftarrow K / \|K\|_1, \, u^0 \leftarrow 0, \, v^0 \leftarrow 0$\;
\While{$\|X_k\mathbbm{1} - r\|_1 + \|X_k^\top \mathbbm{1} - c\|_1 > \varepsilon'$}{
$k \leftarrow k + 1$\;
\If{$k$ odd}{
	$r^k \leftarrow X_k \mathbbm{1}$\;
	$u_i \leftarrow \log(r_i / r^k_i)$ for $i \in [n]$\;
	$u^k \leftarrow u^{k-1} + u$, $v^k \leftarrow v^{k-1}$\;
}
\Else{
	$c^k \leftarrow X_k^\top \mathbbm{1}$\;
	$v_j \leftarrow \log(c_j / c_j^k)$ for $j \in [n]$\;
	$v^k \leftarrow v^{k-1} + v$, $u^k \leftarrow u^{k-1}$\;
}
$X_k \leftarrow \diag(e^{u^k}) K \diag(e^{v^k})$\;
}
\Return{$X_k$}
\caption{\texttt{Sinkhorn}}\label{alg:sinkhorn}
\end{algorithm}
\end{minipage} 
%\vspace{-10mm}
%\end{wrapfigure}
\end{center}

%\clearpage

%\begin{wrapfigure}{r}{0.40\textwidth}
\begin{center}
%\hspace{5mm}
%\vspace{-5mm}
\begin{minipage}{0.4\textwidth}
\begin{algorithm}[H]
\SetKwInOut{Input}{input}
\SetKwInOut{Result}{result}
\DontPrintSemicolon
\SetAlgoLined
\Result{Feasible coupling}
\Input{$F, \Pi(r, c)$}
$r' \leftarrow F \mathbbm{1}$\;
$X \leftarrow \diag(x)$ with $x_i = r_i / r'_i \wedge 1$\;
$F' \leftarrow X F$\;
$c' \leftarrow (F')^\top \mathbbm{1}$\;
$Y \leftarrow \diag(y)$ with $y_j = c_j / c'_j \wedge 1$\;
$F'' \leftarrow F' Y$\;
$r'' \leftarrow F'' \mathbbm{1}$, $c'' \leftarrow (F'')^\top \mathbbm{1}$\;
$\err_{r} \leftarrow r - r''$, $\err_{c} \leftarrow c - c''$\;
\Return{$F'' + \err_{r}\err_{c}^\top / \|\err_{r}\|_1$}
\caption{\texttt{Round}}\label{alg:round}
\end{algorithm}
%\vspace{-7mm}
\end{minipage}
\end{center}
%\end{wrapfigure}

Note that \texttt{ApproxOT} was originally defined for fully-supported marginal probability vectors $(r, c > 0)$.
However, this will not always be the case in Algorithm \ref{alg:approx_entropic_pi}.
In particular, transition couplings may be sparse, even when $P$ and $Q$ are strictly positive.
Thus we add an extra step to \texttt{ApproxOT} that subsets the quantities of interest to their positive entries.
For an index set $\mathcal{I}$ and a vector / matrix $A$ we let $A_{\mathcal{I}}$ denote the subvector / matrix that retains only elements with indices contained in $\mathcal{I}$.

\begin{lem}\label{lemma:bound_on_sinkhorn_error}
Let $r \in \Delta_m$ and $c \in \Delta_n$ have all positive entries, $C \in \mathbb{R}^{m\times n}_+$, $\xi > 0$ and $\varepsilon \in (0,1)$.
Then $\emph{\texttt{ApproxOT}}(r, c, C, \xi, \varepsilon)$ (Algorithm \ref{alg:approx_ot}) returns a coupling $\hat{X} \in \Pi(r, c)$ such that $\|\hat{X} - X_\xi^*\|_1 \leq \varepsilon$, where $X^*_\xi \in \argmin_{X \in \Pi(r,c)} \langle X, C \rangle - \nicefrac{1}{\xi} H(X)$, in time $\tilde{\calO}(mn \varepsilon^{-4} \xi \|C\|_\infty (\xi^2 \|C\|_\infty^2 + (\log b^{-1})^2))$ where $b = \min_{ij} \{r_i, c_j\}$.
\end{lem}
\begin{proof}
Let $\varepsilon' > 0$, $K = e^{-\xi C}$, $X' \in \Delta_{m \times n}$ be the output of $\texttt{Sinkhorn}(K, r, c, \varepsilon')$ and $\hat{X} \in \Pi(r, c)$ be the output of $\texttt{Round}(X', \Pi(r, c))$.
By the triangle inequality,
\begin{equation}\label{eq:sinkhorn_triangle_ineq}
\|\hat{X} - X_\xi^*\|_1 \leq \|\hat{X} - X'\|_1 + \|X' - X_\xi^*\|_1.
\end{equation}
We will first describe how to control the second term on the right hand side.
By Pinsker's inequality, $\|X' - X^*_\xi\|_1^2 \leq 2 \calK(X^*_\xi \| X')$, so it suffices to bound the KL-divergence between the two couplings.
From Lemma 2 of \cite{cuturi2013sinkhorn} that $X^*_\xi = \diag(e^{u^*})K \diag(e^{v^*})$ for some $u^* \in \bbR^m$, $v^* \in \mathbb{R}^n$, and $K = e^{-\xi C}$.
By construction we also have $X' = \diag(e^{u'}) K \diag(e^{v'})$ for some $u' \in \bbR^m$ and $v' \in \mathbb{R}^n$.
Now rewriting the KL-divergence,
\begin{align*}
\calK(X_\xi^* \| X') &= \sum\limits_{ij} X_{\xi,ij}^* \log X^*_{\xi, ij} - \sum\limits_{ij} X^*_{\xi, ij} \log X'_{ij} \\
&= \sum\limits_{ij} X_{\xi,ij}^* \left(u^*_i + v^*_j - \xi C_{ij}\right) - \sum\limits_{ij} X_{\xi,ij}^* \left(u'_i + v'_j - \xi C_{ij}\right) \\
&= \sum\limits_{ij} X_{\xi, ij}^* (u^*_i - u'_i) + \sum\limits_{ij} X_{\xi, ij}^* (v^*_j - v'_j) \\
&= \sum\limits_i (u^*_i - u'_i) \sum\limits_j X_{\xi,ij}^* + \sum_j (v^*_j - v'_j) \sum\limits_i X_{\xi, ij}^* \\
&= \sum\limits_i (u^*_i - u'_i) r_i + \sum\limits_j (v^*_j - v'_i) c_j \\
&= \langle u^* - u', r \rangle + \langle v^* - v', c \rangle.
\end{align*}
Writing $\psi(u, v) = \langle \mathbbm{1}, \diag(e^u)K \diag(e^v)\mathbbm{1}\rangle - \langle u, r\rangle - \langle v, c\rangle$ for the objective of the dual entropic OT problem \citep{dvurechensky2018computational}, we immediately see that
\begin{equation*}
\tilde{\psi}(u', v') := \psi(u', v') - \psi(u^*, v^*) = \langle u^* - u', r \rangle + \langle v^* - v', c \rangle.
\end{equation*}
Now let $r'$ and $c'$ be the row and column marginals of $X'$, respectively.
Using the two previous displays and applying the upper bound from \cite[Lemma 2]{dvurechensky2018computational}, we obtain
\begin{equation*}
\calK(X_\xi^* \|X') = \tilde{\psi}(u, v) \leq J\left(\|r' - r\|_1 + \|c' - c\|_1\right),
\end{equation*}
where $J = \xi\|C\|_\infty - \log \min_{ij}\{r_i, c_j\}$.
For ease of notation, we will let $b := \min_{ij}\{r_i, c_j\}$.
Now by \cite[Theorem 2]{altschuler2017near} and the fact that each iteration of \texttt{Sinkhorn} takes $\calO(mn)$ time, $\texttt{Sinkhorn}(K, r, c, \varepsilon')$ returns a coupling with $X' \in \Pi(r', c')$ satisfying $\|r' - r\|_1 + \|c' - c\|_1 \leq \varepsilon'$ in $\mathcal{O}(mn (\varepsilon')^{-2} \log(s / \ell))$ time where $s = \sum_{ij} K_{ij}$ and $\ell = \min_{ij} K_{ij}$.
As $C$ is non-negative, $s = \sum_{ij} e^{-\xi C_{ij}} \leq \sum_{ij} 1 = mn$.
Furthermore, $\ell = e^{-\xi \|C\|_\infty}$ so we get a total runtime of $\mathcal{O}(mn (\varepsilon')^{-2} (\log mn + \xi \|C\|_\infty)) = \tilde{\mathcal{O}}(mn (\varepsilon')^{-2} \xi \|C\|_\infty)$.
Now choosing $\varepsilon' = \varepsilon^2 / 8J$, we have
\begin{equation*}
\|X' - X_\xi^*\|_1 \leq \sqrt{2J (\|r' - r\|_1 + \|c' - c\|_1)} \leq \sqrt{2J \varepsilon'} = \sqrt{2J \varepsilon^2 / 8J} = \varepsilon / 2.
\end{equation*}
Since $\varepsilon' = \varepsilon^2 / 8J$, the runtime becomes 
\begin{align*}
\tilde{\mathcal{O}}(mn (\varepsilon')^{-2} \xi\|C\|_\infty) &= \tilde{\mathcal{O}}(mn (\varepsilon^2 / 8J)^{-2} \xi \|C\|_\infty) \\
&= \tilde{\mathcal{O}}(mn \varepsilon^{-4} \xi \|C\|_\infty J^2) \\
&= \tilde{\mathcal{O}}(mn \varepsilon^{-4} \xi \|C\|_\infty (\xi \|C\|_\infty - \log b)^2) \\
&= \tilde{\mathcal{O}}(mn \varepsilon^{-4} \xi\|C\|_\infty (\xi^2 \|C\|_\infty^2 + (\log b^{-1})^2)).
\end{align*}
Now we must bound $\|\hat{X} - X'\|_1$.
By \cite[Lemma 7]{altschuler2017near}, Algorithm \ref{alg:round} returns $\hat{X}$ satisfying
\begin{equation*}
\|\hat{X} - X'\|_1 \leq 2(\|r' - r\|_1 + \|c' - c\|_1),
\end{equation*}
in $\mathcal{O}(mn)$ time.
So it suffices to check that $\|r' - r\|_1 + \|c' - c\|_1 \leq \varepsilon' = \varepsilon^2 / 8J$ is enough to guarantee that $\|\hat{X} - X'\|_1 \leq \varepsilon / 2$.
This will follow immediately from $\|\hat{X} - X'\|_1 \leq 2 \varepsilon' = \varepsilon^2 / 4J \leq \varepsilon / 2J$ if we can establish that $J \geq 1$.
To see this, first note that $b = \min_{i,j} \{r_i, c_j\} \leq 1/(m \vee n)$.
This implies that $-\log b \geq \log (m \vee n)$ and since $\xi > 0$,
\begin{equation*}
J = \xi \|C\|_\infty - \log b \geq -\log b \geq \log (m \vee n) \geq 1,
\end{equation*}
assuming that $m \vee n > 2$.
If $m \vee n = 2$, then one can check that letting $\varepsilon' = \varepsilon^2 \log 2 / 8J$ is enough to obtain the desired bounds without affecting the computational complexity.
Thus by \eqref{eq:sinkhorn_triangle_ineq}, we obtain $\|\hat{X} - X^*_\xi\|_1 \leq \varepsilon$ in time $\tilde{\mathcal{O}}(mn \varepsilon^{-4} \xi \|C\|_\infty (\xi^2 \|C\|_\infty^2 + (\log b^{-1})^2) + mn) = \tilde{\mathcal{O}}(mn \varepsilon^{-4} \xi \|C\|_\infty (\xi^2 \|C\|_\infty^2 + (\log b^{-1})^2))$.
\end{proof}

Now we can proceed to the proof of Theorem \ref{thm:policy_improvement_complexity}.

\policyimprovementcomplexity*
%\begin{proof}[Proof of Lemma \ref{lemma:policy_improvement_complexity}]
\begin{proof}
Without loss of generality, we may assume that $h$ is non-negative.
Otherwise, one can consider the modified bias $h + \|h\|_\infty \mathbbm{1}$.
Since we are interested in optimal couplings with respect to $h$ rather than expected cost and $\|h + \|h\|_\infty \mathbbm{1}\|_\infty = \calO(\|h\|_\infty)$, this has no effect on the output of \texttt{ApproxOT} or the computational complexity.
Now, in order to analyze the complexity of \texttt{EntropicTCI}, we must first analyze the complexity of \texttt{ApproxOT}.
Fix $s=(x,y) \in \X \times \Y$ and, after removing points outside of the supports of $P(x,\cdot)$ and $Q(y,\cdot)$, consider the entropic OT problem for marginal probability measures $P(x,\cdot)$ and $Q(y,\cdot)$ and cost $h$,
\begin{align}\label{eq:dual_eot}
\begin{split}
\mbox{minimize} \quad &\langle r, h \rangle - \frac{1}{\xi} H(r) \\
\mbox{subject to} \quad &r \in \tcrow.
\end{split}
\end{align}
Then by \cite[Lemma 2]{cuturi2013sinkhorn}, there exists a unique solution $r^*_s \in \tcrow$ to problem \eqref{eq:dual_eot}.
Furthermore by Lemma \ref{lemma:bound_on_sinkhorn_error}, $\texttt{ApproxOT}(P(x,\cdot)^\top, Q(y,\cdot)^\top, h, \xi, \varepsilon)$ returns $\hat{r}_s \in \tcrow$ such that $\|\hat{r}_s - r^*_s\|_1 \leq \varepsilon$ in $\tilde{\calO}(d^2 \varepsilon^{-4})$ time.
One may also verify using arguments in \cite{altschuler2017near} that $\hat{r}_s \in \ri(\tcrow)$.

Now we may analyze the error and computational complexity of $\texttt{EntropicTCI}(h, \xi, \varepsilon)$.
Calling \newline $\texttt{ApproxOT}(P(x,\cdot)^\top, Q(y,\cdot)^\top, h, \xi, \varepsilon)$ for every $s = (x,y) \in \X\times\Y$, we obtain $\hat{R} \in \tc$, where $\hat{R}(s,\cdot) = \hat{r}_s(\cdot)$, in $d^2 \tilde{\calO}(d^2 \varepsilon^{-4}) = \tilde{\calO}(d^4 \varepsilon^{-4})$ time.
Note that since the relative interior commutes with cartesian products of convex sets, $\hat{R} \in \ri(\tc)$.
Then defining $R^* \in \tc$ such that $R^*(s,\cdot) = r^*_s(\cdot)$, we have
\begin{equation*}
\max\limits_s \|\hat{R}(s,\cdot) - R^*(s,\cdot) \|_1 = \max\limits_s \|\hat{r}_s - r^*_s \|_1 \leq \varepsilon,
\end{equation*}
by construction.
This concludes the proof.
\end{proof}

\subsection{Proofs from Section \ref{sec:consistency}}
\label{sec:proofs_consistency}

Our proof of Theorem \ref{thm:consistency} relies on a well-known result regarding the stability of certain optimization problems.
Before stating this result, fix spaces $\calZ$ and $\calU$ corresponding to the set of possible solutions and set of parameters for the optimization problem of interest, respectively.
Now consider the following problem.
\begin{align}
\begin{split}
\mbox{minimize} \quad & f(z, u) \\
\mbox{subject to} \quad & z \in \Phi(u).
\end{split}
\label{eq:generic_optprob}
\end{align}
\noindent Note that $f(\cdot, u): \calZ \rightarrow \mathbb{R}$ describes the objective to be minimized and $\Phi(u) \subset \calZ$ represents the feasible set of Problem \eqref{eq:generic_optprob}, both indexed by a parameter $u \in \calU$.
We will call a set $\calV \subset \calZ$ a neighborhood of a subset $\mathcal{W} \subset \calZ$ if $\mathcal{W} \subset \setint \calV$.
Neighborhoods in $\calU$ will be defined similarly.
Recall that a multifunction $F: \calU \rightarrow 2^\calZ$ is upper semicontinuous at a point $u_0 \in \calU$ if for any neighborhood $\calV_\calZ$ of the set $F(u_0)$, there exists a neighborhood $\calV_\calU$ of $u_0$ such that for every $u \in \calV_\calU$, $F(u) \subset \calV_\calZ$.

\begin{thm}[\cite{bonnans2013perturbation}, Proposition 4.4]
\label{thm:perturbopt}
Let $u_0$ be a given point in the parameter space $\calU$.
Suppose that (i) the function $f(z, u)$ is continuous on $\calZ \times \calU$, (ii) the graph of the multifunction $\Phi(\cdot)$ is a closed subset of $\calU \times \calZ$, (iii) there exists $\alpha \in \mathbb{R}$ and a compact set $C \subset \calZ$ such that for every $u$ in a neighborhood of $u_0$, the level set $\{z \in \Phi(u): f(z, u) \leq \alpha\}$ is nonempty and contained in $C$, (iv) for any neighborhood $\mathcal{V}_{\calZ}$ of the set $\argmin_{z \in \Phi(u_0)} f(z, u_0)$ there exists a neighborhood $\mathcal{V}_U$ of $u_0$ such that $\mathcal{V}_{\calZ} \cap \Phi(u) \neq \emptyset$ for all $u \in \mathcal{V}_\calU$.
Then the optimal value function $u \mapsto \min_{z \in \Phi(u)} f(z, u)$ is continuous at $u = u_0$ and the multifunction $u \mapsto \argmin_{z \in \Phi(u)} f(z, u)$ is upper semicontinuous at $u_0$.
\end{thm}

Both Problems \eqref{eq:probI} and \eqref{eq:probII} may be recast in the form of Problem  \eqref{eq:generic_optprob}.
Let
\begin{equation*}
\calZ = \left\{(\lambda, R) \in \Delta_{d^2} \times \Delta_{d^2}^{d^2}: R \in \tc \mbox{ for some } P, Q \in \Delta_d^d, \lambda R = \lambda\right\}
\end{equation*}
and $\calU = \Delta_d^d \times \Delta_d^d$ be the set of all valid pairs of transition matrices in $\mathbb{R}^{d \times d}$.
It is straightforward to verify that $\calZ$ and $\calU$ are in fact compact subsets of $\mathbb{R}^{d^2} \times \mathbb{R}^{d^2 \times d^2}$ and $\mathbb{R}^{d\times d} \times \mathbb{R}^{d\times d}$, respectively.
The objective function $f(\cdot)$ is identified with the map $(\lambda, R) \mapsto \langle c, \lambda\rangle$ and does not depend on the parameter $u = (P, Q)$.
We will refer to the constraint functions for Problems \eqref{eq:probI} and \eqref{eq:probII} by $\Phi: \calU \rightarrow 2^\calZ$ and $\Phi_\eta: \calU \rightarrow 2^\calZ$, and their optimal solution functions by $\Phi^*: \calU \rightarrow 2^\calZ$ and $\Phi_\eta^*: \calU \rightarrow 2^\calZ$, respectively.

\consistency*
\begin{proof}
We will prove the result for Problem \eqref{eq:otc_problem_rv} as the proof for Problem \eqref{eq:entropic_otc} is similar.
As the two problems are equivalent, it suffices to check the conditions of Theorem \ref{thm:perturbopt} for Problem \eqref{eq:probI} at the point $u_0 = (P, Q) \in \calU$.
First, (i) is vacuously true since the objective $f(\cdot)$ does not depend on $u$.
Next, we will show that the graph of $\Phi(\cdot)$ is a closed subset of $\calU \times \calZ$.
Fix a sequence $\{(P_n, Q_n, \lambda_n, R_n)\}_{n \geq 1} \subset \graph \Phi(\cdot)$.
As a subset of the compact set $\Delta_d^d \times \Delta_d^d \times \Delta_{d^2} \times \Delta_{d^2}^{d^2}$, it has a subsequence, which we also label as $\{(P_n, Q_n, \lambda_n, R_n)\}_{n \geq 1}$ converging to some $(P', Q', \lambda', R') \in \Delta_d^d \times \Delta_d^d \times \Delta_{d^2} \times \Delta_{d^2}^{d^2}$.
Taking limits of the linear equations $R_n \in \Pi_{\mbox{\tiny TC}}(P_n, Q_n)$ and $\lambda_n R_n = \lambda_n$, we conclude that $R' \in \Pi_{\mbox{\tiny TC}}(P', Q')$ and $\lambda' R' = \lambda'$.
Thus $(P', Q', \lambda', R') \in \graph \Phi(\cdot)$ and (ii) holds.
To show that (iii) is satisfied, note that one may let $\alpha = \|c\|_\infty$ and use the fact that the entire set $\calZ$ is compact.
Finally, we will show that (iv) is satisfied.
Let $\calV_\calZ \subset \calZ$ be a neighborhood of $\argmin_{z \in \Phi(u_0)} f(z, u_0)$.
Then define the neighborhood $\calV_\calU$ of $u_0 = (P, Q)$ as
\begin{equation*}
\calV_\calU := \{(P, Q) \in \Delta_d^d \times \Delta_d^d: R \in \tc \mbox{ for some } (\lambda, R) \in \calV_\calZ\}.
\end{equation*}
Note that $\calV_\calU$ is nonempty by the non-emptiness of $\calV_\calZ$ and the definition of $\calZ$.
Moreover, $\calV_\calZ \cap \Phi(u) \neq \emptyset$ for all $u \in \calV_\calU$ by construction.
Thus all the conditions of Theorem \ref{thm:perturbopt} are satisfied and the desired convergence holds.
\end{proof}

\section*{Acknowledgements}
The authors would like to thank Quoc Tran-Dinh for helpful discussions and Jason Altschuler for contributions to the proof of Lemma \ref{lemma:bound_on_sinkhorn_error}.
K.O. and A.N were supported in part by NIH Grant R01 HG009125-01 and NSF Grant DMS-1613072.  
K.M. was supported in part by NSF Grant DMS-1847144.
K.M. and A.N. were supported in part by NSF Grant DMS-1613261.

\bibliography{references}

\clearpage
%\begin{appendices}
\appendix
\section{Properties of the OTC Problems}\label{app:existence}
In this appendix, we prove that solutions to the OTC and constrained OTC problems exist via continuity and compactness arguments and establish the triangle inequality for the unconstrained problem.
For a metric space $\calU$ and a sequence of Borel probability measures $\{\mu^n\} \subset \calM(\calU)$, we say that $\mu^n$ \emph{converges weakly to} $\mu \in \calM(\calU)$, denoted by $\mu^n  \Rightarrow \mu$, if for every continuous and bounded function $f: \calU \rightarrow \bbR$, $\int f \, d\mu^n \rightarrow \int f \, d\mu$.
A set $\Pi \subset \calM(\calU)$ is said to be \emph{weakly compact} if every sequence in $\Pi$ contains a subsequence converging weakly to an element of $\Pi$.
$\Pi$ is said to be \emph{tight} if for every $\varepsilon > 0$, there exists a compact set $K \subset \calU$ such that $\mu(K) > 1 - \varepsilon$ for every $\mu \in \Pi$.
Tightness and relative compactness are related by Prohorov's theorem which states that if $\calU$ is a separable metric space, $\Pi \subset \calM(\calU)$ is tight if and only if its closure is relatively compact.
Note that $\X^\bbN\times\Y^\bbN$ is complete and separable when equipped with the metric 
\begin{equation*}
d((\bfx^1, \bfy^1), (\bfx^2, \bfy^2)) = \sum\limits_{k=0}^\infty 2^{-k} \delta((x^1_k, y^1_k)\neq (x^2_k, y^2_k)).
\end{equation*}
Finally, we remark that since $c: \X\times\Y \rightarrow \bbR_+$ is continuous and bounded, $\tilde{c}(\bfx, \bfy) = c(x_0, y_0)$ is as well.

\subsection{Existence for the OTC Problem}
We begin by proving that $\tcrv$ is weakly compact.
\begin{lem}\label{lemma:weak_compactness_of_tc}
$\pitc(\bbP, \bbQ)$ is weakly compact.
\end{lem}
\begin{proof}
By \cite[Lemma 4.4]{villani2008optimal}, $\Pi(\bbP, \bbQ)$ is tight.
Since $\tcrv \subset \Pi(\bbP, \bbQ)$, $\tcrv$ is tight as well.
Thus by Prohorov's theorem, the closure of $\tcrv$ is weakly compact.
So we need only prove that $\tcrv$ is closed.
Take a sequence $\{\pi^n\} \subset \pitc(\bbP, \bbQ)$ such that $\pi^n \Rightarrow \pi \in \calM(\X^\bbN\times\Y^\bbN)$.
Since $\Pi(\bbP, \bbQ)$ is weakly compact \citep{villani2008optimal}, $\pi \in \Pi(\bbP, \bbQ)$.
Then it suffices to prove that $\pi$ is stationary, Markov, and has a transition matrix that satisfies the transition coupling property.

We begin by proving that $\pi$ is stationary.
Let $\sigma: \X^\bbN\times\Y^\bbN \rightarrow \X^\bbN\times\Y^\bbN$ be the left-shift map defined for every $(\bfx, \bfy) \in \X^\bbN\times\Y^\bbN$ by $\sigma(\bfx, \bfy) = (x_1^\infty, y_1^\infty)$.
Then stationarity of any $\mu \in \calM(\X^\bbN\times\Y^\bbN)$ is defined by $\mu = \mu \circ \sigma^{-1}$.
Since each $\pi^n$ is stationary, $\pi^n = \pi^n\circ \sigma^{-1}$.
Noting that $\sigma$ is continuous, the continuous mapping theorem implies that $\pi^n \circ \sigma^{-1} \Rightarrow \pi \circ \sigma^{-1}$, so $\pi^n \Rightarrow \pi \circ \sigma^{-1}$.
Since weak limits are unique, we conclude that $\pi = \pi \circ \sigma^{-1}$ and $\pi$ is stationary.

Next we prove that $\pi$ is Markov.
Since $\X\times\Y$ is finite, for any cylinder set $[s_0^k] = \{(\bfx, \bfy) \in (\X \times \Y)^\bbN: (x_j, y_j) = s_j, 0 \leq j \leq k\}$, $\pi^n([s_0^k]) \rightarrow \pi([s_0^k])$.
Then
\begin{equation}\label{eq:convergence_of_transitions1}
\frac{\pi^n([s_0\cdots s_k])}{\pi^n([s_0\cdots s_{k-1}])} \rightarrow \frac{\pi([s_0 \cdots s_k])}{\pi([s_0\cdots s_{k-1}])} 
\end{equation}
and
\begin{equation}\label{eq:convergence_of_transitions2}
\frac{\pi^n([s_{k-1} s_k])}{\pi^n([s_{k-1}])} \rightarrow \frac{\pi([s_{k-1} s_k])}{\pi([s_{k-1}])},
\end{equation}
where we let $\nicefrac{0}{0} = 0$.
But since $\pi^n$ is Markov for each $n \geq 1$,
\begin{equation*}
\frac{\pi^n([s_0 \cdots s_k])}{\pi^n([s_0 \cdots s_{k-1}])} = \frac{\pi^n([s_{k-1} s_k])}{\pi^n([s_{k-1}])}.
\end{equation*}
As a result, $\pi([s_0 \cdots s_k])/\pi([s_0 \cdots s_{k-1}]) = \pi([s_{k-1} s_k]) / \pi([s_{k-1}])$.
Thus, $\pi$ is Markov.

Now, we need only show that $\pi$ satisfies the transition coupling property.
Letting $R_n$ and $R$ denote the transition matrices of $\pi^n$ and $\pi$, respectively,  \eqref{eq:convergence_of_transitions1} and \eqref{eq:convergence_of_transitions2} imply that $R_n(s, s') \rightarrow R(s, s')$ for every $s,s' \in \X\times\Y$.
Then for any $(x,y) \in \X\times\Y$ and $y' \in \Y$,
\begin{equation}\label{eq:convergence_of_trans_mat}
\sum\limits_{x'} R_n((x,y), (x',y')) \rightarrow \sum\limits_{x'} R((x,y),(x',y')).
\end{equation}
But as $R_n \in \tc$, $\sum_{x'} R_n((x,y),(x',y')) = Q(y,y')$ and it follows that $\sum_{x'} R((x,y),$ $(x',y')) = Q(y,y')$.
Employing a similar argument to the other marginal of $R$, one may show that in fact $R \in \tc$.
Therefore, $\pi \in \pitc(\bbP, \bbQ)$ and we conclude that $\pitc(\bbP, \bbQ)$ is weakly compact.
\end{proof}

\begin{prop}\label{prop:existence}
The OTC problem \eqref{eq:otc_problem_rv} has a solution.
\end{prop}
\begin{proof}
Let $\{\pi^n\} \subset \pitc(\bbP, \bbQ)$ be a sequence such that 
\begin{equation*}
\int \tilde{c} \, d\pi^n \rightarrow \inf\limits_{\pi \in \pitc(\bbP, \bbQ)} \int \tilde{c} \, d\pi.
\end{equation*}
By Lemma \ref{lemma:weak_compactness_of_tc}, $\pitc(\bbP, \bbQ)$ is weakly compact.
Thus, there exists a subsequence $\{\pi^{n_k}\}$ such that $\pi^{n_k} \Rightarrow \pi^*$ for some $\pi^* \in \pitc(\bbP, \bbQ)$.
Since $\tilde{c}$ is continuous and bounded,
\begin{equation*}
\int \tilde{c} \, d\pi^* = \lim\limits_{k\rightarrow\infty} \int \tilde{c} \, d\pi^{n_k} = \inf\limits_{\pi \in \pitc(\bbP, \bbQ)}\int \tilde{c} \, d\pi.
\end{equation*}
Thus $\pi^*$ is an optimal solution for Problem \eqref{eq:otc_problem_rv}.
\end{proof}

\subsection{Existence for the Constrained OTC Problem}\label{app:approx_existence}
We begin by proving that $\tceta$ is convex and compact as a subset of $\bbR^{d^2 \times d^2}$.
\begin{lem}\label{lemma:compactness_of_tceta}
For any $\eta > 0$, the constrained set of transition coupling matrices $\tceta$ is convex and compact.
\end{lem}
\begin{proof}
Fixing $\eta > 0$, we begin by showing that $\tceta$ is convex.
Let $R, R' \in \tceta$, $\lambda \in (0,1)$, and define $R_\lambda := \lambda R + (1 - \lambda) R'$.
Since $\tc$ is convex, $R_\lambda \in \tc$.
Moreover, using the convexity of the KL-divergence, for any $s \in \X\times \Y$,
\begin{align*}
\calK(R_\lambda(s,\cdot) \| P\otimes Q(s,\cdot)) &=\calK(\lambda R(s, \cdot) + (1-\lambda) R'(s,\cdot) \| P\otimes Q(s,\cdot)) \\
&\leq \lambda \calK(R(s,\cdot) \| P\otimes Q (s, \cdot)) +(1 - \lambda) \calK(R'(s,\cdot) \| P\otimes Q(s, \cdot)) \\
&\leq \lambda \eta + (1 - \lambda) \eta \\
&= \eta.
\end{align*}
Thus $R_\lambda \in \tceta$ and we conclude that $\tceta$ is convex.

Next we prove compactness.
Note that as a subset of the compact set $\tc$ we need only show that $\tceta$ is closed.
Let $\{R_n\} \subset \tceta$ be a sequence converging to $R \in \bbR^{d^2 \times d^2}$.
By the compactness of $\tc$, $R \in \tc$.
Now for any $s \in \X\times\Y$, note that $R(s,\cdot)$ is absolutely continuous with respect to $P\otimes Q(s,\cdot)$.
This implies that, for every $s' \in \X \times \Y$,
\begin{equation*}
R(s,s') \log\frac{R(s,s')}{P\otimes Q(s,s')} < \infty,
\end{equation*}
where we let $0\log(0/0) = 0$.
Then $\calK(\cdot \| P\otimes Q(s,\cdot))$ is continuous at $R(s,\cdot)$ and we have that 
\begin{equation*}
\calK(R(s,\cdot) \| P\otimes Q(s,\cdot)) = \lim\limits_{n\rightarrow\infty} \calK(R_n(s,\cdot) \| P\otimes Q(s,\cdot)) \leq \eta.
\end{equation*}
Thus $R \in \tceta$ and we conclude that $\tceta$ is compact.
\end{proof}

Next, we show that $\tcetarv$ is weakly compact.

\begin{lem}\label{lemma:compactness_of_entropic_tc}
For any $\eta \geq 0$, $\tcetarv$ is weakly compact.
\end{lem}
\begin{proof}
Let $\{\pi_n\} \subset \pitc^\eta(\bbP, \bbQ)$ be a sequence such that $\pi_n \Rightarrow \pi \in \calM(\X^\bbN\times\Y^\bbN)$.
By Lemma \ref{lemma:compactness_of_tceta}, $\tcrv$ is weakly compact so $\pi \in \tcrv$.
Letting $R$ be the transition matrix of $\pi$, we need only show that $R \in \tceta$.
Letting $R_n$ be the transition matrix of $\pi_n$, it follows from \eqref{eq:convergence_of_trans_mat} that $R_n \rightarrow R$.
Using the weak lower semicontinuity of the KL-divergence, for every $s \in \X\times\Y$,
\begin{equation*}
\calK(R(s,\cdot) \| P \otimes Q(s,\cdot)) \leq \liminf\limits_{n\rightarrow\infty} \calK(R_n(s,\cdot) \|P\otimes Q(s,\cdot)) \leq \eta.
\end{equation*}
Therefore, $R \in \Pi_\eta(P,Q)$ and we find that $\pi \in \pitc^\eta(\bbP, \bbQ)$.
Thus, we conclude that $\pitc^\eta(\bbP, \bbQ)$ is weakly compact.
\end{proof}

\begin{prop}\label{prop:entropic_existence}
For any $\eta > 0$, the constrained OTC problem \eqref{eq:entropic_otc} has a solution.
\end{prop}
\begin{proof}
Let $\{\pi^n\} \subset \pitc^\eta(\bbP, \bbQ)$ be a sequence such that 
\begin{equation*}
\int \tilde{c} \, d\pi^n \rightarrow \inf\limits_{\pi \in \pitc^\eta(\bbP, \bbQ)} \int \tilde{c} \, d\pi.
\end{equation*}
By Lemma \ref{lemma:compactness_of_entropic_tc}, $\tcetarv$ is weakly compact.
So there exists a subsequence $\{\pi^{n_k}\}$ such that $\pi^{n_k} \Rightarrow \pi^*$ for some $\pi^* \in \pitc^\eta(\bbP, \bbQ)$.
Since $\tilde{c}$ is continuous and bounded,
\begin{equation*}
\int \tilde{c} \, d\pi^* = \lim\limits_{k\rightarrow\infty} \int \tilde{c} \, d\pi^{n_k} = \inf\limits_{\pi \in \pitc^\eta(\bbP, \bbQ)}\int \tilde{c} \, d\pi.
\end{equation*}
Thus $\pi^*$ is an optimal solution for Problem \eqref{eq:entropic_otc}.
\end{proof}

\subsection{Triangle Inequality}
Next we prove that the optimal transition coupling cost satisfies the triangle inequality when the cost does.
For probability measures $p_1$, $p_2$, $p_3 \in \calM(\X)$, we let $\Pi(p_1, p_2, p_3)$ denote the set of three-way couplings of $p_1$, $p_2$, and $p_3$ defined in the obvious way.
For stationary Markov process measures $\bbP_1$, $\bbP_2$, $\bbP_3 \in \calM(\X^\bbN)$ we let $\Pi_{\mbox{\tiny TC}}(\bbP_1, \bbP_2, \bbP_3)$ denote the set of three-way transition couplings of $\bbP_1$, $\bbP_2$, and $\bbP_3$, again defined in the obvious way.
If the three process measures have transition matrices $P_1$, $P_2$, and $P_3 \in \bbR^{d\times d}$, we let $\Pi(P_1, P_2, P_3)$ denote the set of three-way transition coupling matrices of $P_1$, $P_2$, and $P_3$.

\begin{lem}[Gluing Lemma]\label{lemma:gluing_lemma}
Let $\bbP_1$, $\bbP_2$, $\bbP_3 \in \calM(\X^\bbN)$ be stationary and irreducible Markov chains with stationary distributions $p_1$, $p_2$, $p_3 \in \calM(\X)$, and let $\pi_{12} \in \Pi_{\mbox{\tiny TC}}(\bbP_1, \bbP_2)$ and $\pi_{23} \in \Pi_{\mbox{\tiny TC}}(\bbP_2, \bbP_3)$.
Then there exists $\pi_{123} \in \Pi_{\mbox{\tiny TC}}(\bbP_1, \bbP_2, \bbP_3)$ such that $\pi_{123}(A_1 \times A_2 \times \X^\bbN) = \pi_{12}(A_1 \times A_2)$ and $\pi_{123}(\X^\bbN \times A_2 \times A_3) = \pi_{23}(A_2 \times A_3)$ for any $A_1$, $A_2$, $A_3 \subset \X^\bbN$.
Furthermore, any stationary distribution $\lambda_{123} \in \calM(\X\times\X\times\X)$ of $R_{123}$ necessarily satisfies $\lambda_{123} \in \Pi(p_1, p_2, p_3)$.
\begin{proof}
Let $\bbP_1$, $\bbP_2$, $\bbP_3$, $\pi_{12}$ and $\pi_{23}$ have transition matrices $P_1$, $P_2$, $P_3$, $R_{12}$ and $R_{23}$, respectively.
By the gluing lemma for optimal couplings \citep{villani2008optimal}, for every $x_1$, $x_2$, $x_3 \in \X$, there exists a coupling $r_{(x_1, x_2, x_3)} \in \Pi(P_1(x_1, \cdot), P_2(x_2, \cdot), P_3(x_3, \cdot))$ such that 
\begin{equation*}
\sum_{\tilde{x}_3} r_{(x_1, x_2, x_3)}(\tilde{x}_1, \tilde{x}_2, \tilde{x}_3) = R_{12}((x_1, x_2), (\tilde{x}_1, \tilde{x}_2))
\end{equation*}
and 
\begin{equation*}
\sum_{\tilde{x}_1} r_{(x_1, x_2, x_3)}(\tilde{x}_1, \tilde{x}_2, \tilde{x}_3) = R_{23}((x_2, x_3), (\tilde{x}_2, \tilde{x}_3)).
\end{equation*}
Let $R_{123} \in \bbR^{d^3 \times d^3}$ be the transition matrix such that for every $(x_1, x_2, x_3), (\tilde{x}_1, \tilde{x}_2, \tilde{x}_3) \in \X \times \X \times \X$, $R_{123}((x_1, x_2, x_3), (\tilde{x}_1, \tilde{x}_2, \tilde{x}_3)) = r_{(x_1, x_2, x_3)}(\tilde{x}_1, \tilde{x}_2, \tilde{x}_3)$.
By construction, $R_{123} \in \Pi(P_1, P_2, P_3)$ and we may let $\pi_{123}$ be the stationary Markov process measure constructed from $R_{123}$ and some stationary distribution $\lambda_{123} \in \calM(\X\times\X\times\X)$ of $R_{123}$.
To see that $\lambda_{123} \in \Pi(p_1, p_2, p_3)$, let the first $\X$-marginal of $\lambda_{123}$ be $\tilde{p}_1 \in \calM(\X)$.
Then for every $x_1 \in \X$, 
\begin{align*}
\tilde{p}_1(x_1) &= \sum\limits_{x_2, x_3} \lambda_{123}(x_1, x_2, x_3) \\
&= \sum\limits_{x_2, x_3} \sum\limits_{\tilde{x}_1, \tilde{x}_2, \tilde{x}_3} \lambda_{123}(\tilde{x}_1, \tilde{x}_2, \tilde{x}_3) \times R_{123}((\tilde{x}_1, \tilde{x}_2, \tilde{x}_3), (x_1, x_2, x_3)) \\
&= \sum\limits_{\tilde{x}_1, \tilde{x}_2, \tilde{x}_3} \lambda_{123}(\tilde{x}_1, \tilde{x}_2, \tilde{x}_3) P_1(\tilde{x}_1, x_1) \\
&= \sum\limits_{\tilde{x}_1} \tilde{p}_1(\tilde{x}_1) P_1(\tilde{x}_1, x_1),
\end{align*}
so $\tilde{p}_1$ is stationary with respect to $P_1$.
Since $\bbP_1$ is irreducible, the stationary distribution of $P_1$ is unique and it follows that $\tilde{p}_1 = p_1$.
Repeating the argument for the second and third marginals, it follows that $\lambda_{123} \in \Pi(p_1, p_2, p_3)$ and thus $\pi_{123} \in \Pi_{\mbox{\tiny TC}}(\bbP_1, \bbP_2, \bbP_3)$.
\end{proof}
\end{lem}

\begin{figure*}[t]
\begin{equation*}
R = \kbordermatrix{& (0, 0) & (0, 1)& (0,2)&(1,0)&(1,1)&(1,2)&(2,0)&(2,1)&(2,2) \\ (0, 0) & 0 & 0.25 & 0 & 0.25 & 0 & 0 & 0 & 0 & 0.50 \\ (0,1) & 0 & 0 & 0.25 & 0 & 0 & 0.25 & 0.25 & 0.25 & 0 \\ (0,2) & 0 & 0 & 0.25 & 0.25 & 0 & 0 & 0.25 & 0.25 & 0 \\ (1,0) & 0.25 & 0 & 0 & 0 & 0 & 0.25 & 0 & 0.25 & 0.25 \\ (1,1) & 0 & 0 & 0.25 & 0.25 & 0 & 0 & 0 & 0.25 & 0.25 \\ (1,2) & 0 & 0.25 & 0 & 0 & 0 & 0.25 & 0.50 & 0 & 0 \\ (2,0)& 0 & 0.25 & 0 & 0.25 & 0 & 0 & 0 & 0 & 0.50 \\ (2,1) & 0.25 & 0 & 0 & 0 & 0 & 0.25 & 0 & 0.25 & 0.25 \\ (2,2) & 0 & 0.25 & 0 & 0 & 0 & 0.25 & 0.50 & 0 & 0}.
\end{equation*}
\caption{A reducible transition coupling of irreducible transition matrices $P$ and $Q$ defined in \eqref{eq:trans_mat_1} and \eqref{eq:trans_mat_2}, respectively.}
\label{fig:reducible_trans_coup}
\end{figure*}

\begin{prop}[Triangle Inequality]\label{prop:triangle_ineq}
Let $\bbP_1$, $\bbP_2$, $\bbP_3 \in \calM(\X^\bbN)$ be stationary and irreducible Markov chains and let $\tilde{c}(\bfx, \tilde{\bfx}) = c(x_0, \tilde{x}_0)$ for every $\bfx$, $\tilde{\bfx} \in \X^\bbN$.
If $c$ satisfies the triangle inequality, then the OTC problem satisfies
\begin{equation}\label{eq:otc_triangle_ineq}
\begin{split}
\min\limits_{\pi \in \Pi_{\mbox{\tiny TC}}(\bbP_1, \bbP_3)} \int \tilde{c} \, d\pi \leq \min\limits_{\pi \in \Pi_{\mbox{\tiny TC}}(\bbP_1, \bbP_2)} \int \tilde{c} \, d\pi + \min\limits_{\pi \in \Pi_{\mbox{\tiny TC}}(\bbP_2, \bbP_3)} \int \tilde{c} \, d\pi.
\end{split}
\end{equation}
\end{prop}
\begin{proof}
By Proposition \ref{prop:existence}, there exist $\pi_{12} \in \Pi_{\mbox{\tiny TC}}(\bbP_1, \bbP_2)$ and $\pi_{23} \in \Pi_{\mbox{\tiny TC}}(\bbP_2, \bbP_3)$ that are optimal in the two problems on the right hand side of \eqref{eq:otc_triangle_ineq}.
Then by Lemma \ref{lemma:gluing_lemma}, there exists $\pi_{123} \in \Pi_{\mbox{\tiny TC}}(\bbP_1, \bbP_2, \bbP_3)$ that admits $\pi_{12}$ and $\pi_{23}$ as $(\X\times\X)^\bbN$-marginals.
Define the measure $\pi_{13} \in \calM((\X\times\X)^\bbN)$ by $\pi_{13}(A_1 \times A_3) = \pi_{123}(A_1 \times \X^\bbN \times A_3)$ for every $A_1$, $A_3 \subset \X^\bbN$.
Clearly, $\pi_{13} \in \Pi_{\mbox{\tiny TC}}(\bbP_1, \bbP_3)$.
Moreover, $\tilde{c}$ satisfies the triangle inequality on $(\X\times\X)^\bbN$ since $c$ satisfies it on $\X\times\X$.
Thus,
\begin{align*}
\min\limits_{\pi \in \Pi_{\mbox{\tiny TC}}(\bbP_1, \bbP_3)} \int \tilde{c} \, d\pi &\leq \int_{(\X\times\X)^\bbN} \tilde{c}(\bfx_1, \bfx_3) \, d\pi_{13}(\bfx_1, \bfx_3) \\
&= \int_{(\X\times\X\times\X)^\bbN} \tilde{c}(\bfx_1, \bfx_3) \, d\pi_{123}(\bfx_1, \bfx_2, \bfx_3) \\
&\leq \int_{(\X\times\X\times\X)^\bbN} (\tilde{c}(\bfx_1, \bfx_2) + \tilde{c}(\bfx_2, \bfx_3)) \, d\pi_{123}(\bfx_1, \bfx_2, \bfx_3) \\
&= \int_{(\X\times\X)^\bbN} \tilde{c}(\bfx_1, \bfx_2) \, d\pi_{12}(\bfx_1, \bfx_2) + \int_{(\X\times\X)^\bbN} \tilde{c}(\bfx_2, \bfx_3) \, d\pi_{23}(\bfx_2, \bfx_3) \\
&= \min\limits_{\pi \in \Pi_{\mbox{\tiny TC}}(\bbP_1, \bbP_2)} \int \tilde{c} \, d\pi + \min\limits_{\pi \in \Pi_{\mbox{\tiny TC}}(\bbP_2, \bbP_3)} \int \tilde{c} \, d\pi.
\end{align*}
\end{proof}

\section{Reducible Transition Coupling of Irreducible Chains}\label{sec:redicible_tc}
In this appendix, we provide an example showing that a transition coupling of two irreducible transition matrices is not necessarily irreducible.
Let 
\begin{equation}\label{eq:trans_mat_1}
P = \kbordermatrix{&0 & 1 & 2 \\ 0 & 0.25 & 0.25 & 0.50 \\ 1 & 0.25 & 0.25 & 0.50 \\ 2& 0.25 & 0.25 & 0.50}
\end{equation}
and
\begin{equation}\label{eq:trans_mat_2}
Q = \kbordermatrix{&0 & 1 & 2 \\ 0 & 0.25 & 0.25 & 0.50 \\ 1 & 0.25 & 0.25 & 0.50 \\ 2 & 0.50 & 0.25 & 0.25}.
\end{equation}
Both $P$ and $Q$ are clearly irreducible, but the transition coupling $R$, given in Figure \ref{fig:reducible_trans_coup}, is reducible.
%While it is clear that a transition coupling of irreducible Markov chains is not necessarily irreducible, one may wonder whether each transition coupling contains a single subchain.
%This is also false as demonstrated by the following example:
%Let the marginal Markov chains have the transition matrices:
%\begin{align*}
%P(i, \cdot) &= (0.2616396, 0.1303317, 0.1701102, 0.2729427, 0.1649758), \quad \forall i \in \mathcal{X} \\
%Q(j, \cdot) &= (0.26740126, 0.19563555, 0.21534975, 0.03325741, 0.28815604), \quad \forall j \in \mathcal{Y}
%\end{align*}
%Again, both $P$ and $Q$ are clearly irreducible.
%However, consider the transition coupling $R$ defined by
%\begin{align*}
%R((i, j), \cdot) &= (0.261639589, 0, 0, 0, 0, 0, 0, 0.09707431, 0.03325741, 0, 0, 0,  0, 0, 0.1701102, \\
%& 0, 0.1958355, 0.07710716, 0, 0, 0.005761669, 0, 0.04116828, 0, 0.1180458)
%\end{align*}
%for $(i, j) = (0, 0), (1, 2), (1, 3), (2, 4), (3, 1), (3,2), (0, 4), (4, 2), (4, 4)$ and
%\begin{align*}
%R((i, j), \cdot) &= (0, 0.01303243, 0.2153497, 0.03325741, 0, 0.09729105, 0.01782733, 0, 0, 0.01521333, \\
%& 0.17011020, 0, 0, 0, 0, 0, 0, 0, 0, 0.27294271, 0, 0.16497578, 0, 0, 0)
%\end{align*}
%for $(i, j) = (0,1), (0,2), (0,3), (1,0), (1,1), (1,4),(2,0),(3,4),(4,1)$.
%Then the transition coupling corresponds to two non-overlapping subchains.
While we do not provide an example here, we remark that transition coupling matrices of aperiodic and irreducible transition matrices may also have multiple recurrent classes.

\section{Comparison to 1-step Optimal Transition Coupling}
\label{app:one_step}
In this appendix, we demonstrate how the 1-step transition coupling problem described in Section \ref{sec:background_on_otc} prioritizes expected cost in the next step over long-term average cost as the OTC problem does.

\begin{ex}
\label{ex:greedyotc_comparison}
Consider stationary Markov chains $X$ and $Y$ with transition distributions defined by the graphs in Figure \ref{fig:greedy_otc_marginals}.
In order to find an OTC of $X$ and $Y$, we must specify a cost for every pair of states $(x, y) \in \X\times\Y$.
Let states $(0,0)$, $(1,2)$, $(2,1)$, $(2,2)$, and $(3,3)$ have cost 0, states $(1,1)$, $(3, 4)$ and $(4, 3)$ have cost 1, 
state $(4,4)$ have cost 9, and let all other states have a cost sufficiently large 1-step OTC and OTC do not assign 
them positive probability.

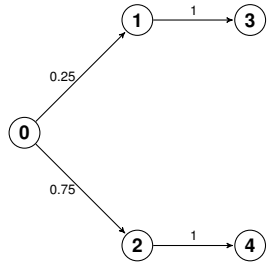
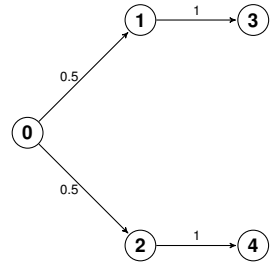
\begin{figure}[h]
\centering
\begin{subfigure}{0.45\linewidth}
\centering
\scalebox{0.5}{
\begin{tikzpicture}[->,>=stealth',shorten >=1pt,auto,node distance=3cm,
                    thick,main node/.style={circle,draw,font=\sffamily\Large\bfseries}]

  \node[main node] (0) {0};
  \node[main node] (1) [right of=0, above of=0]{1};
  \node[main node] (2) [right of=0, below of=0]{2};
  \node[main node] (3) [right of=1]{3};
  \node[main node] (4) [right of=2]{4};
 %  \node[main node] (5) [right of=3,below of=3,dashed]{0};

  \path[every node/.style={font=\sffamily\small}]
    (0) edge [right] node[above, left] {0.25} (1)
    (0) edge [right] node[below, left] {0.75} (2)
    (1) edge [right] node[above] {1} (3)
    (2) edge [right] node[above] {1} (4);
   % (3) edge [right, dashed] node[above, right] {1} (5)
   % (4) edge [right, dashed] node[above, right] {1} (5);
\end{tikzpicture}
}
\vspace{5mm}
\caption{$X$ transition probabilities}
\end{subfigure}
\begin{subfigure}{0.45\linewidth}
\centering
\scalebox{0.5}{
\begin{tikzpicture}[->,>=stealth',shorten >=1pt,auto,node distance=3cm,
                    thick,main node/.style={circle,draw,font=\sffamily\Large\bfseries}]

  \node[main node] (0) {0};
  \node[main node] (1) [right of=0, above of=0]{1};
  \node[main node] (2) [right of=0, below of=0]{2};
  \node[main node] (3) [right of=1]{3};
  \node[main node] (4) [right of=2]{4};
%  \node[main node] (5) [right of=3,below of=3,dashed]{0};

  \path[every node/.style={font=\sffamily\small}]
    (0) edge [right] node[above, left] {0.5} (1)
    (0) edge [right] node[below, left] {0.5} (2)
    (1) edge [right] node[above] {1} (3)
    (2) edge [right] node[above] {1} (4);
   % (3) edge [right, dashed] node[above, right] {1} (5)
   % (4) edge [right, dashed] node[above, right] {1} (5);
\end{tikzpicture}
}
\vspace{5mm}
\caption{$Y$ transition probabilities}
\end{subfigure}
\caption{Marginal stationary Markov chains.
Both chains return to state $0$ from states $3$ and $4$ with probability one.}
\label{fig:greedy_otc_marginals}
\end{figure}
\noindent
The transition distributions of the OTC and 1-step OTC are largely the same except for the transitions from $(0, 0)$ to  $(1, 1)$, $(1, 2)$, $(2, 1)$ and $(2,2)$ (see Figure \ref{fig:greedy_otc} for an illustration).
In particular, since the OTC chooses the transitions to minimize expected cost over the complete trajectory of the chain, it assigns lower probability to the transition $(0, 0) \rightarrow (2, 2)$ in order to avoid the costly state $(4, 4)$.
On the other hand, the 1-step OTC does not utilize this information in deciding how to transition from $(0, 0)$ and assigns a higher probability to the transition $(0, 0) \rightarrow (2, 2)$.
As a result, the expected cost of the 1-step OTC is $5/3$ compared to an expected cost of $1$ for the OTC.
In fact, by increasing the cost of the state $(4, 4)$, one can make the difference between the 1-step OTC and OTC costs arbitrarily large.
The lower expected cost indicates that the OTC constitutes a better alignment of $X$ and $Y$ as compared to the 1-step OTC.

\begin{figure}[h]
\centering
\begin{subfigure}{0.45\linewidth}
\centering
\scalebox{0.5}{
\begin{tikzpicture}[->,>=stealth',shorten >=1pt,auto,node distance=3cm,
                    thick,main node/.style={circle,draw,font=\sffamily\Large\bfseries}]

%  \node[main node] (00) {(0,0)};
%  \node[main node] (12) [right of=0]{(1,2)};
%  \node[main node] (21) [below of=12]{(2,1)};
%  \node[main node] (22) [below of=21]{(2,2)};
%  \node[main node] (11) [above of=12]{(1,1)};
%  \node[main node] (34) [right of=12]{(3,4)};
%  \node[main node] (43) [right of=21]{(4,3)};
%  \node[main node] (44) [right of=22]{(4,4)};
%  \node[main node] (33) [right of=11]{(3,3)};
%  \node[main node] (restart) [right of=34, dashed]{(0,0)};

  \node[main node] (00) {(0,0)};
  \node[main node] (12) [right of=0, above of=0]{(1,2)};
  \node[main node] (21) [right of=0]{(2,1)};
  \node[main node] (22) [right of=0, below of=0]{(2,2)};
  \node[main node] (11) [above of=12]{(1,1)};
  \node[main node] (34) [right of=12]{(3,4)};
  \node[main node] (43) [right of=21]{(4,3)};
  \node[main node] (44) [right of=22]{(4,4)};
  \node[main node] (33) [above of=34]{(3,3)};
 % \node[main node] (restart) [right of=43, dashed]{(0,0)};

  \path[every node/.style={font=\sffamily\small}]
	(00) edge [right] node[above, left] {0.25} (11)
	(00) edge [right] node[above] {0.25} (21)
	(00) edge [right] node[below, left] {0.5} (22)
	(11) edge [right] node[above] {1} (33)
	(12) edge [right] node[above] {1} (34)
	(21) edge [right] node[above] {1} (43)
	(22) edge [right] node[above] {1} (44);
	%(33) edge[right, dashed] node[above, right] {1} (restart)
	%(44) edge[right, dashed] node[above] {1} (restart)
	%(43) edge[right, dashed] node[above] {1} (restart)
	%(34) edge[right, dashed] node[below] {1} (restart);
\end{tikzpicture}

}
\vspace{5mm}
\caption{1-step OTC (expected cost of $5/3$)}
\label{fig:greedy_otc_1}
\end{subfigure}
\begin{subfigure}{0.45\linewidth}
\centering
\scalebox{0.5}{
\begin{tikzpicture}[->,>=stealth',shorten >=1pt,auto,node distance=3cm,
                    thick,main node/.style={circle,draw,font=\sffamily\Large\bfseries}]

  \node[main node] (00) {(0,0)};
  \node[main node] (12) [right of=0, above of=0]{(1,2)};
  \node[main node] (21) [right of=0]{(2,1)};
  \node[main node] (22) [right of=0, below of=0]{(2,2)};
  \node[main node] (11) [above of=12]{(1,1)};
  \node[main node] (34) [right of=12]{(3,4)};
  \node[main node] (43) [right of=21]{(4,3)};
  \node[main node] (44) [right of=22]{(4,4)};
  \node[main node] (33) [above of=34]{(3,3)};
 % \node[main node] (restart) [right of=43, dashed]{(0,0)};

  \path[every node/.style={font=\sffamily\small}]
	(00) edge [right] node[above, left] {0.25} (12)
	(00) edge [right] node[above] {0.5} (21)
	(00) edge [right] node[below, left] {0.25} (22)
	(11) edge [right] node[above] {1} (33)
	(12) edge [right] node[above] {1} (34)
	(21) edge [right] node[above] {1} (43)
	(22) edge [right] node[above] {1} (44);
	%(33) edge[right, dashed] node[above, right] {1} (restart)
	%(34) edge[right, dashed] node[above] {1} (restart)
	%(43) edge[right, dashed] node[above] {1} (restart)
	%(44) edge[right, dashed] node[below] {1} (restart);
\end{tikzpicture}
}
\vspace{5mm}
\caption{OTC (expected cost of $1$)}
\label{fig:greedy_otc_2}
\end{subfigure}
\caption{An example where the 1-step OTC has sub-optimal expected cost. Both chains return to state $(0, 0)$ from states $(3, 3)$, $(3, 4)$, $(4, 3)$, and $(4, 4)$ with probability one.
Note that Figures \ref{fig:greedy_otc_1} and \ref{fig:greedy_otc_2} omit the edges that are the same between the two transition couplings.}
\label{fig:greedy_otc}
\end{figure}
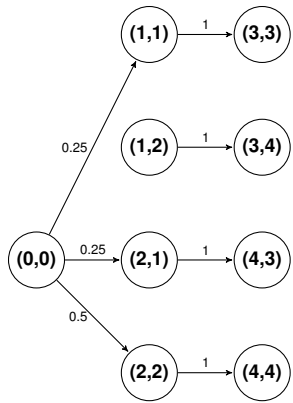
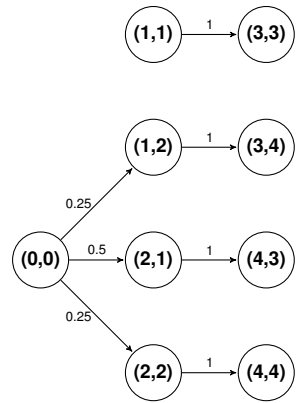
\end{ex}

\clearpage
%\end{appendices}

\end{document}